\documentclass[a4paper,11pt,leqno]{amsart}
\usepackage{blocks_style}

\title[Dualities from $\Phi$-Cuspidal Pairs and Affine Springer Fibers]{Level-Rank Dualities from $\Phi$-Cuspidal Pairs and Affine Springer Fibers}

\author[Minh-T\^{a}m Trinh]{Minh-T\^{a}m Quang Trinh}
\address{Department of Mathematics, Yale University, New Haven, CT 06520}
\email{minh-tam.trinh@yale.edu}

\author{Ting Xue}
\address{School of Mathematics and Statistics, University of Melbourne, VIC 3010, Australia}
\email{ting.xue@unimelb.edu.au}

\begin{document}

\begin{abstract}
We propose a generalization of the level-rank dualities arising from Uglov's work on higher-level Fock spaces.
The statements use Hecke algebras defined by Brou\'{e}--Malle, which conjecturally describe the endomorphisms of Lusztig induction modules, and a generalization of Harish-Chandra theory due to Brou\'{e}--Malle--Michel.
For any generic finite reductive group $\mathbb{G}$ and integers $e, m > 0$, we conjecture that:
(1) the intersection of a $\Phi_e$-Harish-Chandra series and a $\Phi_m$-Harish-Chandra series is parametrized by a union of blocks of the Hecke algebra of the $\Phi_e$-cuspidal pair at an $m$th root of unity, and similarly for the Hecke algebra of the $\Phi_m$-cuspidal pair at an $e$th root of unity;
(2) these parametrizations match the blocks on the two sides;
(3) when two blocks match, the bijection between them lifts to a derived equivalence between associated blocks of rational DAHAs.
Surprisingly, these structures also appear in bimodules formed from the cohomology of affine Springer fibers studied by Oblomkov--Yun.
When $\mathbb{G} = \mathbb{GL}_n$ and $e, m$ are coprime, we show that (1)--(3) hold, and that (3) recovers the level-rank dualities conjectured by Chuang--Miyachi and later proved through the work of several other people.
Finally, we verify for many cases where $\mathbb{G}$ is exceptional that Brou\'e--Malle's parameters are numerically compatible with our conjectures.
\end{abstract}

\maketitle

\thispagestyle{empty}



\section{Introduction}

\subsection{}

\dfemph{Double affine Hecke algebras}, or \dfemph{DAHAs}, were introduced by Cherednik to prove Macdonald's conjectures on orthogonal polynomials~\cite{cherednik_92, cherednik_95}.
Their rational degenerations, introduced in~\cite{drinfeld, eg}, are also known as \dfemph{rational Cherednik algebras}.
A rational DAHA has an analogue of the Bernstein--Gelfand--Gelfand category O of a semisimple Lie algebra.
Among other applications, this is a highest-weight category that controls the decomposition numbers of a Hecke algebra associated with the underlying reflection group~\cite{ggor}.

In~\cite{springer_76}, Springer constructed representations of finite Weyl groups using the cohomology of fixed-point varieties in flag manifolds, now known as Springer fibers.
In~\cite{lusztig_96}, Lusztig extended this work to affine Weyl groups, using ind-varieties now known as \dfemph{affine Springer fibers}.
In~\cite{oy}, motivated by $K$-theoretic work in~\cite{vv}, Oblomkov--Yun developed a double-affine analogue, involving actions of trigonometric and rational DAHAs on the modified cohomology of \dfemph{homogeneous} affine Springer fibers.
In this setting, the Springer actions intertwine with monodromy actions of certain braid groups, associated with smaller complex reflection groups in the manner of~\cite{bmr}.
This paper grew from our attempts to find formulas for the resulting (DAHA, braid-group) bimodules.

In~\cite{trinh_21}, motivated by a putative Betti avatar of Oblomkov--Yun's setup, the  first named author conjectured a formula for the virtual graded character of their rational DAHA module in the split case, taking the form
\begin{align}\label{eq:trinh}
\sum_\chi {\Deg_\chi(e^{2\pi i\nu})}[\Delta_\nu(\chi)].
\end{align}
Above, $\nu \in \bb{Q}_{> 0}$ is the constant parameter of the rational DAHA.
The sum runs over the irreducible characters of the Weyl group $W$.
The expression $[\Delta_\nu(\chi)]$ is the graded character of the standard module of the rational DAHA indexed by $\chi$.
Its coefficient is an \emph{integer} value of the \dfemph{generic-degree polynomial} $\Deg_\chi(x)$ defined as follows:
For any split finite reductive group with Weyl group $W$ over a finite field $\bb{F}_q$, the value $\Deg_\chi(q)$ is the degree of its unipotent principal series character indexed by $\chi$.
In this way, \eqref{eq:trinh} suggests a connection between the rational DAHA modules in~\cite{oy} and representations of finite groups of Lie type.

In~\cite{gvx, vx}, in the context of their ongoing study of character sheaves for graded Lie algebras,\footnote{Also known as anti-orbital complexes.} Grinberg, Vilonen, and the second named author constructed a local system over the same base space as the affine Springer fibration in~\cite{oy}.
They computed that its monodromy, a priori a braid-group action, factors through a Hecke algebra for the underlying complex reflection group.
We expect that the local system in Oblomkov--Yun's work is essentially Grinberg--Vilonen--Xue's, via a double application of Fourier duality passing through work of Lusztig--Yun~\cite{ly_17, ly_18}.
Thus the monodromy representations in~\cite{oy} would factor through the same Hecke algebras.

\subsection{}

These starting points have led us to a more general framework, phrased most clearly in the formalism of \dfemph{generic finite reductive groups}~\cite{bm_92, bmm}.
We carefully review this formalism in \Crefrange{sec:hc}{sec:hecke}.

A generic finite reductive group $\bbb{G}$ consists of a root datum $\Gamma_\bbb{G}$ and a finite-order automorphism of $\Gamma_\bbb{G}$, defined up to composition with the Weyl group.
For each prime power $q > 1$, the root datum defines a connected reductive group $G$ over $\bar{\bb{F}}_q$, and the automorphism defines a $q$-Frobenius $F : G \to G$, hence a finite reductive group $\bbb{G}(q) \vcentcolon= G^F$.
The work~\cite{bmm} concerns the parts of the representation theory of $\bbb{G}(q)$ where $q$ can be treated as an indeterminate, or which do not depend on $q$ at all.
In particular, the irreducible characters of $\bbb{G}(q)$ which are \dfemph{unipotent} in the sense of Deligne--Lusztig theory can be indexed by a set $\Uch(\bbb{G})$ depending only on $\bbb{G}$.

For each integer $e > 0$, Brou\'e--Malle--Michel define the notion of a $\Phi_e$-cuspidal pair of $\bbb{G}$ and the associated $\Phi_e$-Harish-Chandra series of $\Uch(\bbb{G})$.
The Weyl group of $\bbb{G}$ acts on the set of $\Phi_e$-cuspidal pairs.
The fundamental theorem of~\cite[\S3.2]{bmm} states that as we run over representatives for the orbits, their $\Phi_e$-Harish-Chandra series are pairwise disjoint and partition $\Uch(\bbb{G})$.
When $e = 1$, this partition is that found in Harish-Chandra's philosophy of cusp forms.

In more detail:
We say that a generic Levi subgroup $\bbb{L}$ of $\bbb{G}$ is \dfemph{$\Phi_e$-split} if and only if it is the centralizer of a \dfemph{$\Phi_e$-torus}: the generic version of a torus whose order is a power of $\Phi_e(q)$, where $\Phi_e$ is the $e$th cyclotomic polynomial.
For such $\bbb{L}$, we say that $\lambda \in \Uch(\bbb{L})$ is \dfemph{$\Phi_e$-cuspidal}, and that $(\bbb{L}, \lambda)$ is a \dfemph{$\Phi_e$-cuspidal pair}, if and only if the Lusztig restriction of $\lambda$ to any strictly smaller $\Phi_e$-split Levi subgroup $\bbb{M}$ is zero.
The \dfemph{$\Phi_e$-Harish-Chandra series} $\Uch(\bbb{G})_{\bbb{L}, \lambda}$ is the set of all irreducible constituents of the Lusztig induction of $\lambda$ from $\bbb{L}$ to $\bbb{G}$.

Each $\Phi_e$-cuspidal pair $(\bbb{L}, \lambda)$ gives rise to a complex reflection group called its \dfemph{relative Weyl group} $W^\bbb{G}_{\bbb{L}, \lambda}$.
In~\cite{bm_93}, Brou\'e--Malle define a ring $H^\bbb{G}_{\bbb{L}, \lambda}(x)$ with an indeterminate $x$, and specializations of the form $H^\bbb{G}_{\bbb{L}, \lambda}(\alpha) = H^\bbb{G}_{\bbb{L}, \lambda}(x)|_{x \to \alpha}$, such that if $\zeta_e$ is a primitive $e$th root of unity, then $H^\bbb{G}_{\bbb{L}, \lambda}(\zeta_e)$ recovers the group algebra of $W^\bbb{G}_{\bbb{L}, \lambda}$.
They conjecture that if $q > 1$ is a prime power, then $H^\bbb{G}_{\bbb{L}, \lambda}(q)$ recovers the $\bbb{G}(q)$-equivariant endomorphism algebra of the Lusztig induction module arising from $(\bbb{L}, \lambda)$.
See \Cref{conj:end} for a precise statement.

This conjecture, together with an isomorphism between $H^\bbb{G}_{\bbb{L}, \lambda}(q)$ and the group algebra of $W^\bbb{G}_{\bbb{L}, \lambda}$, would imply a bijection
\begin{align}\label{eq:chi-intro}
\chi_{\bbb{L}, \lambda}^\bbb{G} : \Uch(\bbb{G})_{\bbb{L}, \lambda} \xrightarrow{\sim} \Irr(W^\bbb{G}_{\bbb{L}, \lambda})
\end{align}
compatible with induction from smaller $\Phi_e$-split Levis.
Even though the conjecture has only been shown in certain cases, Brou\'e--Malle--Michel establish the putative bijections $\chi_{\bbb{L}, \lambda}^\bbb{G}$ for all $\bbb{G}$ and $(\bbb{L}, \lambda)$, via case-by-case arguments~\cite{bmm}.

\subsection{}

Suppose that $\bbb{A}$ is a  maximal torus of $\bbb{G}$ that is $\Phi_1$-split as a Levi, and that the automorphisms defining $\bbb{G}$ and $\bbb{A}$ are induced by the same Dynkin-diagram automorphism.
In this case, the BGG category O of the rational DAHA in~\cite{oy}, depending on a \dfemph{slope} $\nu \in \bb{Q}_{> 0}$, forms a highest-weight cover of the module category of $H^\bbb{G}_{\bbb{A}, 1}(e^{2\pi i\nu})$.

On the other hand, if $m$ is the denominator of $\nu$ in lowest terms and $\bbb{T}$ is a maximal torus that is $\Phi_m$-split as a Levi of $\bbb{G}$, then $W^\bbb{G}_{\bbb{T}, 1}$ is the complex reflection group whose braid group appears in~\cite{oy}.
Recall that we expect the braid action in \emph{ibid.}\ to factor through the Hecke algebra from~\cite{vx}.
It turns out that the latter algebra is precisely $H^\bbb{G}_{\bbb{T}, 1}(1)$.


These observations suggest that the (DAHA, braid-group) bimodules in~\cite{oy} manifest some sort of duality between $H^\bbb{G}_{\bbb{A}, 1}(e^{2\pi i\nu})$ and $H^\bbb{G}_{\bbb{T}, 1}(1)$.
More generally, one might expect that for any $e, m > 0$, any $\Phi_e$-cuspidal pair $(\bbb{L}, \lambda)$, and any $\Phi_m$-cuspidal pair $(\bbb{M}, \mu)$, there are dualities that match certain categories of $H^\bbb{G}_{\bbb{L}, \lambda}(\zeta_m)$-modules and $H^\bbb{G}_{\bbb{M}, \mu}(\zeta_e)$-modules.
We have found evidence that this is so.

As explained in~\cite{gp, gj}, there is a \dfemph{decomposition map} going essentially from characters of $W^\bbb{G}_{\bbb{L}, \lambda}$ to isomorphism classes of $H^\bbb{G}_{\bbb{L}, \lambda}(\zeta_m)$-modules.
Via this map, $\Irr(W^\bbb{G}_{\bbb{L}, \lambda})$ is partitioned into subsets $\Irr(W^\bbb{G}_{\bbb{L}, \lambda})_\sf{b}$ indexed by blocks $\sf{b}$ of the category of $H^\bbb{G}_{\bbb{L}, \lambda}(\zeta_m)$-modules.
We will refer to these subsets as \dfemph{$H^\bbb{G}_{\bbb{L}, \lambda}(\zeta_m)$-blocks}.
Let $\Irr(W^\bbb{G}_{\bbb{L}, \lambda})_{\bbb{M}, \mu}$ and $\Irr(W^\bbb{G}_{\bbb{M}, \mu})_{\bbb{L}, \lambda}$ respectively denote the images of the left and right maps below:
\begin{align}\label{eq:chi-double-intro}
\Irr(W^\bbb{G}_{\bbb{L}, \lambda})
	\xleftarrow{\chi_{\bbb{L}, \lambda}^\bbb{G}}
	\Uch(\bbb{G})_{\bbb{L}, \lambda} \cap \Uch(\bbb{G})_{\bbb{M}, \mu}
	\xrightarrow{\chi_{\bbb{M}, \mu}^\bbb{G}}
		\Irr(W^\bbb{G}_{\bbb{M}, \mu}).
\end{align}
The statements below comprise \Cref{conj:bij,conj:cat} in the body text.

\begin{mainconj}\label[mainconj]{mainconj:1}
For any $\bbb{G}, e, m, (\bbb{L}, \lambda), (\bbb{M}, \mu)$ as above:
\begin{enumerate}
\item[(I)]
		$\Irr(W^\bbb{G}_{\bbb{L}, \lambda})_{\bbb{M}, \mu}$ and $\Irr(W^\bbb{G}_{\bbb{M}, \mu})_{\bbb{L}, \lambda}$ are, respectively, unions of $H^\bbb{G}_{\bbb{L}, \lambda}(\zeta_m)$-blocks and $H^\bbb{G}_{\bbb{M}, \mu}(\zeta_e)$-blocks.

\item[(II)]
		The maps in \eqref{eq:chi-double-intro} induce a bijection 
		\begin{align}
		\{\sf{b} \mid \Irr(W^\bbb{G}_{\bbb{L}, \lambda})_\sf{b} \subseteq \Irr(W^\bbb{G}_{\bbb{L}, \lambda})_{\bbb{M}, \mu}\}
			\xleftrightarrow{\sim}
			\{\sf{c} \mid \Irr(W^\bbb{G}_{\bbb{M}, \mu})_\sf{c} \subseteq \Irr(W^\bbb{G}_{\bbb{M}, \mu})_{\bbb{L}, \lambda}\}.
			\end{align}

\item[(III)]
		When $\sf{b}$ and $\sf{c}$ correspond to each other in (II), the resulting bijection 
		\begin{align}
		\Irr(W^\bbb{G}_{\bbb{L}, \lambda})_\sf{b} \xleftrightarrow{\sim} \Irr(W^\bbb{G}_{\bbb{M}, \mu})_\sf{c}
		\end{align}
		arises from an equivalence between the bounded derived categories of highest-weight covers of $\sf{b}$ and $\sf{c}$.

\end{enumerate}
\end{mainconj}

\subsection{}

In \Cref{sec:asf}, we review the precise setup of Oblomkov--Yun.
Then we present evidence toward a conjectural formula for their (DAHA, braid-group) bimodules, or rather, its class in a certain Grothendieck group:
See \Cref{prop:evidence}.

For now, we merely sketch their geometry.
Suppose that the root datum $\Gamma_\bbb{G}$ is irreducible, that $\bbb{G}$ arises from a Dynkin-diagram automorphism, and that $\bbb{A}$ is a maximal torus of $\bbb{G}$ that is $\Phi_1$-split  as a Levi.
Let $\fr{g}$ be the Lie algebra over $\bb{C}[\![z]\!]$ corresponding to $\bbb{G}$, and let $L\fr{g}$ be the formal loop space for which $L\fr{g}(\bb{C}) = \fr{g}(\bb{C}(\!(z)\!))$ with $z$ an indeterminate.

For any $\nu \in \bb{Q}$, Oblomkov--Yun construct a finite-dimensional subvariety $L\fr{g}_\nu^\rs \subseteq L\fr{g}$, stable under a $\bb{G}_m$-action on $L\fr{g}$ that depends on $\nu$ and a choice of simple roots.
The superscript $\rs$ is meant to connote regular semisimple elements.
We write $\Br_{\nu, \gamma}$ for the fundamental group of $L\fr{g}_\nu^\rs$ with basepoint $\gamma$.
Recall that if $\bbb{T}$ is a maximal torus of $\bbb{G}$ that is $\Phi_m$-split as a Levi, where $m$ is the denominator of $\nu$, then $\Br_{\nu, \gamma}$ is the braid group of $W^\bbb{G}_{\bbb{T}, 1}$.

Let $D_{\bbb{G}, \bbb{A}}^\rat(\vec{\nu})$ be the rational DAHA of $(\bbb{G}, \bbb{A})$ with parameter $\vec{\nu}$ depending on $\nu$ as in~\cite[\S{4.2}]{oy}.
(Note that some of our variable names do not match theirs.)
When $\bbb{G}$ is split, $\nu > 0$, and $m$ is a \dfemph{regular elliptic number} for the group $W^\bbb{G}_\bbb{A}$~\cite[\S{1.1}]{vv}, Oblomkov--Yun construct a local system of bigraded vector spaces over $L\fr{g}_\nu^\rs$, which we will denote by $\cal{E}_\nu$, from the $\bb{G}_m$-equivariant cohomology of a certain affine Springer fibration, together with its \dfemph{perverse filtration}.
They construct a fiberwise $D_{\bbb{G}, \bbb{A}}^\rat(\vec{\nu})$-action on $\cal{E}_\nu$, commuting with the action of $\Br_{\nu, \gamma}$ on $\cal{E}_{\nu, \gamma}$ by monodromy.
The $\Br_{\nu, \gamma}$-invariants of $\cal{E}_{\nu, \gamma}$ form the simple $D_{\bbb{G}, \bbb{A}}^\rat(\vec{\nu})$-module often denoted $L_{\vec{\nu}}(1)$.
We write $[\cal{E}_{\nu, \gamma}]$ for the virtual graded bimodule formed from $\cal{E}_{\nu, \gamma}$ by taking the alternating sum over cohomological degrees:
See \eqref{eq:bimod-1}.

To relate our next conjecture to \eqref{eq:trinh}, recall from~\cite{bmm} that for any generic character $\rho \in \Uch(\bbb{G})_{\bbb{T}, 1}$ and primitive $m$th root of unity $\zeta$, the generic degree $\Deg_\rho(x) \in \bb{Q}[x]$ satisfies 
\begin{align}\label{eq:deg-intro}
\Deg_\rho(\zeta) = \varepsilon_{\bbb{T}, 1}^\bbb{G}(\rho) \deg \chi_{\bbb{T}, 1}^\bbb{G}(\rho)
	\quad\text{for some sign $\varepsilon_{\bbb{T}, 1}^\bbb{G}(\rho) \in \{\pm 1\}$}.
\end{align}
We also write $\chi^\bbb{G}_{\bbb{T}, 1, 1}$ for the character of $H^\bbb{G}_{\bbb{T}, 1}(1)$ corresponding to a character $\chi$ of $W^\bbb{G}_{\bbb{T}, 1}$ under the decomposition map:
See \S\ref{subsec:decomp}--\ref{subsec:bij}.

\begin{mainconj}\label[mainconj]{mainconj:2}
Let $\bbb{G}, \bbb{A}, \nu, m, \bbb{T}$ be as above, with $\nu > 0$ and $m$ a regular elliptic number for $W^\bbb{G}_\bbb{A}$, and let $\gamma \in L\fr{g}_\nu^\rs(\bb{C})$.
Assume that either $\bbb{G}$ is split or~\cite[Conj.\ 8.2.5]{oy} holds. 
Then:
\begin{enumerate}
\item 	The $\Br_{\nu, \gamma}$-action on $\cal{E}_{\nu, \gamma}$ factors through $H^\bbb{G}_{\bbb{T}, 1}(1)$.

\item 	In the Grothendieck group $\sf{K}_0(\bb{C}W^\bbb{G}_\bbb{A} \otimes H^\bbb{G}_{\bbb{T}, 1}(1)^\op)[\![t]\!][t^{-1}]$, we have
\begin{align}
[\cal{E}_{\nu, \gamma}] = \sum_{\rho \in \Uch(\bbb{G})_{\bbb{A}, 1} \cap \Uch(\bbb{G})_{\bbb{T}, 1}}
	\varepsilon_{\bbb{T}, 1}^\bbb{G}(\rho)
	[\Delta_{\vec{\nu}}(\chi_{\bbb{A}, 1}^\bbb{G}(\rho)) \otimes \chi_{\bbb{T}, 1, 1}^\bbb{G}(\rho)],
\end{align}
		where $\Delta_{\vec{\nu}}(\chi)$ denotes the standard module of $D_{\bbb{G}, \bbb{A}}^\rat(\vec{\nu})$ indexed by $\chi$, and the variable $t$ tracks its $W^\bbb{G}_\bbb{A}$-equivariant Euler grading.
		
\end{enumerate}
\end{mainconj}

In the split case, \Cref{mainconj:2} refines \eqref{eq:trinh}, via \eqref{eq:deg-intro}.
The refined formula bears a remarkable analogy with a virtual bimodule that, under the Brou\'e--Malle--Michel conjectures, can be constructed from Deligne--Lusztig representations:
\begin{align}\label{eq:analogy}
\begin{array}{lll}
\textsf{bimodule}
	&\textsf{algebras}
	&\textsf{parameters}\\
	\hline
\cal{E}_{\nu, \gamma}
	&(D_{\bbb{G}, \bbb{A}}^\rat(\vec{\nu}), H^\bbb{G}_{\bbb{T}, 1}(1))
	&(e^{2\pi i\nu}, 1)\\
R_L^G(\lambda_q) \otimes_{\QL G^F} R_M^G(\mu_q)
	&(H^\bbb{G}_{\bbb{L}, \lambda}(q), H^\bbb{G}_{\bbb{M}, \mu}(q))
	&(q, q)
\end{array}
\end{align}
Above, $q > 1$ is a prime power; $G, L, \lambda_q, M, \mu_q$ are the finite-group data arising from $q, \bbb{G}, \bbb{L}, \lambda, \bbb{M}, \mu$; and $R_L^G(\lambda_q), R_M^G(\mu_q)$ are the compactly-supported cohomologies of appropriate Deligne--Lusztig varieties.
The top right entry alludes to how $D_{\bbb{G}, \bbb{A}}^\rat(\vec{\nu})$ provides a highest-weight cover of $H^\bbb{G}_{\bbb{A}, 1}(e^{2\pi i\nu})$.
\Cref{sec:asf} gives more details.

\subsection{}

Most work on the Brou\'e--Malle conjecture relating $H^\bbb{G}_{\bbb{L}, \lambda}(x)$ to Lusztig induction has focused on the case where $\bbb{L}$ is a maximal torus of $\bbb{G}$.
Beyond tori, the work of Dudas in~\cite{dudas} implicitly confirms the prediction of~\cite[\S{2.10}]{bm_93} when $\bbb{G}$ is a generic general linear group and $W^\bbb{G}_{\bbb{L}, \lambda}$ is cyclic.


In \Cref{sec:gl-gu}, we describe the Hecke algebras of~\cite{bm_93} explicitly in the general linear cases $\bbb{G} = \GGGL_n$ and the general unitary cases $\bbb{G} = \GGGU_n$.
Here, the groups $W^\bbb{G}_{\bbb{L}, \lambda}$ are wreath products of the form $\bb{Z}_e \wr S_a \vcentcolon= (\bb{Z}/e\bb{Z})^a \rtimes S_a$.
Thus the Hecke algebras $H^\bbb{G}_{\bbb{L}, \lambda}(x)$  are specializations of Ariki--Koike algebras, whose blocks at roots of unity were described combinatorially by Lyle--Mathas~\cite{lm}.
At the conclusion of \Cref{sec:gl-gu}, we use their work to prove:

\begin{mainthm}\label[mainthm]{mainthm:1}
In the setup of \Cref{mainconj:1}, suppose that $\bbb{G} = \GGGL_n$ or $\GGGU_n$ and that $e, m$ are coprime.
Then (I)--(II) hold.
In fact, $\Irr(W^\bbb{G}_{\bbb{L}, \lambda})_{\bbb{M}, \mu}$ and $\Irr(W^\bbb{G}_{\bbb{M}, \mu})_{\bbb{L}, \lambda}$ correspond to a single $H^\bbb{G}_{\bbb{L}, \lambda}(\zeta_m)$-block and a single $H^\bbb{G}_{\bbb{M}, \mu}(\zeta_e)$-block, respectively.
\end{mainthm}

For the general linear cases, we will further relate the bijections in parts (II)--(III) of \Cref{mainconj:1} to bijections that previously appeared in Uglov's work on \dfemph{higher-level Fock spaces}~\cite{uglov}.

Recall that for each tuple $\vec{s} \in \bb{Z}^e$, the $v$-deformed Fock space of level $e$ and charge $\vec{s}$ is the vector space $\Lambda_v^{\vec{s}}$ over $\bb{Q}(v)$ spanned by symbols $|\vec{\lambda}, \vec{s}\rangle$, where $\vec{\lambda}$ runs over $e$-tuples of integer partitions, or \dfemph{$e$-partitions}.
For each integer $m > 0$, it may be viewed as a module over the quantum affine algebra $\ur{U}_v(\widehat{\fr{sl}}_m)$, in which case the residue of $\vec{s}$ modulo $m$ describes the highest weight of the simple submodule generated by $|\vec{0}, \vec{s}\rangle$.
Generalizing the work of Leclerc--Thibon on the level-$1$ case, Uglov constructed a canonical basis for $\Lambda_v^{\vec{s}}$, related to the standard basis by an upper-triangular transition matrix of affine Kazhdan--Lusztig polynomials.
To do so, he made use of vector-space isomorphisms
\begin{align}
\bigoplus_{\substack{\vec{s} \in \bb{Z}^e \\ s_1 + \cdots + s_e = s}}
	\Lambda_v^{\vec{s}}
	\xleftarrow{\sim} \Lambda_v^s
	\xrightarrow{\sim} \bigoplus_{\substack{\vec{r} \in \bb{Z}^m \\ r_1 + \cdots + r_m = s}} \Lambda_v^{\vec{r}}
\end{align}
for each integer $s$, relating level-$e$ and level-$m$ Fock spaces to the level-$1$ Fock space of charge $s$.
Surprisingly, these isomorphisms are not defined symmetrically.
Together, they define an interesting bijection between charged $e$-partitions and charged $m$-partitions.
In~\cite{cm_12}, this bijection is described as a \dfemph{level-rank duality} in the sense of Frenkel~\cite{frenkel}.

As already observed in~\cite{uglov}, the left-hand map is essentially induced by the map sending an integer partition to its $e$-core and $e$-quotient.
The right-hand map is a modification of the analogous map for $m$.
It turns out that for $\bbb{G} = \GGGL_n$, where $\Uch(\bbb{G})$ is indexed by partitions of $n$, the $\Phi_e$-Harish-Chandra series of a unipotent irreducible character is indexed by the $e$-core of its partition, while its image under \eqref{eq:chi-intro} is indexed by a shifted version of its $e$-quotient.
This similarity of structure motivates the following result, proved in \Cref{sec:fock}:

\begin{mainthm}\label[mainthm]{mainthm:2}
Let $\Pi$ be the set of all integer partitions.
In the setup of \Cref{mainthm:1} for $\bbb{G} = \GGGL_n$, the maps of \eqref{eq:chi-double-intro} fit into a commutative diagram
\begin{equation}
\begin{tikzpicture}[baseline=(current bounding box.center), >=stealth]
\matrix(m)[matrix of math nodes, row sep=3em, column sep=3em, text height=2ex, text depth=0.5ex]
{ 		
	\Irr(W^\bbb{G}_{\bbb{L}, \lambda})
		&\Uch(\bbb{G})_{\bbb{L}, \lambda} \cap \Uch(\bbb{G})_{\bbb{M}, \mu}
		&\Irr(W^\bbb{G}_{\bbb{M}, \mu})\\
	\Pi^e \times \bb{Z}^e
		&\Pi
		&\Pi^m \times \bb{Z}^m\\	
	\Pi^e \times \bb{Z}^e
		&
		&\Pi^m \times \bb{Z}^m\\
};
\path[->, font=\scriptsize, auto]
(m-1-2)	
	edge node[above]{$\chi_{\bbb{L}, \lambda}^\bbb{G}$} (m-1-1)
	edge node{$\chi_{\bbb{M}, \mu}^\bbb{G}$} (m-1-3)
	edge (m-2-2)
(m-1-1)
	edge (m-2-1)
(m-1-3)
	edge (m-2-3)
(m-2-2)
	edge node[above]{$\Upsilon_e^1(|\rho,  \ell_\lambda\rangle) \mapsfrom \rho$} (m-2-1)
	edge node{$\rho \mapsto \Upsilon_m^1(|\rho,  \ell_\mu\rangle)$} (m-2-3)
(m-2-1)
	edge node[left]{$\tilde{w}_{e, m, s}$} (m-3-1)
(m-2-3)
	edge node{$\tilde{w}_{m, e, t}$} (m-3-3)
(m-3-1)
	edge node{$\Upsilon_m^e$} (m-3-3);
\end{tikzpicture}
\end{equation}
where the integers $\ell_\lambda, \ell_\mu$ are the respective lengths of $\lambda, \mu$; the maps $\tilde{w}_{e, m, s}, \tilde{w}_{m, e, t}$ arise from affine permutations via \eqref{eq:affine-perm}; the integers $s, t$ need only satisfy $\ell_\lambda - s = \ell_\mu - t$; and the map $\Upsilon_m^e$ is essentially Uglov's bijection in~\cite[\S{4}]{uglov}, as explained in \S\ref{subsec:uglov}.
All maps in the diagram are injective or bijective.

Furthermore, the diagram is compatible with Lusztig induction between $\Phi_e$- and $\Phi_m$-split Levis, in a precise sense corresponding to \Cref{thm:bmm}(2a).
\end{mainthm}

Chuang--Miyachi conjectured that Uglov's bijections could be categorified by Koszul-type equivalences\footnote{More precisely, compositions of commuting Koszul and Ringel equivalences: See~\cite[\S{3.5}]{svv}.\label{foot:koszul-ringel}} between blocks of highest-weight covers for Ariki--Koike algebras, \emph{i.e.}, blocks of categories O for cyclotomic rational DAHAs~\cite{cm_12}.
This categorical level-rank duality was proved by Rouquier--Shan--Varagnolo--Vasserot in~\cite{rsvv, svv}, relying on equivalences between such categories O and truncations of the parabolic categories O of the affine Lie algebras $\widehat{\fr{sl}}_e$.
The latter equivalences were proved by Losev~\cite{losev} and Rouquier--Shan--Varagnolo--Vasserot~\cite{rsvv} independently.
Using these results, we deduce:

\begin{maincor}\label{maincor:1}
In the setup of \Cref{mainconj:1}, suppose that $\bbb{G} = \GGGL_n$ and that $e, m$ are coprime.
Then (III) holds.
\end{maincor}

\subsection{}

In \Cref{sec:exceptional}, we give evidence for \Cref{mainconj:1} in exceptional types. 
	More precisely, 
	we verify the partitioning of sets in (I), and the agreement of cardinalities implied by (II), though not the bijections:
	See \Cref{prop:exceptional}.
	Note that by \Cref{rem:singular}, it suffices to check pairs of \dfemph{singular numbers} $e, m$.
	For (split) $\bbb{G}$ of type $G_2$ or $F_4$, we give the explicit details for all such pairs.

\subsection{Future Work}

We will address the following generalizations in a sequel:
\begin{itemize}
\item 	We expect to lift the hypothesis about $e, m$ being coprime from \Cref{mainthm:1}.
		In this generality, $\Irr(W^\bbb{G}_{\bbb{L}, \lambda})_{\bbb{M}, \mu}$ and $\Irr(W^\bbb{G}_{\bbb{M}, \mu})_{\bbb{L}, \lambda}$ can each correspond to multiple blocks, though the same number for each.

\item 	We expect precise analogues of \Crefrange{mainthm:1}{mainthm:2} and \Cref{maincor:1} for reductive groups in types $B$, $C$, $D$.

\end{itemize}

The bimodule in \Cref{mainconj:2} also seems closely related to work of Boixeda Alvarez--Losev~\cite{bl}.
Using the same setup as Oblomkov--Yun, but at an \emph{integral} slope $d$ rather than a regular elliptic slope, they construct a $(D_{\bbb{G}, \bbb{A}}^\trig(d), D_{\bbb{G}, \bbb{A}}^\trig(0))$-bimodule, where $D_{\bbb{G}, \bbb{A}}^\trig(\nu)$ denotes the \emph{trigonometric} DAHA rather than the rational one.
This seems to add a third row to the analogy in \eqref{eq:analogy}.
In a separate future paper, we will address evidence toward an analogue of \Cref{mainconj:2}(2) for the bimodule in~\cite{bl}.

\subsection{Acknowledgments}

We thank Gunter Malle for especially thorough comments and corrections.
We also thank Pablo Boixeda Alvarez, Olivier Dudas, Ivan Losev, George Lusztig, Andrew Mathas, Jean Michel, Kari Vilonen, Ben Webster, and Zhiwei Yun for useful comments, and Ian Grojnowski and Rapha\"el Rouquier for their kind interest.
During the preparation of this work, MT was supported by an NSF Mathematical Sciences Research Fellowship, Award DMS-2002238. TX thanks the support of ARC grant FL200100141.

\section{\texorpdfstring{$\Phi$-Harish-Chandra Series}{Φ-Harish-Chandra Series}}\label{sec:hc}

\subsection{}

For any associative algebra $H$ of finite type over a field, we write $\sf{Rep}(H)$ to mean the category of finitely-generated $H$-modules, and $\sf{K}_0(H)$ to mean the Grothendieck group of $\sf{Rep}(H)$.
For any finite group $\Gamma$, we set $\sf{K}_0(\Gamma) = \sf{K}_0(\bb{K}\Gamma)$, where $\bb{K}$ is any splitting field for $\Gamma$ of characteristic zero.
We will often identify isomorphism classes of representations of $\Gamma$ with their characters.
We write $\Irr(\Gamma)$ for the set of irreducible characters of $\Gamma$.
We will often write $1$ to denote the trivial character when $\Gamma$ can be inferred from context.

\subsection{}\label{subsec:ind-res}

Consider a prime power $q > 1$ and a connected, reductive algebraic group $G$ over $\bar{\bb{F}}_q$, equipped with a $q$-Frobenius map $F: G \to G$ defining an $\bb{F}_q$-structure.
For any $F$-stable Levi subgroup $L \subseteq G$, Lusztig introduced induction and restriction homomorphisms~\cite{lusztig_76}:
\begin{align}
R_L^G : \sf{K}_0(L^F) \rightleftarrows \sf{K}_0(G^F) : {}^\ast R_L^G.
\end{align}
For any parabolic subgroup $P \subseteq G$ containing $L$, not necessarily $F$-stable, the induction homomorphism $R_L^G$ may be defined using $\ur{H}_c^\ast(Y_{L \subseteq P}^G)$, where $Y_{L \subseteq P}^G$ is an algebraic variety over $\bar{\bb{F}}_q$ equipped with commuting actions of $G^F$ and $L^F$.
Here, we have written $\ur{H}_c^\ast(Y)$ to denote the compactly-supported $\ell$-adic cohomology of $Y_{\bar{\bb{F}}_q}$, where $\ell$ is invertible in $\bb{F}_q$.
The restriction homomorphism ${}^\ast R_L^G$ is defined as the right adjoint of $R_L^G$.
Bonnaf\'e, extending work of Lusztig, showed that $R_L^G, {}^\ast R_L^G$ do not depend on $P$ when $q$ is large enough~\cite{bo_98}.
For a more detailed exposition of Lusztig induction and restriction, see the second edition of the book~\cite{dm} by Digne--Michel, or the book~\cite{gm} by Geck--Malle.

The variety $Y_{L \subseteq P}^G$ forms an $L^F$-torsor over a $G^F$-variety $X_{L \subseteq P}^G$.
An irreducible character of $G^F$ is \dfemph{unipotent} if and only if it occurs in $\ur{H}_c^\ast(X_{T \subseteq B}^G)$ for some $F$-stable maximal torus $T$ and Borel $B$, or equivalently, in $ R_T^G(1)$ for some such $T$.
Lusztig observed that while $\Irr(G^F)$ grows in size with $q$, the subset of unipotent irreducible characters $\Uch(G^F)$ can be indexed in a way depending only on the root system of $G$ and its Frobenius action, not on $q$ itself~\cite{lusztig_78, lusztig_84}.

\begin{ex}\label[ex]{ex:principal}
Suppose that $A \subseteq G$ is a maximally $F$-split maximal torus, and that $B$ is an $F$-stable Borel containing $A$.
Then $Y_{A \subseteq B}^G = G^F/U^F$, where $U$ is the unipotent radical of $B$, and $X_{A \subseteq B}^G = G^F/B^F$.
The representation $R_A^G(1)$ is just $\ur{H}^0(G^F/B^F)$, the space of $\QL$-valued functions on $G^F/B^F$.
The irreducible constituents of $R_A^G(1)$ form a subset of $\Uch(G^F)$ called the \dfemph{unipotent principal series}, parametrized by the irreducible characters of the group $W^{G^F}_{A^F} \vcentcolon= N_{G^F}(A^F)/A^F$.
Later, we will discuss how to generalize this parametrization.
\end{ex}

\subsection{}

Brou\'e--Malle~\cite{bm_92} introduced generic finite reductive groups in order to study properties of the homomorphisms $R_L^G, {}^\ast R_L^G$ and the sets $\Uch(G^F)$ that only depend on $q$ through specializations of an indeterminate variable.
A \dfemph{generic finite reductive group} $\bbb{G}$, or \dfemph{generic group} for short, consists of:
\begin{enumerate}
\item 	A root datum $\Gamma_\bbb{G} = (X, R, X^\vee, R^\vee)$.

\item 	A coset of the form $[f]_\bbb{G} \vcentcolon= W_{\Gamma_\bbb{G}}f \subseteq \Aut(\Gamma_\bbb{G})$, where $W_{\Gamma_\bbb{G}}$ is the Weyl group of $\Gamma_\bbb{G}$ and $f$ a finite-order automorphism of $\Gamma_\bbb{G}$ normalizing $W_{\Gamma_\bbb{G}}$.
		We say that $\bbb{G}$ is \dfemph{split} if and only if $[f]_\bbb{G} = W_{\Gamma_\bbb{G}}$.

\end{enumerate}
An \dfemph{isomorphism} between generic groups is an isomorphism between the root data in (1), matching the coset data in (2).

We say that a generic group $\bbb{G}'$ is a \dfemph{generic subgroup} of $\bbb{G}$ if and only if its cocharacter lattice embeds into that of $\bbb{G}$, its root system embeds into that of $\bbb{G}$ as a parabolic subsystem, and $[f]_{\bbb{G}'} \subseteq [f]_\bbb{G}$.
In this case, we write $\bbb{G}' \leq \bbb{G}$.
To indicate that $\bbb{G}' \neq \bbb{G}$ as well, we write $\bbb{G}' < \bbb{G}$.

\begin{rem}\label{assump}
Our definitions exclude any generic group $\bbb{G}$ with a generic subgroup of the \dfemph{Suzuki} type ${}^2C_2$ or the \dfemph{Ree} types ${}^2G_2, {}^2F_4$.
Here, the letter and subscript indicate the root datum, while the superscript is the minimal order among elements of the coset datum.
\end{rem}

\subsection{}

Henceforth, $\bbb{G}$ is always a generic group.

Any choice of prime power $q > 1$ and representative $f \in [f]_\bbb{G}$ gives rise to a tuple $(G, T, F)$, where $G$ is a connected, reductive algebraic group over $\bar{\bb{F}}_q$ with $q$-Frobenius map $F$, as in \S\ref{subsec:ind-res}, and $T \subseteq G$ is an $F$-stable maximal torus.
Every tuple $(G, T, F)$ (such that $G^F$ contains no Suzuki or Ree subgroups) arises from a generic group in this way.
We set 
\begin{align}
\bbb{G}(q) = G^F.
\end{align}
If $(G', T', F')$ is the tuple arising from a generic subgroup $\bbb{G}' \leq \bbb{G}$, then we can embed $G'$ into $G$ so that $T' = T$ and $G'$ is $F$-stable and $F' = F|_{G'}$.
In this way, $\bbb{G}'(q)$ forms a subgroup of $\bbb{G}(q)$.

As an abstract group, $\bbb{G}(q)$ only depends on $q$ and $\bbb{G}$.
The orders of the groups $\bbb{G}(q)$ are generic, in the sense that there is a polynomial $|\bbb{G}|(x) \in \bb{Q}[x]$ such that $|\bbb{G}|(q) = |\bbb{G}(q)|$ for all $q$. 

\subsection{}\label{subsec:levi}

A \dfemph{Levi subgroup} of $\bbb{G}$ is a generic subgroup $\bbb{L} \leq \bbb{G}$ whose cocharacter lattice is the same as that of $\bbb{G}$.
For such $\bbb{L}$, the \dfemph{relative Weyl group} 
\begin{align}
W^\bbb{G}_\bbb{L} 
\vcentcolon= N_{W_{\Gamma_\bbb{G}}}([f]_\bbb{L})/W_{\Gamma_\bbb{L}}
\simeq (N_{W_{\Gamma_\bbb{G}}}(W_{\Gamma_\bbb{L}})/W_{\Gamma_\bbb{L}})^{[f]_\bbb{L}}
\end{align}
is a complex reflection group, isomorphic to $W^{\bbb{G}(q)}_{\bbb{L}(q)} \vcentcolon= N_{\bbb{G}(q)}(\bbb{L}(q))/\bbb{L}(q)$ for all $q$.
Note that the definition of $W^\bbb{G}_\bbb{L}$ in~\cite[75]{broue} has a typo.

A \dfemph{generic torus} is a generic group $\bbb{T}$ whose root and coroot lattices are empty.
For such $\bbb{T}$, the group $W_{\Gamma_\bbb{T}}$ is trivial, so we can write $[f]_\bbb{T} = \{f_\bbb{T}\}$.
The order of $\bbb{T}$ satisfies $|\bbb{T}|(x) = \det(x - f_\bbb{T} \mid X^\vee)$, where $X^\vee$ is the cocharacter lattice of $\bbb{T}$.
The orbit of $f_\bbb{T}$ in $[f]_\bbb{G}$ under conjugation by $W_{\Gamma_\bbb{G}}$ is called the \dfemph{type} of $\bbb{T}$ in $\bbb{G}$.
A \dfemph{subtorus} of $\bbb{G}$ is a generic subgroup $\bbb{T} \leq \bbb{G}$ given by a generic torus.

Every subtorus $\bbb{T} \leq \bbb{G}$ defines a Levi subgroup $Z_\bbb{G}(\bbb{T}) \leq \bbb{G}$, called its \dfemph{centralizer}.
The roots of $Z_\bbb{G}(\bbb{T})$ are the roots of $\bbb{G}$ 
orthogonal to the cocharacter lattice of $\bbb{T}$, and $[f]_{Z_\bbb{G}(\bbb{T})} = W_{\Gamma_{Z_\bbb{G}(\bbb{T})}}f_\bbb{T}$.

\subsection{}\label{subsec:relative Weyl}

There is a set $\Uch(\bbb{G})$ that indexes the sets $\Uch(\bbb{G}(q))$ uniformly in $q$.
Its construction is functorial with respect to isomorphisms between generic groups.
The degrees of the unipotent irreducible characters of the groups $\bbb{G}(q)$ are generic, in the sense that there is a polynomial $\Deg_\rho(x) \in \bb{Q}[x]$ for each $\rho \in \Uch(\bbb{G})$ such that $\Deg_\rho(q) = \deg \rho_q$ for all $q$, where $\rho_q \in \Uch(\bbb{G}(q))$ corresponds to $\rho$.

The maps $R_L^G, {}^\ast R_L^G$  have generic versions for unipotent characters.
If $G, F$ arise from $\bbb{G}$, and $L$ from $\bbb{L}$ for some Levi subgroup $\bbb{L} \leq \bbb{G}$, then we have maps
\begin{align}
R_\bbb{L}^\bbb{G} : \bb{Z}\Uch(\bbb{L}) \rightleftarrows \bb{Z}\Uch(\bbb{G}) : {}^\ast R_\bbb{L}^\bbb{G}
\end{align}
that recover Lusztig's maps on $\bb{Z}\Uch(G^F)$ and $\bb{Z}\Uch(L^F)$.
The maps $R_\bbb{L}^\bbb{G}, {}^\ast R_\bbb{L}^\bbb{G} $ are moreover compatible with isomorphisms between Levi subgroups induced by conjugation by $W_{\Gamma_\bbb{G}}$.

\subsection{}

Motivated by observations in the $\ell$-modular representation theory of $G^F$ for large primes $\ell$, Brou\'e--Malle--Michel used generic groups to formulate a generalization of Harish-Chandra theory, depending on an integer $e > 0$ by way of the cyclotomic polynomial $\Phi_e(x) \in \bb{Z}[x]$~\cite[\S{3}]{bmm}.
(For a fixed prime power $q > 1$, they take $e$ to be the order of $q$ in $\bb{F}_\ell$.)

As preparation:
We say that a generic torus $\bbb{T}$ is a \dfemph{$\Phi_e$-torus} if and only if $|\bbb{T}|(x)$ is a power of $\Phi_e(x)$.
We say that a Levi subgroup $\bbb{L} \leq \bbb{G}$ is \dfemph{$\Phi_e$-split} if and only if $\bbb{L} = Z_\bbb{G}(\bbb{T})$ for some $\Phi_e$-subtorus $\bbb{T} \leq \bbb{G}$.
In this case, we say that $\lambda \in \Uch(\bbb{L})$ is \dfemph{$\Phi_e$-cuspidal} if and only if ${}^\ast R_\bbb{M}^\bbb{L}(\lambda) = 0$ for any smaller $\Phi_e$-split Levi $\bbb{M} < \bbb{L}$.
Then we also say that $(\bbb{L}, \lambda)$ is a \dfemph{$\Phi_e$-cuspidal pair} for $\bbb{G}$.

For general $\bbb{L}$, there is a $W^\bbb{G}_\bbb{L}$-action on $\Uch(\bbb{L})$, arising from $W^{\bbb{G}(q)}_{\bbb{L}(q)}$-actions on the sets $\Irr(\bbb{L}(q))$ that stabilize their respective subsets $\Uch(\bbb{L}(q))$.
For all $\lambda \in \Uch(\bbb{L})$, we define its \dfemph{relative Weyl group} $W^\bbb{G}_{\bbb{L}, \lambda}$ to be the centralizer of $\lambda$ in $W^\bbb{G}_\bbb{L}$. 
It turns out that if $(\bbb{L}, \lambda)$ is an $\Phi_e$-cuspidal pair, then $W^\bbb{G}_{\bbb{L}, \lambda}$ is again a complex reflection group~\cite[Thm.\ 5.7]{broue}.

\begin{rem}\label[rem]{rem:lambda-indep}
For irreducible $\Gamma_\bbb{G}$ and any $e > 0$, the possible relative Weyl groups for the $\Phi_e$-cuspidal pairs of $\bbb{G}$ are listed in~\cite[\S{3}]{bm_93}.
For such $\bbb{G}$, case-by-case analysis shows that we always have $W^\bbb{G}_{\bbb{L}, \lambda} = W^\bbb{G}_\bbb{L}$, except when $\bbb{G}$ is of untwisted type $D_n$ for $n$ even or of type $E_7$.
See pages 48--53 and Table 1 in \cite{bmm}.
\end{rem}

\subsection{}

For any $\Phi_e$-cuspidal pair $(\bbb{L}, \lambda)$ of $\bbb{G}$, the corresponding \dfemph{$\Phi_e$-Harish-Chandra series} of $\Uch(\bbb{G})$ is the set
\begin{align}
\Uch(\bbb{G})_{\bbb{L}, \lambda} = \{\rho \in \Uch(\bbb{G}) \mid \text{$\rho$ occurs in $R_\bbb{L}^\bbb{G}(\lambda)$}\}.
\end{align}
The action of $W_{\Gamma_\bbb{G}}$ on the set of Levi subgroups extends to an action on the set of $\Phi_e$-cuspidal pairs.
The $\Phi_e$-Harish-Chandra series of a pair only depends on its $W_{\Gamma_\bbb{G}}$-orbit.

Let $\sf{HC}_e(\bbb{G})$ be the set of $W_{\Gamma_\bbb{G}}$-orbits of $\Phi_e$-cuspidal pairs.
Below, items (1) and (2a)--(2c) comprise~\cite[Thm.\ 3.2]{bmm}, while item (2d) follows from~\cite[Thm.\ 5.24]{bmm}:
See~\cite[75--76]{broue}.

\begin{thm}[Brou\'e--Malle--Michel]\label[thm]{thm:bmm}
For any integer $e > 0$:
\begin{enumerate}
\item
	The $\Phi_e$-Harish-Chandra series form a partition of $\Uch(\bbb{G})$:
	That is,
	\begin{align}
	\Uch(\bbb{G}) = \coprod_{[\bbb{L}, \lambda] \in \sf{HC}_e(\bbb{G})} \Uch(\bbb{G})_{\bbb{L}, \lambda}.
	\end{align}
		
\item
	For any $\Phi_e$-cuspidal pair $(\bbb{L}, \lambda)$ of $\bbb{G}$, there is a map
	\begin{align}
	(\varepsilon_{\bbb{L}, \lambda}^\bbb{G}, \chi_{\bbb{L}, \lambda}^\bbb{G}) : \Uch(\bbb{G})_{\bbb{L}, \lambda} \to \{\pm 1\} \times \Irr(W^\bbb{G}_{\bbb{L}, \lambda}),
	\end{align}
	such that $\varepsilon_{\bbb{L}, \lambda}^\bbb{G} \chi_{\bbb{L}, \lambda}^\bbb{G} : \bb{Z}\Uch(\bbb{G})_{\bbb{L}, \lambda} \to \bb{Z}\Irr(W^\bbb{G}_{\bbb{L}, \lambda})$ is an isometry with respect to appropriate multiplicity pairings.
	In particular, the map $\chi_{\bbb{L}, \lambda}^\bbb{G}$ is bijective.
	Moreover:
		\begin{enumerate}	
		\item 	For any $\Phi_e$-split Levi $\bbb{M}$ with $\bbb{L} \leq \bbb{M} \leq \bbb{G}$, we have the following commutative diagram, in which the horizontal arrows are induced by linearity:
		\begin{equation}
\begin{tikzpicture}[baseline=(current bounding box.center), >=stealth]
\matrix(m)[matrix of math nodes, row sep=3em, column sep=4em, text height=2ex, text depth=0.5ex]
{ 		
	\bb{Z}\Uch(\bbb{M})_{\bbb{L}, \lambda}
		&\bb{Z}\Irr(W^\bbb{M}_{\bbb{L}, \lambda})\\	
	\bb{Z}\Uch(\bbb{G})_{\bbb{L}, \lambda}
		&\bb{Z}\Irr(W^\bbb{G}_{\bbb{L}, \lambda})\\	
};
\path[->, font=\scriptsize, auto]
(m-1-1)		edge node{$\varepsilon_{\bbb{L}, \lambda}^\bbb{M} \chi_{\bbb{L}, \lambda}^\bbb{M}$} (m-1-2)
(m-1-1)		edge node[left]{$R_\bbb{M}^\bbb{G}$} (m-2-1)
(m-1-2)		edge node{$\Ind_{W^\bbb{M}_{\bbb{L}, \lambda}}^{W^\bbb{G}_{\bbb{L}, \lambda}}$} (m-2-2)
(m-2-1)		edge node{$\varepsilon_{\bbb{L}, \lambda}^\bbb{G} \chi_{\bbb{L}, \lambda}^\bbb{G}$} (m-2-2);
\end{tikzpicture}
\end{equation}
				
		\item 	The collection of maps $(\varepsilon_{\bbb{L}, \lambda}^\bbb{G}, \chi_{\bbb{L}, \lambda}^\bbb{G})$ is stable under the action of $W_{\Gamma_\bbb{G}}$ on the set of $\Phi_e$-cuspidal pairs.
		
		\item 	$\chi_{\bbb{G}, \lambda}^\bbb{G}(\lambda)$ is the trivial character of the trivial group $W^\bbb{G}_{\bbb{G}, \lambda}$.
				
		\item 	If $\rho \in \Uch(\bbb{G})$ satisfies $\Deg_\rho(x) \not\equiv 0 \pmod{\Phi_e(x)}$, then $\rho \in \Uch(\bbb{G})_{\bbb{T}, 1}$, where $\bbb{T} \leq \bbb{G}$ is a $\Phi_e$-split maximal torus.
				In this case,
		\begin{align}
		\Deg_\rho(x) \equiv \varepsilon_{\bbb{L}, \lambda}^\bbb{G}(\rho) \deg \chi_{\bbb{L}, \lambda}^\bbb{G}(\rho)
			\pmod{\Phi_e(x)}.
		\end{align}

		\end{enumerate}
\end{enumerate}
\end{thm}

\begin{ex}\label[ex]{ex:principal-generic}
The generic version of \Cref{ex:principal} consists of a generic group $\bbb{G}$ and a $\Phi_1$-split maximal torus $\bbb{A} \leq \bbb{G}$.
If $G, F, A$ arise from $q, \bbb{G}, \bbb{A}$, then $A$ is a maximally $F$-split maximal torus of $G$, and $\Uch(\bbb{G})_{\bbb{A}, 1}$ parametrizes the unipotent principal series of $G^F$.
Moreover, the endomorphism algebra $H^G_{A, 1}(q) \vcentcolon= \End_{\QL G^F}(R_A^G(1))$ is isomorphic to $\QL W^\bbb{G}_\bbb{A}$.

Let $\rho_q \in \Uch(G^F)$ correspond to $\rho \in \Uch(\bbb{G})_{\bbb{A}, 1}$.
For all $\rho \in \Uch(\bbb{G})_{\bbb{A}, 1}$, the sign $\varepsilon_{\bbb{A}, 1}^\bbb{G}(\rho)$ is positive.
The map $\chi_{\bbb{A}, 1}^\bbb{G} : \Uch(\bbb{G})_{\bbb{A}, 1} \to \Irr(W^\bbb{G}_\bbb{A})$ is determined by the existence of a $(G^F, H^G_{A, 1}(q))$-bimodule isomorphism
\begin{align}
R_A^G(1) = \bigoplus_{\rho \in \Uch(\bbb{G})_{\bbb{A}, 1}}
	\rho_q \otimes \chi_{\bbb{A}, 1, q}^\bbb{G}(\rho),
\end{align}
where $\chi_{\bbb{A}, 1, q}^\bbb{G}(\rho)$ corresponds to $\chi_{\bbb{A}, 1}^\bbb{G}(\rho)$ under $H^G_{A, 1}(q) \simeq \QL W^\bbb{G}_\bbb{A}$.
\end{ex}

\subsection{Regular Numbers}\label{subsec:regular}

The existence of a $\Phi_e$-split maximal torus constrains $e$.
To explain how this works, we review some notions due to Springer, following~\cite[\S{5B}]{bmm_99}.

Suppose that $C$ is a complex reflection group with reflection representation $V$ over $\bb{C}$, and that $f$ is a finite-order automorphism of $V$ normalizing $C$.
Let $V^\reg \subseteq V$ be the open locus of elements on which $C$ acts freely.
For any root of unity $\zeta \in \bb{C}^\times$ and $w \in C$, we say that $wf$ is \dfemph{$\zeta$-regular}, or just \dfemph{regular}, if and only if $wf$ has an eigenvector $v \in V^\reg$ with eigenvalue $\zeta$.

In particular, if $C = W_{\Gamma_\bbb{G}}$ and $Cf = [f]_\bbb{G}$ for some generic group $\bbb{G}$, then $C$ is crystallographic, so by Galois theory, $wf$ is $\zeta$-regular if and only if it is $\zeta'$-regular for any other root of unity $\zeta'$ of the same order: say, $e$.
In this case, we refer to $e$ as a \dfemph{regular number} for $\bbb{G}$, and to $\zeta$-regular elements of $Cf$ as \dfemph{$\Phi_e$-regular}.

\begin{ex}\label[ex]{ex:twisted-coxeter}
Suppose that there is an $f$-stable system of simple reflections $S \subseteq C$, or in other words, that $f$ is an automorphism of the Coxeter system $(C, S)$.
Let $w \in C$ be the product, in any order, of a full set of representatives for the $f$-orbits on $S$, and let $c = wf$.
Then $c$ is a regular element of $Cf$.
Elements that take this form for some $f$-stable $S$ and choice of $f$-orbit representatives are called \dfemph{twisted Coxeter elements} of $Cf$~\cite[\S{7.3}]{springer_74}.

In particular, if $\bbb{G}$ is a generic group for which $\Gamma_\bbb{G}$ is irreducible and $f$ is induced by a Dynkin-diagram automorphism, then we define the \dfemph{twisted Coxeter number} of $\bbb{G}$ to be the order of any twisted Coxeter element of $[f]_\bbb{G}$.
\end{ex}

\begin{prop}[Brou\'e--Malle--Michel]\label[prop]{prop:regular}
If $e$ is a regular number for $\bbb{G}$, then $\bbb{G}$ admits a $\Phi_e$-split maximal torus.
Conversely, if $\bbb{T} \leq \bbb{G}$ is a $\Phi_e$-split maximal torus, then $\bbb{T}$ has type $[wf]$ for some $\Phi_e$-regular element $wf \in [f]_\bbb{G}$.
\end{prop}

\begin{proof}
The first claim follows from~\cite[Prop.\ 5.8]{bmm_99}.

To show the second claim:
Suppose that $\bbb{T} = Z_\bbb{G}(\bbb{S})$ for some $\Phi_e$-torus $\bbb{S}$.
If $\bbb{T}$ is itself a torus, then $\bbb{S}$ must be maximal or \dfemph{Sylow} in the terminology of~\cite{bmm, broue}.
Then by~\cite[Cor.\ 5.10]{bmm_99}, $e$ must be a regular number for $\bbb{G}$, so by the first claim, there is a torus of type $[wf]$ for some $\Phi_e$-regular $wf$.
Again by~\cite[Cor.\ 5.10]{bmm_99}, $f_\bbb{T}$ must be conjugate to $wf$ under $W_{\Gamma_\bbb{G}}$, so $\bbb{T}$ also has type $[wf]$.
\end{proof}

\subsection{Singular Numbers}\label{subsec:singular}

The \dfemph{adjoint form} of $\bbb{G}$ is the generic group $\bbb{G}_\ad$ defined by $\Gamma_{\bbb{G}_\ad} = (Q, R, Q^\vee, R^\vee)$ and $[f]_{\bbb{G}_\ad} = W_{\Gamma_{\bbb{G}_\ad}}f|_{\Gamma_{\bbb{G}_\ad}}$, where $Q = \bb{Z}R$.
Using it, we can state~\cite[Prop.~2.9]{bmm}:

\begin{prop}[Brou\'e--Malle--Michel]\label[prop]{prop:cuspidal}
An element $\rho \in \Uch(\bbb{G})$ is $\Phi_e$-cuspidal if and only if the largest power of $\Phi_e(x)$ dividing $\Deg_\rho(x)$ equals the largest power of $\Phi_e(x)$ dividing $|\bbb{G}_\ad|(x)$, in $\bb{Q}[x]$.
\end{prop}

Observe that if $\rho \in \Uch(\bbb{G})$ is $\Phi_e$-cuspidal, then there is a $\Phi_e$-Harish-Chandra series $\Uch(\bbb{G})_{\bbb{G}, \rho}$ consisting of $\rho$ alone.

We say that $e$ is a \dfemph{singular number} for $\bbb{G}$ if and only if (a positive power of) $\Phi_e$ divides $|\bbb{G}_\ad|$.
Observe that there are finitely many such numbers.
For groups of exceptional type, we list them in \Cref{sec:exceptional}.

\begin{rem}\label[rem]{rem:singular}
As Malle noted to us, the facts below follow from the Sylow theorems:
\begin{itemize}
\item 	Every regular number is singular.
\item 	If $e$ is \emph{not} a singular number for $\bbb{G}$, then every element of $\Uch(\bbb{G})$ is $\Phi_e$-cuspidal.
Hence the $\Phi_e$-Harish-Chandra partition is nontrivial only if $e$ is a singular number.
\end{itemize}
\end{rem}

\section{Cyclotomic Hecke Algebras}\label{sec:hecke}

\subsection{}

In~\cite{lusztig_coxeter}, Lusztig determined the $G^F$-equivariant endomorphisms of the representation $\ur{H}_c^\ast(X_{T \subseteq B}^G)$ in the case where $T \subseteq G$ is a \dfemph{Coxeter torus}: a maximal torus of type $[c]$ for some twisted Coxeter element $c$.
He observed that this Hecke algebra is a $q$-analogue of the group algebra of $W^{G^F}_{T^F}$, with parameters given by the eigenvalues of Frobenius.
Soon afterward, in~\cite{lusztig_77}, Lusztig determined the corresponding algebras for various unipotent cuspidal pairs arising from $F$-split Levi subgroups of classical groups.
Again, they turn out to be $q$-analogues of the group algebras of the associated relative Weyl groups.
Lusztig's calculations were later extended to general $F$-split Levis in~\cite{lusztig_78}, and to non-unipotent characters by Howlett--Lehrer in~\cite{hl}.

Let $C$ be a complex reflection group of the form $W^\bbb{G}_{\bbb{L}, \lambda}$ for some generic group $\bbb{G}$, integer $e > 0$, and $\Phi_e$-cuspidal pair $(\bbb{L}, \lambda)$, as in \S\ref{subsec:relative Weyl}.
In~\cite{bm_93, bmr}, Brou\'e--Malle--Rouquier study a ring $H_C(\vec{u})$ that only depends on the structure of $C$ as a complex reflection group.
It may be viewed as a multi-parameter version of the Hecke algebras studied in the works above, in which the parameters are also allowed to be as generic as possible.

In~\cite{bm_93}, motivated by~\cite{lusztig_coxeter, lusztig_78} and Brou\'e's abelian defect group conjecture~\cite[Ch.\ VI]{broue}, Brou\'e--Malle defined a specialization of $H_{W^\bbb{G}_{\bbb{L}, \lambda}}(\vec{u})$ depending on $\bbb{G}$ and $(\bbb{L}, \lambda)$ and sending its parameters to various rational powers of a single variable $x$.
They conjectured that the resulting ring $H^\bbb{G}_{\bbb{L}, \lambda}(x)$ specializes, in turn, to all of the $\bbb{G}(q)$-endomorphism algebras arising from the Deligne--Lusztig varieties associated with $\bbb{G}$ and $(\bbb{L}, \lambda)$.
We review their work below.

\subsection{}\label{subsec:bmr}

First, we review $H_C(\vec{u})$, following~\cite{bmr} and~\cite[Ch.\ II--III]{broue}.

Let $C$ be an arbitrary finite complex reflection group and $V = V(C)$ its reflection representation over $\bb{C}$.
Let $\cal{A} = \cal{A}(C)$ be the set of hyperplanes in $V$ fixed by pseudo-reflections in $C$, so that $V^\reg = V - \bigcup_{H \in \cal{A}} H$, and let $\Br_C = \pi_1(V^\reg / C)$, the braid group of $C$.
For each orbit $\cal{C} \in \cal{A}/C$ and hyperplane $H \in \cal{C}$, let $\sigma_H \in \Br_C$ be a choice of a generator of the monodromy around $H$ that is a distinguished braid reflection in the sense of~\cite[Def.\@ 2.15]{bmr}.
Let the pseudo-reflection $s_H$ be the image of $\sigma_\cal{C}$ under the quotient map $\Br_C \to C$, and let $e_\cal{C}$ be the order of $s_H$, which only depends on $\cal{C}$.
The ``distinguished'' condition on $\sigma_H$ means that
\begin{align}\label{eq:distinguished}
\textstyle\det_V(s_H) = e^{2\pi i/e_\cal{C}}
	\quad\text{for all $H \in \cal{C}$}.
\end{align}
Let $\bb{Z}[\vec{u}_C^{\pm 1}] = \bb{Z}[u_{\cal{C}, j}^{\pm 1} \mid \cal{C} \in \cal{A}/C,\, 0 \leq j < e_\cal{C}]$.
When $C$ is clear from context, we write $\vec{u}^{\pm 1}$ in place of $\vec{u}_C^{\pm 1}$.
Let
\begin{align}
H_C(\vec{u}) 
	&= \frac{\bb{Z}[\vec{u}^{\pm 1}]\Br_C}{\langle (\sigma_H - u_{\cal{C}, 0}) \cdots (\sigma_H - u_{\cal{C}, e_\cal{C} - 1}) \mid \cal{C} \in \cal{A}/C,\, H \in \cal{C}\rangle}.
\end{align}
It was conjectured in~\cite{bmr} for all $C$, and proved for the groups denoted $C = G(m, p, n)$ in Shephard--Todd notation, that $H_C(\vec{u})$ is free over $\bb{Z}[\vec{u}^{\pm 1}]$ of rank $|C|$.
The exceptional cases were shown through the effort of many authors:
See~\cite{ts} and the references therein.
For us, the following result from~\cite{etingof} suffices:

\begin{thm}[Etingof--Rains, Losev, Marin--Pfeiffer]\label[thm]{thm:free}
Suppose that $\bb{K}$ is a field of characteristic zero.
Then $\bb{K} \otimes_\bb{Z} H_C(\vec{u})$ is a free module over $\bb{K}[\vec{u}^{\pm 1}]$ of rank $|C|$.
In particular, if we fix a ring homomorphism $\bb{Z}[\vec{u}^{\pm 1}] \to \bb{K}$, then the corresponding base change $\bb{K} \otimes_{\bb{Z}[\vec{u}^{\pm 1}]} H_C(\vec{u})$ is a $\bb{K}$-algebra of dimension $|C|$.
\end{thm}

\begin{ex}\label[ex]{ex:group-alg}
Let $\bb{K}$ be any ring containing a primitive $e_\cal{C}$th root of unity $\zeta_{e_\cal{C}}$ for all $\cal{C}$.
Then the group algebra $\bb{K}C$ is precisely the base change of $H_C(\vec{u})$ along the homomorphism $\bb{Z}[\vec{u}^{\pm 1}] \to \bb{K}$ that sends $u_{\cal{C}, j} \mapsto \zeta_{e_\cal{C}}^j$~\cite[45--46]{broue}.
\end{ex}

\subsection{}\label{subsec:bm_93}

Let $\bb{Q}_\cyc$ be the maximal cyclotomic field extension of $\bb{Q}$, and let $\bb{Z}_\cyc$ be its ring of integers.
Fix an inverse system of roots of unity $\{\zeta_m\}_{m \geq 1} \subset \bb{Z}_\cyc$, meaning $\zeta_m$ is of order $m$ and $\zeta_{mn}^n = \zeta_m$ for all $m, n$.

Fix $\bbb{G}$, $e$, and $\Phi_e$-cuspidal $(\bbb{L}, \lambda)$.
Take $C = W^\bbb{G}_{\bbb{L}, \lambda}$ in the setup of the previous subsection.
In~\cite{bm_93, malle}, Brou\'e--Malle define a ring homomorphism 
$$\cal{S}^\bbb{G}_{\bbb{L}, \lambda} : \bb{Z}[\vec{u}^{\pm 1}] \to \bb{Z}_\cyc[x^{\pm 1/\infty}]$$ of the form
\begin{align}\label{eq:specialization}
u_{\cal{C}, j} \mapsto \zeta_{e_\cal{C}}^j (\zeta_e^{-1} x)^{m_{\cal{C}, j}}
	\quad\text{for some $m_{\cal{C}, j} \in \bb{Q}$}.
\end{align}
Here, the parameters $m_{\cal{C}, j}$ can depend on $\bbb{G}$ and $(\bbb{L}, \lambda)$, not just on the structure of $W^\bbb{G}_{\bbb{L}, \lambda}$ as a complex reflection group.
Let
\begin{align}
H^\bbb{G}_{\bbb{L}, \lambda}(x)
	\vcentcolon= \bb{Z}_\cyc[x^{\pm 1/\infty}] \otimes_{\bb{Z}[\vec{u}^{\pm 1}]} H_{W^\bbb{G}_{\bbb{L}, \lambda}}(\vec{u})
\end{align}
denote the base change of $H_{W^\bbb{G}_{\bbb{L}, \lambda}}(\vec{u})$ along $\cal{S}^\bbb{G}_{\bbb{L}, \lambda}$.
Strictly speaking, $H^\bbb{G}_{\bbb{L}, \lambda}(x)$ is an abuse of notation, as this algebra may involve higher roots of $x$.
Note, as well, that $H^\bbb{G}_{\bbb{L}, \lambda}(x)|_{x \to \zeta_e}$ recovers the group algebra of $W^\bbb{G}_{\bbb{L}, \lambda}$ over $\bb{Z}_\cyc$.

Brou\'e--Malle infer the general definition of $\cal{S}^\bbb{G}_{\bbb{L}, \lambda}$ from the cases in which $\Gamma_\bbb{G}$ is irreducible and $[f]_\bbb{G}$ arises from a Dynkin-diagram automorphism.
Here, they show that $\cal{S}^\bbb{G}_{\bbb{L}, \lambda}$ is essentially determined by certain numerical predictions related to the conjecture below.
For an overview of these predictions, see (U5.3.1)--(U.5.3.4) in the survey~\cite{broue}.
For $\bbb{G}$ of type $A$, we describe $\cal{S}^\bbb{G}_{\bbb{L}, \lambda}$ explicitly in \Cref{sec:gl-gu}.
For $\bbb{G}$ of exceptional type, we describe $\cal{S}^\bbb{G}_{\bbb{L}, \lambda}$ in \Cref{sec:exceptional}.

Suppose that $G, F, L$ arise from $q, \bbb{G}, \bbb{L}$. 
Recall the $G^F$-variety $Y_{L \subseteq P}^G$ over $\bb{F}_q$.
For any $\lambda_q \in \Uch(L^F)$, 
we write $\ur{H}_c^\ast(Y_{L \subseteq P}^G)[\lambda_q]$ to denote the $\lambda_q$-isotypic component of $\ur{H}_c^\ast(Y_{L \subseteq P}^G)$, viewed as a graded representation of $G^F$ of finite dimension over $\QL$.
What follows is essentially conjecture (HC) of~\cite[84]{broue}.

\begin{conj}[Brou\'e--Malle]\label[conj]{conj:end}
For $\bbb{G}, e, (\bbb{L}, \lambda)$ as above, and for any prime power $q > 1$ and prime $\ell$ invertible in $\bb{F}_q$, there is an isomorphism of algebras
\begin{align}\label{eq:end}
\QL \otimes_{\bb{Z}_\cyc[x^{\pm1/\infty}]} H^\bbb{G}_{\bbb{L}, \lambda}(x)
	\simeq \End_{\QL G^F}(\ur{H}_c^\ast(Y_{L \subseteq P}^G)[\lambda_q]),
\end{align}
where on the left, the base change sends $x^{1/n}$ to an $n$th root of $q$ in $\QL$ for each integer $n > 0$, and on the right, $G, F, L, \lambda_q$ arise from $q, \bbb{G}, \bbb{L}, \lambda$.
\end{conj}

Note that the conjecture would imply that $H^\bbb{G}_{\bbb{L}, \lambda}(x)$ is uniquely determined by the isomorphisms \eqref{eq:end}, up to Galois automorphisms of $\bb{Z}_\cyc$.

\begin{ex}\label[ex]{ex:principal-hecke}
In \Cref{ex:principal-generic}, write $W = W^\bbb{G}_\bbb{A}$.
Then $H_W(\vec{u})$ takes the form
\begin{align}
H_W(\vec{u}) = \bb{Z}[\vec{u}^{\pm 1}]\Br_W/\langle (\sigma_H - u_{\cal{C}, 0})(\sigma_H - u_{\cal{C}, 1}) \mid \cal{C} \in \cal{A}/W,\, H \in \cal{C}\rangle.
\end{align}
Meanwhile, $H^G_{A, 1}(q)$ is generated by Hecke operators $T_s$ for each simple reflection $s \in W$, modulo braid relations and quadratic relations.

Suppose that $\bbb{G}$ is split.
Then $W = W_{\Gamma_\bbb{G}}$, and the quadratic relations in~\cite{iwahori} take the form $(T_s + 1)(T_s - q) = 0$.
The homomorphism $\cal{S}^\bbb{G}_{\bbb{A}, 1}$ in~\cite{bm_93} sends $(u_{\cal{C}, 0}, u_{\cal{C}, 1}) \mapsto (1, -x)$ for all $\cal{C}$.
We take \eqref{eq:end} to be given by $\sigma_H \mapsto -T_{s_H}$.
\end{ex}

\subsection{}

In cases where $(\bbb{L}, \lambda) = (\bbb{T}, 1)$ for some maximal torus $\bbb{T}$ of $\bbb{G}$, Brou\'e--Michel constructed a $\Br_{W^\bbb{G}_\bbb{T}}$-action on the associated Deligne--Lusztig cohomologies, and conjectured that the $H^\bbb{G}_{\bbb{T}, 1}(x)$-action of \Cref{conj:end} would arise from this braid action.
See~\cite{bm_97} and~\cite[84--88]{broue} for details.
This conjecture has since been established in many cases by Digne, Michel, and Rouquier~\cite{dmr, dm_06}, building on~\cite{lusztig_coxeter}.
To summarize the state of the art, we follow~\cite{dm_06}.

\begin{thm}[Lusztig, Digne--Michel--Rouquier]\label[thm]{thm:ldmr}
Let $\bbb{T} \leq \bbb{G}$ be a maximal torus.
If any of the following hold, then \Cref{conj:end} holds for $(\bbb{L}, \lambda) = (\bbb{T}, 1)$.
Throughout, $f$ is a Dynkin-diagram automorphism representing $[f]_\bbb{G}$.
\begin{enumerate}
\item 	$\bbb{T}$ is $\Phi_1$-split.

\item 	$\bbb{T}$ is \dfemph{Coxeter} in the sense of~\cite[(1.14)]{lusztig_coxeter}.
		That is, $\bbb{T}$ is of type $[c]$ for some twisted Coxeter element $c \in [f]_\bbb{G}$, as defined in \Cref{ex:twisted-coxeter}.

\item 	$\bbb{T}$ is of type $[w_0f]$, where $w_0 \in W_{\Gamma_\bbb{G}}$ is the longest element.

\item 	$\bbb{G}$ is split of type $A$.

\item 	$\bbb{G}$ is of type $B$ and $\bbb{T}$ is $\Phi_e$-split for some even $e > 0$.

\item 	$\bbb{G}$ is split of type $D_4$ and $\bbb{T}$ is $\Phi_4$-split.

\end{enumerate}
\end{thm}

\begin{proof}
\begin{enumerate}
\item 	If $\bbb{G}$ is split, then this is due to Iwahori~\cite{iwahori}.
		The map $\cal{S}^\bbb{G}_{\bbb{T}, 1}$ was described above, in \Cref{ex:principal-hecke}.
		If $\bbb{G}$ is not split, then it is due to Lusztig~\cite{lusztig_78}, and the map $\cal{S}^\bbb{G}_{\bbb{T}, 1}$ is given by Table II of \emph{ibid.}

\item 	This is due to Lusztig~\cite{lusztig_coxeter}.
		The map $\cal{S}^\bbb{G}_{\bbb{T}, 1}$ sends the variables $u_{\cal{C}, j}$ to the powers of $x$ that specialize to the eigenvalues of $F$ on $\ur{H}_c^\ast(X_{T \subseteq B}^G)$, when $G, F, T$ arise from $q, \bbb{G}, \bbb{T}$.
		When $G$ is almost-simple, the eigenvalues are listed in (7.3) of \emph{ibid.}

\item 	This is~\cite[Th.\ 5.4.1]{dmr}, reviewed in~\cite[Prop.\ 7.2]{dm_06}.

\item 	This is~\cite[Thm.\ 10.1]{dm_06}.
		Note that in type $A$, the Weyl group contains two conjugacy classes of regular elements, corresponding to the two cases of the theorem.
		The requirement that $\bbb{G}$ be split is due to Section 9 of \emph{ibid.}
		
\item 	This is~\cite[Thm.\ 11.1]{dm_06}.

\item 	This is~\cite[Prop.\ 12.2]{dm_06}.
		\qedhere

\end{enumerate}
\end{proof}

\subsection{}

Beyond tori, \Cref{conj:end} has been established in the cases where:

\begin{enumerate}
\item 	$\bbb{L}$ is an arbitrary $\Phi_1$-split Levi.
		This case is again due to Lusztig, with the map $\cal{S}^\bbb{G}_{\bbb{L}, \lambda}$ given by~\cite[Table II]{lusztig_78}.

\item 	$\bbb{G}$ is a generic general linear group and $W^\bbb{G}_{\bbb{L}, \lambda}$ is cyclic.
		This case follows from work of Dudas~\cite{dudas}:
		See \Cref{prop:gl}(2) below.

\end{enumerate}
We defer the precise statements to \Cref{sec:gl-gu}.

\section{Conjectures about Blocks}\label{sec:blocks}



\subsection{}

Throughout this section, to simplify our discussion, we embed $\bb{Q}_\cyc$ into $\bb{C}$ via the map that sends $\zeta_m \mapsto e^{2\pi i/m}$.


For any associative algebra $H$ of finite type over a field, form the coarsest equivalence relation on the set of isomorphism classes of indecomposable, finitely-generated $H$-modules such that $[M], [M']$ are equivalent whenever $M, M'$ share a simple Jordan--H\"older factor.
Then each equivalence class defines a \dfemph{block} of the category $\sf{Rep}(H)$: the full Serre subcategory generated by the Jordan--H\"older factors of the indecomposable modules $M$ such that $[M]$ belongs to the equivalence class.
For any block $\sf{b}$, we write $[\sf{b}]$ to denote its image in $\sf{K}_0(H)$.

\subsection{}\label{subsec:decomp}

Let $C$ be any finite complex reflection group.
Using the notation $\vec{u} = \vec{u}_C$ of \Cref{sec:hecke}, let $\cal{T} : \bb{Z}[\vec{u}^{\pm 1}] \to \bb{C}$ be an arbitrary homomorphism, and let
\begin{align}
H_{C, \cal{T}} = \bb{C} \otimes_{\bb{Z}[\vec{u}^{\pm 1}]} H_C(\vec{u})
\end{align}
denote the base change of $H_C(\vec{u})$ along $\cal{T}$.
Let $\bb{K} \supseteq \bb{Q}_\cyc(u_{\cal{C}, j} \mid \cal{C}, j)$ be a splitting field for $H_C(\vec{u})$, and let $\bb{K}H_C(\vec{u})$ be the base change of $H_C(\vec{u})$ to $\bb{K}$.
Then there is a decomposition map
\begin{align}
\sf{d}_\cal{T} : \sf{K}_0(\bb{K}H_C(\vec{u})) \to \sf{K}_0(H_{C, \cal{T}})
\end{align}
uniquely determined by a certain compatibility with characteristic polynomials, as explained in~\cite[Ch.\ 7]{gp} and~\cite[Ch.\ 3]{gj}.

Let $\Irr(\bb{K}H_C(\vec{u}))$ be the set of simple $\bb{K}H_C(\vec{u})$-modules up to isomorphism.
By combining \Cref{thm:free}, \Cref{ex:group-alg}, and Tits deformation~\cite[Thm.\ 7.4.6]{gp}, we obtain a bijection $\Irr(C) \xrightarrow{\sim} \Irr(\bb{K}H_C(\vec{u}))$.
For any block $\sf{b}$ of $H_{C, \cal{T}}$, let $\Irr(C)_\sf{b}$ be the preimage of $[\sf{b}]$ along the composition of maps
\begin{align}\label{eq:brauer}
\Irr(C)
	\xrightarrow{\sim}
	\Irr(\bb{K}H_C(\vec{u}))
	\subset
	\sf{K}_0(\bb{K}H_C(\vec{u}))
	\xrightarrow{\sf{d}_\cal{T}}
	\sf{K}_0(H_{C, \cal{T}}).
\end{align}
This composition does not depend on $\bb{K}$.
Abusing language, we refer to the sets $\Irr(C)_\sf{b}$ as the \dfemph{$H_{C, \cal{T}}$-blocks} of $\Irr(C)$.

\subsection{}\label{subsec:bij}

Now let $\bbb{G}$ be a generic group.
For the first time, we fix two separate integers $e, m > 0$.
Let $(\bbb{L}, \lambda)$, \emph{resp.}\@ $(\bbb{M}, \mu)$, be a $\Phi_e$-cuspidal pair, \emph{resp.}\@ a $\Phi_m$-cuspidal pair, for $\bbb{G}$.
As in the introduction, let
\begin{align}
\Irr(W^\bbb{G}_{\bbb{L}, \lambda})_{\bbb{M}, \mu} \subseteq \Irr(W^\bbb{G}_{\bbb{L}, \lambda})
	\quad\text{and}\quad
	\Irr(W^\bbb{G}_{\bbb{M}, \mu})_{\bbb{L}, \lambda} \subseteq \Irr(W^\bbb{G}_{\bbb{M}, \mu})
	\end{align}
be the images of the maps
\begin{align}
\Irr(W^\bbb{G}_{\bbb{L}, \lambda})
	\xleftarrow{\chi_{\bbb{L}, \lambda}^\bbb{G}}
	\Uch(\bbb{G})_{\bbb{L}, \lambda} \cap \Uch(\bbb{G})_{\bbb{M}, \mu}
	\xrightarrow{\chi_{\bbb{M}, \mu}^\bbb{G}}
		\Irr(W^\bbb{G}_{\bbb{M}, \mu}).
\end{align}
Form the algebra 
\begin{align}
H^\bbb{G}_{\bbb{L}, \lambda}(\zeta_m) = H_{W^\bbb{G}_{\bbb{L}, \lambda}, \cal{T}^\bbb{G}_{\bbb{L}, \lambda, m}},
\end{align} 
where $\cal{T}^\bbb{G}_{\bbb{L}, \lambda, m}$ is the homomorphism defined by
\begin{align}
\bb{Z}[\vec{u}_C^{\pm 1}]
	\xrightarrow{\cal{S}^\bbb{G}_{\bbb{L}, \lambda}}
		\bb{Z}_\cyc[x^{\pm1/\infty}]
	\xrightarrow{x^{1/n} \mapsto \zeta_{mn}}
		\bb{C},
\end{align}
and $\cal{S}^\bbb{G}_{\bbb{L}, \lambda}$ is reviewed in~\S\ref{subsec:bm_93}.
Let $H^\bbb{G}_{\bbb{M}, \mu}(\zeta_e)$, $\cal{T}^\bbb{G}_{\bbb{M}, \mu, e}$ be defined similarly.

\begin{conj}\label[conj]{conj:bij}
For any $\bbb{G}, e, (\bbb{L}, \lambda), m, (\bbb{M}, \mu)$:
\begin{enumerate}
\item[(I)]
		$\Irr(W^\bbb{G}_{\bbb{L}, \lambda})_{\bbb{M}, \mu}$ and $\Irr(W^\bbb{G}_{\bbb{M}, \mu})_{\bbb{L}, \lambda}$ are, respectively, unions of $H^\bbb{G}_{\bbb{L}, \lambda}(\zeta_m)$-blocks and $H^\bbb{G}_{\bbb{M}, \mu}(\zeta_e)$-blocks.

\item[(II)]
		The maps in \eqref{eq:chi-double-intro} induce a bijection 
		\begin{align}
		\{\sf{b} \mid \Irr(W^\bbb{G}_{\bbb{L}, \lambda})_\sf{b} \subseteq \Irr(W^\bbb{G}_{\bbb{L}, \lambda})_{\bbb{M}, \mu}\}
			\xrightarrow{\sim}
			\{\sf{c} \mid \Irr(W^\bbb{G}_{\bbb{M}, \mu})_\sf{c} \subseteq \Irr(W^\bbb{G}_{\bbb{M}, \mu})_{\bbb{L}, \lambda}\}.
		\end{align}
In particular, we get a bijection
\begin{align}
\chi_\sf{c}^\sf{b} : \Irr(W^\bbb{G}_{\bbb{L}, \lambda})_\sf{b} \xrightarrow{\sim} \Irr(W^\bbb{G}_{\bbb{M}, \mu})_\sf{c}
\end{align}
whenever the previous bijection sends $\sf{b} \mapsto \sf{c}$.

\end{enumerate}
\end{conj}

For any $e, m, (\bbb{L}, \lambda)$ as above, it is natural to set
\begin{align} 
\chi^\bbb{G}_{\bbb{L}, \lambda, \zeta_m} \vcentcolon= \sf{d}_{\cal{T}^\bbb{G}_{\bbb{L}, \lambda, m}}(\chi)
\quad\text{for all $\chi \in \Irr(W^\bbb{G}_{\bbb{L}, \lambda})$},
\end{align}
where we have implicitly identified $\chi$ with an element of $\Irr(\bb{K}H_{W^\bbb{G}_{\bbb{L}, \lambda}}(\vec{u}))$.
But we will only use specific cases of this notation, in~\Cref{sec:asf} and~\Cref{mainconj:2}.

\subsection{The Rational DAHA}\label{subsec:daha}

Before stating the next conjecture about blocks, we review background about rational double affine Hecke algebras (DAHAs).

We keep the notations $C, V, \cal{A}, \cal{C}$ from \S\ref{subsec:bmr}.
For each $H \in \cal{A}$, fix $\alpha_H \in V^\vee$ such that $H = \ker(\alpha_H)$ and $\alpha_H^\vee \in V$ such that $\langle \alpha_H, \alpha_H^\vee \rangle = 2$.
Let $\Refl = \Refl_C$ be the set of pseudo-reflections of $C$.
For any $s \in \Refl$ with $H = \ker(s - 1)$, let $\alpha_s = \alpha_H$ and $\alpha_s^\vee = \alpha_H^\vee$.
Fix a vector 
\begin{align}
\vec{\nu} = (\nu_s)_s \in \bb{C}^\Refl
\end{align}
invariant under conjugation by $C$.
We define the \dfemph{rational Cherednik algebra} or \dfemph{rational DAHA} of $C$ with \dfemph{parameter} $\vec{\nu}$ to be the $\bb{C}$-algebra
\begin{align}
D_C^\rat(\vec{\nu}) = \frac{\bb{C}C \ltimes (\Sym(V) \otimes \Sym(V^\vee)))}{\left\langle
		\sf{y}\sf{x} - \sf{x}\sf{y} - \langle \sf{x}, \sf{y}\rangle + \displaystyle\sum_{s \in \Refl} \nu_s \langle \sf{x}, \alpha_s^\vee\rangle\langle \alpha_s, \sf{y}\rangle s
	\,\middle|\,
	\sf{x} \in V,\,
	\sf{y} \in V^\vee
	\right\rangle}.
\end{align}
This is the definition in~\cite[\S{3.1}]{em}.

We write $\sf{O}_C^\rat(\vec{\nu})$ for the BGG category O of $D_C^\rat(\vec{\nu})$ defined in~\cite[\S{3.2}]{ggor}.
Its \dfemph{Verma} or \dfemph{standard} objects are indexed by $\Irr(C)$.
For all $\chi \in \Irr(C)$, we write $\Delta_{\vec{\nu}}(\chi)$ for the corresponding standard module, and $L_{\vec{\nu}}(\chi)$ for its simple quotient.

Let $\sf{eu} \in D_C^\rat(\vec{\nu})$ be the Euler element denoted $\bb{h}$ in~\cite[\S{3.6}]{em}.
It commutes with $C$, and its action on any object of $\sf{O}_C^\rat(\vec{\nu})$ is locally finite.
For any object $M$ of $\sf{O}_C^\rat(\vec{\nu})$, we define the \dfemph{graded character} of $M$ to be
\begin{align}
[M] = \sum_{\alpha\in\bb{C}} t^\alpha [M_\alpha] \in \sf{K}_0(C)[t^\bb{C}],
\end{align}
where $M_\alpha\subset M$ is the eigenspace of $\sf{eu}$ with eigenvalue $\alpha$.
For instance, $[\Delta_{\vec{\nu}}(\chi)]$ is given by $\chi \cdot \sum_i t^i [\Sym^i(V)]$, up to a certain monomial in $t$ also depending on $\chi$.

\subsection{The KZ Functor}\label{subsec:kz}

For any finite complex reflection group $C$ and $\vec{\nu} \in \bb{C}^{\Refl_C}$, Ginzburg--Guay--Opdam--Rouquier define, in the notation of \S\ref{subsec:decomp}, a specialization $\cal{T}_{\vec{\nu}} : \bb{Z}[\vec{u}_C^{\pm 1}] \to \bb{C}$ and an exact functor
\begin{align}
\sf{KZ} : \sf{O}_C^\rat(\vec{\nu}) \to \sf{Rep}(H_{C, \cal{T}_{\vec{\nu}}})
\end{align}
called the \dfemph{Knizhnik--Zamolodchikov (KZ) functor}.
Geometrically, $\sf{KZ}(M)$ is defined by localizing $M$ to $V^\reg/C$, then using~\cite{bmr} to show that the monodromy of its Knizhnik--Zamolodchikov connection factors through $H_{C, \cal{T}_{\vec{\nu}}}$:
See~\cite[\S{5}]{ggor} and~\cite[\S{6}]{em} for details.
Explicitly, $\cal{T}_{\vec{\nu}}(u_{\cal{C}, j}) = \zeta_{\cal{C}, j}$ for all $\cal{C}, j$, where
\begin{align}\label{eq:kappa-vs-zeta}
\zeta_{\cal{C}, j}
	&= {\textstyle\det_V(s_H)^{j}} e^{2\pi i\kappa_{\cal{C}, j}}
		&&\text{for any $H \in \cal{C}$}\\
	&= e^{2\pi i(\kappa_{\cal{C}, j} + j/e_\cal{C})}
		&&\text{by \eqref{eq:distinguished}},
\end{align}
and $\kappa_{\cal{C}, j}$ is defined by
\begin{align}\label{eq:kappa-vs-nu}
\kappa_{\cal{C}, j}
&=
\frac{2}{e_\cal{C}}
\sum_{k = 1}^{e_\cal{C} - 1}
	\frac{\nu_{s_H^k}(1 - \det_V(s_H)^{jk})}{1 - \det_V(s_H)^{-k}}
	&&\text{for any $H \in \cal{C}$}\\
&=
\frac{2}{e_\cal{C}}
\sum_{k = 1}^{e_\cal{C} - 1}
	\frac{\nu_{s_H^k}(1 - e^{2\pi ijk/e_\cal{C}})}{1 - e^{-2\pi ik/e_\cal{C}}}
	&&\text{by \eqref{eq:distinguished}}.
\end{align}
Above, our $e_\cal{C}$ and $\kappa_{\cal{C}, j}$ respectively correspond to $m_H$ and $b_{j, H}/m_H$ in \cite{em}.
Note that $\zeta_{\cal{C}, 0} = 1$ for all $\cal{C}$, and if $e_\cal{C} = 2$, then $\zeta_{\cal{C}, 1} = -e^{-2\pi i\nu_{s_H}}$.

By~\cite[\S{5}]{ggor}, $\sf{KZ}$ is representable by a finitely-generated projective object $P_\sf{KZ}$, and the resulting map $H_{C, \cal{T}_{\vec{\nu}}} \to \End(P_\sf{KZ})$ is an isomorphism.
Moreover, $\sf{KZ}$ is fully faithful on projective objects.
Thus the pair $(\sf{O}_C^\rat(\vec{\nu}), \sf{KZ})$ forms a \dfemph{highest weight cover} of $\sf{Rep}(H_{C, \cal{T}_{\vec{\nu}}})$ in the sense of~\cite[\S{4}]{rouquier_08}.
Via a double-centralizer phenomenon for the image of any progenerator of $\sf{O}_C^\rat(\vec{\nu})$, the KZ functor induces a bijection between the blocks of $\sf{O}_C^\rat(\vec{\nu})$ and the blocks of $\sf{Rep}(H_{C, \cal{T}_{\vec{\nu}}})$.




\subsection{}\label{subsec:s-charge}

The next conjecture relates the objects of \Cref{conj:bij} to rational DAHAs.
First, we record the following statement of when the homomorphisms $\cal{T}^\bbb{G}_{\bbb{L}, \lambda, m}$ of 
\S\ref{subsec:bij} can be expressed in the form $\cal{T}_{\vec{\nu}}$ of \S\ref{subsec:kz}.
Once more, recall that our embedding $\bb{Q}_\cyc \subseteq \bb{C}$ is given by $\zeta_m = e^{2\pi i/m}$.

\begin{prop}\label[prop]{prop:s-charge}
Let $\cal{S}$ be a $\Phi_e$-specialization of $H_C(\vec{u})$, given by parameters $m_{\cal{C}, j}$ in the sense of \eqref{eq:specialization}, and let $\cal{T} : \bb{Z}[\vec{u}_C^{\pm 1}] \to \bb{C}$ be defined by
\begin{align}
\bb{Z}[\vec{u}_C^{\pm 1}]
	\xrightarrow{\cal{S}}
		\bb{Z}_\cyc[x^{\pm 1/\infty}]
	\xrightarrow{x^{1/n} \mapsto \zeta_{mn}}
		\bb{C}
\end{align}
for some $m$.
Then a vector $\vec{\nu} \in \bb{C}^\Refl$ satisfies $\cal{T}_{\vec{\nu}} = \cal{T}$ if and only if the numbers $\kappa_{\cal{C}, j}$ defined by $\vec{\nu}$ via \eqref{eq:kappa-vs-nu} satisfy $\kappa_{\cal{C}, j} \in (\tfrac{1}{m} - \tfrac{1}{e})m_{\cal{C}, j} + \bb{Z}$ for all $\cal{C}, j$. 
In this case, $(\frac{1}{m} - \frac{1}{e})m_{\cal{C}, 0} \in \bb{Z}$ for all $\cal{C}$.
\end{prop}


\begin{rem}
For the Brou\'e--Malle specializations $\cal{S}^\bbb{G}_{\bbb{L}, \lambda}$, it turns out that we have $m_{\cal{C}, 0} = 0$ for all $\cal{C}$ and $j$.
For type $A$, see \Cref{sec:gl-gu} and the references there; for the other classical types, see~\cite[\S{2.14--2.19}]{bm_93}; and for exceptional types, see \Cref{sec:exceptional} and the references there.
\end{rem}

Given a homomorphism $\cal{T} : \bb{Z}[\vec{u}_C^{\pm 1}] \to \bb{C}$, we will say that $\vec{\nu}$ is a \dfemph{$\cal{T}$-parameter} if and only if $\cal{T}_{\vec{\nu}} = \cal{T}$.
In this case, for any block $\sf{b}$ of $\sf{Rep}(H_{C, \cal{T}})$, we write $\sf{O}_C^\rat(\vec{\nu})_\sf{b}$ to denote the corresponding block of $\sf{O}_C^\rat(\vec{\nu})$.
Now we return to the setup of \S\ref{subsec:bij}.

\begin{conj}\label[conj]{conj:cat}
In the situation of \Cref{conj:bij}(II):
\begin{enumerate}\setcounter{enumi}{2}
\item[(III)]
	Suppose that $\vec{\nu}_e$ is an $\cal{T}^\bbb{G}_{\bbb{L}, \lambda, m}$-parameter, $\vec{\nu}_m$ is an $\cal{T}^\bbb{G}_{\bbb{M}, \mu, e}$-parameter, and $\chi_{\bbb{M}, \mu}^{\bbb{L}, \lambda}(\sf{b}) = \sf{c}$.
	Then the bijection $\chi_\sf{c}^\sf{b}$ is categorified by an equivalence 
	\begin{align}
	\sf{D}^b(\sf{O}_{W^\bbb{G}_{\bbb{L}, \lambda}}^\rat(\vec{\nu}_e)_\sf{b}) \xrightarrow{\sim} \sf{D}^b(\sf{O}_{W^\bbb{G}_{\bbb{M}, \mu}}^\rat(\vec{\nu}_m)_\sf{c}).
	\end{align}
	
\end{enumerate}
\end{conj}


\section{Conjectures about Affine Springer Fibers}\label{sec:asf}

\subsection{}

In this section, we give the precise definition of the local system of (DAHA, braid-group) bimodules $\cal{E}_{\nu, \gamma}$, and the evidence for \Cref{mainconj:2} arising from twisted $\bbb{G}$ of rank $2$.

In the notation of \Cref{sec:hc}, we henceforth assume that the underlying root datum $\Gamma_\bbb{G}$ is irreducible, and that $\theta$ is an automorphism of its Dynkin diagram that represents $[f]_\bbb{G}$.
Let $\delta$ be the order of $\theta$, and let $\bbb{A} \leq \bbb{G}$ be the maximal torus defined by $[f]_\bbb{A} = \{\theta\}$.

Fix a rational number $\nu > 0$.
We will use the rational DAHA of $W^\bbb{G}_\bbb{A}$ with parameter $\vec{\nu} = (\nu_s)_s$ defined by
\begin{align}\label{eq:nu}
\nu_s = \left\{\begin{array}{ll}
	\nu
		&\text{$s$ arises from one of the longest roots}\\
		&\quad\text{in the relative root system for $(\bbb{G}, \bbb{A})$},\\
	\delta\nu
		&\text{else}.
\end{array}\right.
\end{align}
Note that if $\bbb{G}$ is split, then $\nu_s= \nu$ for all reflections $s$.
We also choose the vectors $\alpha_H$ and $\alpha_H^\vee$ of \S\ref{subsec:daha} to be the roots and coroots of the relative root system for $(\bbb{G}, \bbb{A})$.
Doing so gives $\langle \alpha_H, \alpha_H^\vee\rangle = 2$ for all hyperplanes $H$.

Altogether, our definition of the rational DAHA recovers the definition in~\cite[\S{4.2.2}]{oy}.
To abbreviate, we will write $D_{\bbb{G}, \bbb{A}}^\rat(\vec{\nu})$ and $\sf{O}_{\bbb{G}, \bbb{A}}^\rat(\vec{\nu})$ in place of $D_{W^\bbb{G}_\bbb{A}}^\rat(\vec{\nu})$ and $\sf{O}_{W^\bbb{G}_\bbb{A}}^\rat(\vec{\nu})$.

\subsection{}

We now shift our geometric setting from finite fields to the field $\bb{C}(\!(z)\!)$.

Let $G$ be a simply-connected, connected quasi-split algebraic group over $\bb{C}(\!(z)\!)$ defined by the same root datum and the same Dynkin-diagram automorphism as $\bbb{G}$.
Let $A \subseteq G$ be the maximal torus corresponding to $\bbb{A}$.
Explicitly,
\begin{align}
G = \Res_{\bb{C}(\!(z)\!)}^{\bb{C}(\!(z^{\frac{1}{\delta}})\!)}(\bar{G} \otimes \bb{C}(\!(z^{\frac{1}{\delta}})\!))^\theta
	\quad\text{and}\quad
	A = \Res_{\bb{C}(\!(z)\!)}^{\bb{C}(\!(z^{\frac{1}{\delta}})\!)}(\bar{A} \otimes \bb{C}(\!(z^{\frac{1}{\delta}})\!))^\theta,
\end{align}
where $\Res$ is Weil restriction, and $\bar{G}$ and $\bar{A}$ are the complex algebraic groups defined by the (split) root data.

\begin{rem}
Note that in~\cite[\S{2.2.2}]{oy}, the variable $G$ is the same as ours, but their $T$ is our $A$, and other notations do not match.
For instance, their $e$ is  our $\delta$.
\end{rem}

As explained in~\cite[\S{2.2.3}]{oy}, there is a smooth $\bb{C}[\![z]\!]$-group scheme $\bb{K}$ with generic fiber $G$ and connected special fiber: that is, an integral model of $G$ over $\bb{C}[\![z]\!]$.
We define $\fr{g}$ to be the Lie algebra of $\bb{K}$.
Similarly, there is an integral model of $A$ over $\bb{C}[\![z]\!]$ with connected special fiber, which we denote by $\bb{A}$.

If we fix a $\theta$-stable system of simple roots in $\Gamma_\bbb{G}$, then its orbits under $\theta$ are in biijection with the simple roots of $G$ with respect to $A$.
In particular, $(X^\vee)_\theta \simeq A(\bb{C}(\!(z)\!))/\bb{A}(\bb{C}[\![z]\!])$, where $(-)_\theta$ means the $\theta$-coinvariant quotient.
The $\bb{R}$-span of $(X^\vee)_{\theta, \free} \vcentcolon= (X^\vee)_\theta/(X^\vee)_{\theta, \tors}$ can be identified with the apartment for $A$ in the building of $G$, once we identify the origin with the point corresponding to $\bb{K}$.
The facets of the apartment cut out by the affine roots of $G$ with respect to $A$ are in bijection with the parahoric subgroups of $G$ containing $\bb{A}$.

Let $LG$ be the loop group defined by $LG(R) = G(R[\![z]\!][z^{-1}])$ for any $\bb{C}$-algebra $R$.
Let $\bb{I} \subseteq \bb{K}$ be the Iwahori subgroup defined by the simple roots of $(G, A)$ above, and let $I \subseteq LG$ be the corresponding sub-ind-group, so that $I(R) = \bb{I}(R[\![z]\!])$.
Recall that the affine flag variety of $G$ is the fpqc quotient $\Flag = LG/I$.

Write $\nu = \frac{d}{m}$ in lowest terms, with $d, m > 0$.
Let $\rho^\vee$ be the half-sum of the positive coroots of $(G, A)$.
Then there is a $\bb{G}_m$-action on $LG$ defined by 
\begin{align}
c \cdot g(z) = \Ad(c^{-2d\rho^\vee})g(c^{2m}z).
\end{align}
It descends to an action on $\Flag$, and also, the loop Lie algebra $L\fr{g} \vcentcolon= \Lie(LG)$.
If $\gamma \in L\fr{g}$ is an eigenvector for this action, then the \dfemph{affine Springer fiber}
\begin{align}
\Flag^\gamma = \{[g] \in LG/I \mid \Ad(g^{-1})\gamma \in \Lie(I)\}
\end{align}
is $\bb{G}_m$-stable.
We write $L\fr{g}_\nu$ for the weight-$2d$ eigenspace of $L\fr{g}$ under the $\bb{G}_m$-action, and $L\fr{g}_\nu^\rs \subseteq L\fr{g}_\nu$ for its open locus of generically regular semisimple elements.
Recall that if $\gamma$ is generically regular semisimple, then $\Flag^\gamma$ is finite-dimensional by~\cite[Prop.\ 1]{kl}.

Let $\Br_{\nu, \gamma}$ be the topological fundamental group of $L\fr{g}_\nu^\rs(\bb{C})$ with basepoint $\gamma$.
If $m$ is a regular number for $\bbb{G}$ in the sense of \S\ref{subsec:regular}, then~\cite[\S{3.3.6}]{oy} shows that $\Br_{\nu, \gamma}$ is the braid group of $W^\bbb{G}_\bbb{T}$ for any $\Phi_m$-split maximal torus $\bbb{T} \leq \bbb{G}$.
Henceforth, we fix such a generic torus.
In \emph{loc.\ cit.}, $W^\bbb{G}_\bbb{T}$ is called the \dfemph{little Weyl group} for $\bar{\nu} \vcentcolon= \nu \pmod{\bb{Z}}$.

\subsection{}

Henceforth, we fix $\gamma \in L\fr{g}_\nu^\rs(\bb{C})$.
Let $T_\gamma \subseteq LG$ be the centralizer of $\gamma$ in $LG$.
Let $T_{0, \gamma}$ be the intersection of $T_\gamma$ with the connected reductive group $G_0$ whose Lie algebra is the weight-$0$ eigenspace of the $\bb{G}_m$-action on $L\fr{g}$.
Then the commuting actions of $\bb{G}_m$ and $G_0$ on $LG$ descend to commuting actions of $\bb{G}_m$ and $T_{0, \gamma}$ on $\Flag^\gamma$.
Note that $T_{0, \gamma}$ corresponds to $S_a$ in~\cite{oy}.

We say that $w\theta \in W_{\Gamma_\bbb{G}}\theta$ is \dfemph{elliptic} if and only if its only fixed point on $X^\vee \otimes \bb{Q}$ is zero.
We say that it is \dfemph{regular elliptic} if and only if it is both regular and elliptic; in this case, its order is called a \dfemph{regular elliptic number} for $\bbb{G}$.
This definition then agrees with the definitions in~\cite{oy, vv}: \emph{e.g.}, by the discussion in~\cite[\S{3.2}]{oy}.
If $m$ is a regular elliptic number for $\bbb{G}$, then $\Flag^\gamma$ is a projective scheme by~\cite[Cor.\ 2]{kl}~\cite[Lem.\ 5.2.3]{oy}, and $T_{0,\gamma}$ is finite by~\cite[Lem.\ 3.3.5(3)]{oy} and the definition of ellipticity.

\subsection{}

Henceforth, we assume that $m$ is a regular elliptic number for $\bbb{G}$.
We write $\ur{H}_{\bb{G}_m}^\ast(-)$ to denote $\bb{G}_m$-equivariant singular cohomology with $\bb{C}$-coefficients.
In particular, we write $\epsilon \in \ur{H}_{\bb{G}_m}^2(\point)$ for a generator of $\ur{H}_{\bb{G}_m}^\ast(\point)$.

As we vary $\gamma$, the groups $T_{0,\gamma}$ form a finite group scheme over $L\fr{g}_\nu^\rs$, and the vector spaces $\ur{H}_{\bb{G}_m}^\ast(\Flag^\gamma)$ form a local system on which this group scheme acts.
In particular, $\Br_{\nu, \gamma}$ acts on the invariant subspace $\ur{H}_{\bb{G}_m}^\ast(\Flag^\gamma)^{T_{0,\gamma}}$.

When $\bbb{G}$ is split, meaning $\theta = 1$, Oblomkov--Yun endow $\ur{H}_{\bb{G}_m}^\ast(\Flag^\gamma)^{T_{0,\gamma}}$ with an increasing, $\Br_{\nu, \gamma}$-stable \dfemph{perverse filtration} $\sf{P}_{\leq \ast}$.
The precise construction uses an equivariant Ng\^{o}-type comparison between $\Flag^\gamma$ and a fiber of a twisted, parabolic Hitchin system over a stacky projective line; in the Hitchin setting, the filtration is defined using perverse truncation of the Hitchin sheaf complex and the Ng\^{o} support theorem.
We refer to \S{8.3.6} of \emph{ibid.}\ for details.
Moreover, Oblomkov--Yun construct a $D_{\bbb{G}, \bbb{A}}^\rat(\vec{\nu})$-action on the bigraded vector space
\begin{align}
\cal{E}_{\nu, \gamma} \vcentcolon= \gr_\ast^\sf{P} \ur{H}_{\bb{G}_m}^\ast(\Flag^\gamma)^{T_{0,\gamma}}|_{\epsilon \to 1}
\end{align}
that commutes with the $\Br_{\nu, \gamma}$-action~\cite[Thm.\ 8.2.3(1)]{oy}.
What follows are Theorem 8.2.3(2) and Conjecture 8.2.5 of \emph{ibid.}

\begin{thm}[Oblomkov--Yun]\label[thm]{thm:oy}
In the setup above, with $\nu \in \bb{Q}_{> 0}$ such that its lowest denominator is a regular elliptic number for $\bbb{G}$, and $\vec{\nu}$ defined by \eqref{eq:nu}, the $\Br_{\nu, \gamma}$-invariants of $\cal{E}_{\nu, \gamma}$ form the simple $D_{\bbb{G}, \bbb{A}}^\rat(\vec{\nu})$-module $L_{\vec{\nu}}(1)$.
\end{thm}

\begin{conj}[Oblomkov--Yun]\label[conj]{conj:oy}
The constructions above, and \Cref{thm:oy}, extend from split $\bbb{G}$ to quasi-split $\bbb{G}$, while still satisfying the further properties of $\sf{P}_{\leq \ast}$ in~\cite[Thm.\ 8.2.3(1)]{oy}.
\end{conj}

\subsection{\texorpdfstring{Evidence for Conjecture \ref{mainconj:2}}{Evidence for Conjecture 2}}\label{subsec:evidence}

In Section 9 of their paper, Oblomkov--Yun verify~\Cref{conj:oy} for various non-split cases where $\Gamma_\bbb{G}^\theta$ has rank $\leq 2$ and $\nu = \frac{1}{m}$.
As they remark, the other possibilities for $\nu$ can be reduced to this one by~\cite[Prop.\ 5.5.8]{oy}.
They arrive at the following results and conjectures:
\begin{enumerate}
\item[(OY1)]
	 	The $\Br_{\nu, \gamma}$-action on $\cal{E}_{\nu, \gamma}$ is trivial for:
		\begin{enumerate}
		\item 	The cases where $m$ is the twisted Coxeter number~\cite[Ex.\ 8.2.6]{oy}.
		
		\item 	The cases where
				\begin{align}
				(\sf{type}, m) = ({}^2A_2, 2), (C_2, 2), (G_2, 3), (G_2, 2).
				\end{align}
		\end{enumerate}
		Here, $\cal{E}_{\nu, \gamma}$ is the $D_{\bbb{G}, \bbb{A}}^\rat(\vec{\nu})$-module $L_{\vec{\nu}}(1)$ by \Cref{thm:oy}.

\item[(OY2)] 	
		The $\Br_{\nu, \gamma}$-action on $\cal{E}_{\nu, \gamma}$ can be computed for
		\begin{align}
		(\sf{type}, m) = ({}^2A_3, 2), ({}^2A_4, 2), ({}^3D_4, 6), ({}^3D_4, 3).
		\end{align}
		Here, Oblomkov--Yun state a conjecture for the $(D_{\bbb{G}, \bbb{A}}^\rat(\vec{\nu}), \Br_{\nu, \gamma})$-bimodule structure in terms of a direct sum of simple bimodules, but the existence of the perverse filtration on the whole cohomology remains open.
		
\end{enumerate}

To explain what we can verify:
let 
\begin{equation}\label{eq:bimod-1}
[\cal{E}_{\nu, \gamma}] = t^{-{\dim(\Flag^\gamma)}}\sum_{i, j} {(-1)^i} t^j \gr_j^\sf{P} \ur{H}_{\bb{G}_m}^i(\Flag^\gamma)^{T_{0,\gamma}}|_{\epsilon \to 1} \in \sf{K}_0(W^\bbb{G}_\bbb{A})[t^{\pm 1}].
\end{equation}
In the notation of~\S\ref{subsec:bij}, recall that \Cref{mainconj:2}(2) predicts that
\begin{align}\label{eq:bimodule}
[\cal{E}_{\nu, \gamma}]
	\stackrel{?}{=}
	\sum_{\rho \in \Uch(\bbb{G})_{\bbb{A}, 1} \cap \Uch(\bbb{G})_{\bbb{T}, 1}}
		\varepsilon_{\bbb{T}, 1}^\bbb{G}(\rho)
		[\Delta_{\vec{\nu}}(\chi_{\bbb{A}, 1}^\bbb{G}(\rho)) \otimes \chi_{\bbb{T}, 1, 1}^\bbb{G}(\rho)]
\end{align}
in $\sf{K}_0(\bb{C}W^\bbb{G}_\bbb{A} \otimes H^\bbb{G}_{\bbb{T}, 1}(1)^\op)[\![t]\!][t^{-1}]$, and via \Cref{thm:bmm}(2d), implies 
that
\begin{align}\label{eq:trinh-detail}
[\cal{E}_{\nu, \gamma}]
	\stackrel{?}{=}
	\sum_{\rho \in \Uch(\bbb{G})_{\bbb{A}, 1} \cap \Uch(\bbb{G})_{\bbb{T}, 1}}
		\Deg_\rho(e^{2\pi i\nu}) [\Delta_{\vec{\nu}}(\chi_{\bbb{A}, 1}^\bbb{G}(\rho))]
\end{align}
in $\sf{K}_0(\bb{C}W^\bbb{G}_\bbb{A})[\![t]\!][t^{-1}]$.
We will prove:

\begin{prop}\label[prop]{prop:evidence}
\begin{enumerate}
\item 	In every case of (OY1)--(OY2), the formulas of~\cite{oy} imply or would imply \eqref{eq:trinh-detail}.

\item 	In every case of (OY1)--(OY2), the $\Br_{\nu, \gamma}$-action factors through $H^\bbb{G}_{\bbb{T}, 1}(1)$, and the formulas of~\cite{oy} imply or would imply \eqref{eq:bimodule}.


\end{enumerate}
\end{prop}


Throughout the proof, it will be convenient to write $\bb{C}$ for the trivial $\bb{C}\Br_{\nu, \gamma}$-module.
Outside of the twisted Coxeter case (OY1a), the group $W^\bbb{G}_\bbb{A}$ will be dihedral; we will adopt Chmutova's notation for its irreducible characters~\cite{chmutova}.
To prove (1), we use her $D_{\bbb{G}, \bbb{A}}^\rat(\vec{\nu})$-module character formulas in \emph{ibid.}.
To prove (2), we use the following observation, together with geometric results of Chen--Vilonen--Xue in~\cite{cvx1, cvx2, cvx3}:

\begin{lem}\label[lem]{lem:multiplicity}
Suppose that there exist a subset $\Irr(W^\bbb{G}_\bbb{A})_\nu \subseteq \Irr(W^\bbb{G}_\bbb{A})$ and a map $(\varepsilon_\nu, D_\nu) : \Irr(W^\bbb{G}_\bbb{A})_\nu \to \{\pm 1\} \times \Irr(H^\bbb{G}_{\bbb{T}, 1}(1)) \sqcup \{0\}$ such that:
\begin{enumerate}
\item 	In $\sf{K}_0(\bb{C}W^\bbb{G}_\bbb{A} \otimes H^\bbb{G}_{\bbb{T}, 1}(1)^\op)[\![t]\!][t^{-1}]$, we have
		\begin{align}
		[\cal{E}_{\nu, \gamma}]
			= \sum_{\psi \in \Irr(W^\bbb{G}_\bbb{A})_\nu}
				\varepsilon_\nu(\psi)
				[L_{\vec{\nu}}(\psi) \otimes D_\nu(\psi)].
		\end{align}

\item 	For all $\rho \in \Uch(\bbb{G})_{\bbb{A}, 1} \cap \Uch(\bbb{G})_{\bbb{T}, 1}$ and $\psi \in \Irr(W^\bbb{G}_\bbb{A})_\nu$, we have
		\begin{align}\label{eq:multiplicity}
		(L_{\vec{\nu}}(\psi) : \Delta_{\vec{\nu}}(\chi_{\bbb{A}, 1}^\bbb{G}(\rho)))
		=
		\varepsilon_\nu(\psi)
		\varepsilon_{\bbb{T}, 1}^\bbb{G}(\rho)
		[\chi_{\bbb{T}, 1, 1}^\bbb{G}(\rho) : D_\nu(\psi)],
		\end{align}
		where the left-hand side is the virtual multiplicity of $\Delta_{\vec{\nu}}(\chi_{\bbb{A}, 1}^\bbb{G}(\rho))$ in $L_{\vec{\nu}}(\psi)$ given by the inverse decomposition matrix for $\sf{O}_{\bbb{G}, \bbb{A}}^\rat(\vec{\nu})$. 

\end{enumerate}
Then \eqref{eq:bimodule} holds.
\end{lem}

\begin{proof}
Substitute $[L_{\vec{\nu}}(\psi)] = \sum_\chi (L_{\vec{\nu}}(\psi) : \Delta_{\vec{\nu}}(\chi))[\Delta_{\vec{\nu}}(\chi)]$ into (1).
Apply (2).
\end{proof}

In practice, the sign $\varepsilon_{\bbb{T}, 1}^\bbb{G}(\rho)$ is most easily calculated as the sign of the integer $\Deg_\rho(e^{2\pi i/m})$, via \Cref{thm:bmm}(2d).
It will be convenient to write $\Deg_\chi$ and $\varepsilon_{\bbb{T}, 1}^\bbb{G}(\chi)$ in place of $\Deg_\rho$ and $\varepsilon_{\bbb{T}, 1}^\bbb{G}(\rho)$ whenever $\chi = \chi_{\bbb{A}, 1}^\bbb{G}(\rho)$.

\subsubsection{The Twisted Coxeter Case}

The $\Br_{\nu, \gamma}$-action factors through $H^\bbb{G}_{\bbb{T}, 1}(1)$ because it is trivial.
Hypothesis (1) in \Cref{lem:multiplicity} holds with $\Irr(W^\bbb{G}_\bbb{A})_\nu = \{1\}$ and $(\varepsilon_\nu, D_\nu)(1) = (1, \bb{C})$.
Writing $V$ for the reflection representation of $W^\bbb{G}_\bbb{A}$, we have
\begin{align}\label{eq:bgk-rouquier}
\Deg_\chi(e^{2\pi i/m})
	&= \left\{\begin{array}{ll}
	{(-1)^k}
		&\chi = \bigwedge^k(V),\\
	0
		&\text{else}
	\end{array}\right.\\
	&= (L_{\vec{\nu}}(1) : \Delta_{\vec{\nu}}(\chi)),
\end{align}
where the first equality follows from~\cite[Thm.\ 6.6, Rem.\ 6.9]{bgk} and the second from~\cite[Thm.\ 5.15]{rouquier_08}.
So $(L_{\vec{\nu}}(1) : \Delta_{\vec{\nu}}(\chi)) = \varepsilon_{\bbb{T}, 1}^\bbb{G}(\chi)$.
To show hypothesis (2) in \Cref{lem:multiplicity}, it remains to show that 
\begin{align}\label{eq:multiplicity-1}
[\chi_{\bbb{T}, 1, 1}^\bbb{G}(\rho) : \bb{C}] = 1
	\quad\text{for all $\rho \in \Uch(\bbb{G})_{\bbb{A}, 1} \cap \Uch(\bbb{G})_{\bbb{T}, 1}$}.
\end{align}
This follows from computing $H^\bbb{G}_{\bbb{T}, 1}(1)$ explicitly, by setting $q \to 1$ in~\cite[(7.3)]{lusztig_coxeter}.
See also~\cite[\S{6.2}]{vx}.

\subsubsection{The Case $({}^2A_2, 2)$}\label{subsubsec:oy-9-2}

Here, $W^\bbb{G}_\bbb{A} \simeq W_{A_1} \simeq S_2$, whereas $W^\bbb{G}_\bbb{T} \simeq W_{A_2} \simeq S_3$. 
Again, the $\Br_{\nu, \gamma}$-action factors through $H^\bbb{G}_{\bbb{T}, 1}(1)$ because it is trivial, and hypothesis (1) in \Cref{lem:multiplicity} holds with $\Irr(W^\bbb{G}_\bbb{A})_\nu = \{1\}$ and $(\varepsilon_\nu, D_\nu)(1) = (1, \bb{C})$.

Via the case $(\sf{type}, m) = (A_1, 2)$, we compute $[L_{\vec{\nu}}(1)] = [\Delta_\nu(1)] - [\Delta_\nu(\sf{sgn})]$.
At the same time, we see from~\cite[\S{13.8}]{carter_93} 
that the generic-degree polynomials are $\Deg_1(x) = 1$ and $\Deg_\sf{sgn}(x) = x^3$, whence $\Deg_1(-1) = 1$ and $\Deg_\sf{sgn}(-1) = -1$.
So to show hypothesis (2) in \Cref{lem:multiplicity}, we again reduce to showing \eqref{eq:multiplicity-1}.
Here, it follows from observing that $H^\bbb{G}_{\bbb{T}, 1}(1) \simeq H_{W_{S_3}}(-1)$, where $H_{W_{S_3}}(x)$ is the usual Hecke algebra for $S_3$ as a Weyl group, \emph{cf.}\ \Cref{ex:principal-hecke}.
Alternately, one can use~\cite[Table 7.2]{gj}, or more classically,~\cite{james}.

\subsubsection{The Case $(C_2, 2)$}\label{subsubsec:oy-9-3}

Here, $W^\bbb{G}_\bbb{A} \simeq W^\bbb{G}_\bbb{T} \simeq W_{BC_2}$.
The argument is analogous to that in \Cref{subsubsec:oy-9-2}, except that we compute the left-hand side of \eqref{eq:multiplicity} from~\cite[\S{3.2}]{chmutova}.
We have $H^\bbb{G}_{\bbb{T}, 1}(1) = H_{\bbb{G}^-, \bbb{T}^-, 1}(-1)\simeq \bb{C}W_{BC_2}$ which  implies \eqref{eq:multiplicity-1}.

\subsubsection{The Case $(G_2, 3)$}

Here, $W^\bbb{G}_\bbb{T} \simeq \bb{Z}_6$.
The argument is analogous to the Coxeter case, except that we compute the left-hand side of \eqref{eq:multiplicity} from~\cite[\S{3.2}]{chmutova}.
We get an analogue of \eqref{eq:bgk-rouquier}, except with $V$ replaced by the representation where $w \in W^\bbb{G}_\bbb{A}$ acts on $V$ by $w^2$. So again \eqref{eq:multiplicity-1} follows from the fact that $H^\bbb{G}_{\bbb{T}, 1}(1) = H_{\bbb{G}^-, \bbb{T}^-, 1}(-1)\simeq \bb{Q}_\cyc[\sigma]/((\sigma - 1)^3(\sigma^3 + 1))$, which we get from~\cite[(7.3)]{lusztig_coxeter} because $\bbb{T}^-$ is a Coxeter torus. 




\subsubsection{The Case $(G_2, 2)$}

Here, $W^\bbb{G}_\bbb{A} \simeq W^\bbb{G}_\bbb{T} \simeq W_{G_2}$.
The argument is analogous to that in \Cref{subsubsec:oy-9-3}.

\subsubsection{The Case $({}^2A_3, 2)$}\label{subsubsec:oy-9-4}

Here, $W^\bbb{G}_\bbb{A} \simeq W_{BC_2}$ and $W^\bbb{G}_\bbb{T} \simeq W_{A_3} \simeq S_4$.

Section 9.4 of~\cite{oy} predicts an isomorphism of $(D_{\bbb{G}, \bbb{A}}^\rat(\vec{\nu}), \bb{C}\Br_{\nu, \gamma})$-bimodules
\begin{align}
\cal{E}_{\nu, \gamma} \stackrel{?}{\simeq} L_{\vec{\nu}}(1) \otimes \bb{C} \oplus L_{\vec{\nu}}(\varepsilon_1) \otimes \ur{H}^1(C_\gamma),
\end{align}
where $\ur{H}^1(C_\gamma)$ is the simple $\bb{C}\Br_{\nu, \gamma}$-module in cohomological degree $1$ formed by the monodromy of a certain family of genus-$1$ curves.

We can compute the left-hand side of \eqref{eq:multiplicity} from~\cite[\S{3.2}]{chmutova}.
We can again determine the generic-degree polynomials, and hence the signs $\varepsilon_{\bbb{T}, 1}^\bbb{G}(\chi)$, from~\cite[\S{13.8}]{carter_93}.

The Hecke algebra $H^\bbb{G}_{\bbb{T}, 1}(1)$ has two simple modules, both in the principal block.
One is trivial.
The other arises from the $\Br_{\nu, \gamma}$-action on $\ur{H}^1(C_\gamma)$ above, by~\cite[\S{2.3}]{cvx1} combined with~\cite{cvx2}.
Thus the $\Br_{\nu, \gamma}$-action factors through $H^\bbb{G}_{\bbb{T}, 1}(1)$.
Hypothesis (1) of \Cref{lem:multiplicity} holds with $\Irr(W^\bbb{G}_\bbb{A})_\nu = \{1, \varepsilon_1\}$ and $(\varepsilon_\nu, D_\nu)(1) = (1, \bb{C})$ and $(\varepsilon_\nu, D_\nu)(\varepsilon_1) = (-1, \ur{H}^1(C_\gamma))$.
Finally, the multiplicities on the right-hand side of \eqref{eq:multiplicity} can be computed from the table labeled $D_{(2)}$ in~\cite[184]{gj}, or from~\cite{james}.
In this way, we can verify hypothesis (2) in \Cref{lem:multiplicity}.

\subsubsection{The Case $({}^2A_4, 2)$}

Here, $W^\bbb{G}_\bbb{A} \simeq W_{BC_2}$ and $W^\bbb{G}_\bbb{T} \simeq W_{A_4} \simeq S_5$.

Section 9.5 of~\cite{oy} predicts an isomorphism of bimodules
\begin{align}
\cal{E}_{\nu, \gamma} \stackrel{?}{\simeq} L_{\vec{\nu}}(1) \otimes \bb{C} \oplus L_{\vec{\nu}}(\varepsilon_2) \otimes \textstyle\bigwedge^2 \,\ur{H}^1(C_\gamma)_\ur{prim},
\end{align}
where $\bigwedge^2\, \ur{H}^1(C_\gamma)_\ur{prim}$ is the simple $\bb{C}\Br_{\nu, \gamma}$-module  in cohomological degree $2$ formed by a certain summand of the monodromy of a family of genus-$2$ curves.

Again, we compute the left-hand side of \eqref{eq:multiplicity} from~\cite[\S{3.2}]{chmutova} and the signs $\varepsilon_{\bbb{T}, 1}^\bbb{G}(\chi)$ from~\cite[\S{13.8}]{carter_93}.

The Hecke algebra $H^\bbb{G}_{\bbb{T}, 1}(1)$ again has two simple modules, both principal.
One is trivial, and the other arises from the $\Br_{\nu, \gamma}$-action on $\bigwedge^2\, \ur{H}^1(C_\gamma)_\ur{prim}$ by~\cite[(4.18)]{cvx3} combined with~\cite{cvx2}.
Thus the $\Br_{\nu, \gamma}$-action factors through $H^\bbb{G}_{\bbb{T}, 1}(1)$.
The rest of the verification of \eqref{eq:bimodule} is analogous to that in \Cref{subsubsec:oy-9-4}, except that $\varepsilon_\nu(\bigwedge^2\, \ur{H}^1(C_\gamma)_\ur{prim}) = 1$.

\subsubsection{The Case $({}^3D_4, 6)$}\label{subsubsec:oy-9-8}

Here, $W^\bbb{G}_\bbb{A} \simeq W_{G_2}$, and $W^\bbb{G}_\bbb{T} \simeq G_4$ in Shephard--Todd notation for complex reflection groups. 
Explicitly, 
\begin{align}
G_4 = \langle s, t \mid s^3 = t^3 = (st)^3 = 1\rangle,
\end{align}
so $G_4$ has the same braid group as $W_{A_2} \simeq S_3$.

Section 9.8 of~\cite{oy} predicts an isomorphism of bimodules
\begin{align}
\cal{E}_{\nu, \gamma} \stackrel{?}{\simeq} L_{\vec{\nu}}(1) \otimes \bb{C} \oplus L_{\vec{\nu}}(\varepsilon_1) \otimes M_\gamma,
\end{align}
where $M_\gamma$ is the simple $\bb{C}\Br_{\nu, \gamma}$-module in cohomological degree $0$ formed by the $2$-dimensional simple module of its quotient $\bb{C}S_3$.

We can use~\cite[\S{3.2}]{chmutova} and~\cite[\S{13.8}]{carter_93} to verify \eqref{eq:trinh-detail}.
Our parameter $\vec{\nu}$ corresponds to the parameters $(\tfrac{1}{2}, \tfrac{1}{6})$ in Chmutova's notation.
Explicitly,
\begin{align}
[L_{\vec{\nu}}(1)] + 2[L_{\vec{\nu}}(\varepsilon_1)]
	&=
		[\Delta_{\vec{\nu}}(1)]
			+ 2[\Delta_{\vec{\nu}}(\varepsilon_1)]
			+ 2[\Delta_{\vec{\nu}}(\varepsilon_2)]
			+ [\Delta_{\vec{\nu}}(\sf{sgn})]
			- 3[\Delta_{\vec{\nu}}(\tau_1)].
\end{align}

\subsubsection{The Case $({}^3D_4, 3)$}

Here, $W^\bbb{G}_\bbb{A} \simeq W_{G_2}$ and $W^\bbb{G}_\bbb{T} \simeq G_4$ once again.
Section 9.9 of~\cite{oy} predicts an isomorphism of bimodules
\begin{align}
\cal{E}_{\nu, \gamma} \stackrel{?}{\simeq} L_{\vec{\nu}}(1) \otimes \bb{C} \oplus L_{\vec{\nu}}(\varepsilon_1) \otimes \ur{H}^1(C'_\gamma),
\end{align}
where $\ur{H}^1(C'_\gamma)$ is the simple $\bb{C}\Br_{\nu, \gamma}$-module in cohomological degree $1$ formed by the monodromy of a family of genus-$1$ curves.
The rest of the verification of \eqref{eq:trinh-detail} is analogous to that in \Cref{subsubsec:oy-9-8}, but now $\vec{\nu}$ corresponds to $(1, \tfrac{1}{3})$, and $\varepsilon_\nu(\ur{H}^1(C'_\gamma)) = -1$.
Explicitly,
\begin{align}
[L_{\vec{\nu}}(1)] - 2[L_{\vec{\nu}}(\varepsilon_1)]
	&=
		[\Delta_{\vec{\nu}}(1)]
			- 2[\Delta_{\vec{\nu}}(\varepsilon_1)]
			- 2[\Delta_{\vec{\nu}}(\varepsilon_2)]
			+ [\Delta_{\vec{\nu}}(\sf{sgn})]
			- [\Delta_{\vec{\nu}}(\tau_1)]\\
	&\qquad
			+ 2[\Delta_{\vec{\nu}}(\tau_2)].
\end{align}

\subsection{Comparison to Deligne--Lusztig Bimodules}

To conclude this section, we explain the table of analogies \eqref{eq:analogy} from the introduction.

Henceforth, we return to using $G$ to denote a reductive algebraic group over $\bar{\bb{F}}_q$ arising from $\bbb{G}$.
Thus the group $G$ of our new notation will be most analogous to the group $\bar{G} \otimes \bb{C}(\!(z)\!)$ of our old notation, not the group $G$.

Fix an integer $e > 0$ and a $\Phi_e$-cuspidal pair $(\bbb{L}, \lambda)$ for $\bbb{G}$.
For any prime power $q > 1$, let $H^\bbb{G}_{\bbb{L}, \lambda}(q)$ be the base change of $H^\bbb{G}_{\bbb{L}, \lambda}(x)$ on the left-hand side of \eqref{eq:end}.
Below we state a stronger form of \Cref{conj:end}, based on~\cite[(d-V4) and Prop.~1.4]{bm_93}.
It essentially says that whenever $G, F, L, \lambda_q$ arise from $q, \bbb{G}, \bbb{L}, \lambda$, the maps $(\varepsilon_{\bbb{L}, \lambda}^\bbb{G}, \chi_{\bbb{L}, \lambda}^\bbb{G})$ of \Cref{thm:bmm} are induced by the $(G^F, H^\bbb{G}_{\bbb{L}, \lambda}(q))$-bimodule structure of $R_L^G(\lambda_q)$.


\begin{conj}[Brou\'e--Malle]\label[conj]{conj:end-enhanced}
Assume that \Cref{conj:end} holds for $q, G, F, L, \lambda_q$.
Then in the Grothendieck group $\sf{K}_0(\QL G^F \otimes H^\bbb{G}_{\bbb{L}, \lambda}(q)^\op)$, we have
\begin{align}
\sum_i {(-1)^i}\, \ur{H}_c^i(Y_{L \subseteq P}^G[\lambda_q])
	=
	\sum_{\rho \in \Uch(\bbb{G})_{\bbb{L}, \lambda}}
		\varepsilon_{\bbb{L}, \lambda}^\bbb{G}(\rho)
		[\rho_q \otimes \chi_{\bbb{L}, \lambda, q}^\bbb{G}(\rho)],
\end{align}
where $\rho_q$ corresponds to $\rho \in \Uch(\bbb{G})_{\bbb{L}, \lambda}$, and $\chi_{\bbb{L}, \lambda, q}^\bbb{G}$ corresponds to $\chi \in \Irr(W^\bbb{G}_{\bbb{L}, \lambda})$ under an isomorphism $H^\bbb{G}_{\bbb{L}, \lambda}(q) \simeq \QL W^\bbb{G}_{\bbb{L}, \lambda}$.
\end{conj}

Note that this conjecture recovers \Cref{ex:principal-generic}, because if $\bbb{A} \leq \bbb{G}$ is a $\Phi_1$-split maximal torus, then $\varepsilon_{\bbb{A}, 1}^\bbb{G} = 1$ uniformly, per \Cref{ex:principal-generic}.

\Cref{conj:end} would imply that for any further integer $m > 0$ and $\Phi_m$-cuspidal pair $(\bbb{M}, \mu)$ giving rise to $M, \mu_q$, we would have
\begin{align}\label{eq:bimodule-dl}
&\sum_{i, j} {(-1)^{i + j}}\, \ur{H}_c^i(Y_{L \subseteq P}^G[\lambda_q]) \otimes_{\QL G^F} \ur{H}_c^j(Y_{M \subseteq Q}^G[\mu_q])\\
	&\qquad=
		\sum_{\rho \in \Uch(\bbb{G})_{\bbb{L}, \lambda} \cap \Uch(\bbb{G})_{\bbb{M}, \mu}}
		\varepsilon_{\bbb{L}, \lambda}^\bbb{G}(\rho)
		\varepsilon_{\bbb{M}, \mu}^\bbb{G}(\rho)
		[\chi_{\bbb{L}, \lambda, q}^\bbb{G}(\rho) \otimes \chi_{\bbb{L}, \lambda, q}^\bbb{G}(\rho)]
\end{align}
in $\sf{K}_0(H^\bbb{G}_{\bbb{L}, \lambda}(q) \otimes H^\bbb{G}_{\bbb{M}, \mu}(q)^\op)$.
(Above, $Q \subseteq G$ is a parabolic subgroup, like $P$.)
This is the bimodule that produces the second row of \eqref{eq:analogy}.

In the notation of~\S\ref{subsec:bij}, consider the $(H^\bbb{G}_{\bbb{A}, 1}(\zeta_m), H^\bbb{G}_{\bbb{T}, 1}(1))$-bimodule
\begin{align}\label{eq:bimodule-oy}
\bar{\cal{E}}_{\nu, \gamma} = \sum_{\rho \in \Uch(\bbb{G})_{\bbb{A}, 1} \cap \Uch(\bbb{G})_{\bbb{T}, 1}}
		\varepsilon_{\bbb{T}, 1}^\bbb{G}(\rho)
		[\chi_{\bbb{A}, 1, \zeta_m}^\bbb{G}(\rho) \otimes \chi_{\bbb{M}, 1, 1}^\bbb{G}(\rho)].
\end{align}
Observe that by Theorem 6.8 of~\cite{ggor}, 
\begin{align} 
\sf{KZ}(\Delta_{\vec{\nu}}(\chi^\bbb{G}_{\bbb{A}, 1}(\rho)) = \chi^\bbb{G}_{\bbb{A}, 1, \zeta_m}(\rho).
\end{align}
So according to this theorem and \Cref{thm:ldmr}(1), we can construct $\bar{\cal{E}}_{\nu, \gamma}$ by applying the KZ functor term by term to the left factors in the conjectural decomposition of $\cal{E}_{\nu, \gamma}$ in \Cref{mainconj:2}(2).
Since $\varepsilon_{\bbb{A}, 1}^\bbb{G}(\rho) = 1$ for all $\rho$, the bimodules \eqref{eq:bimodule-dl} and \eqref{eq:bimodule-oy} only differ in the following ways:
\begin{itemize}
\item 	\eqref{eq:bimodule-dl} works for any $(\bbb{L}, \lambda)$ and $(\bbb{M}, \mu)$, whereas \eqref{eq:bimodule-oy} requires us to take $(\bbb{L}, \bbb{M}) = (\bbb{A}, \bbb{T})$.

\item 	\eqref{eq:bimodule-dl} uses the specializations of $H^\bbb{G}_{\bbb{L}, \lambda}(x)$ and $H^\bbb{G}_{\bbb{M}, \mu}(x)$ at $x = q$, whereas \eqref{eq:bimodule-oy} uses the specializations of $H^\bbb{G}_{\bbb{A}, 1}(x)$ and $H^\bbb{G}_{\bbb{T}, 1}(x)$ at $x = \zeta_m$ and $x = 1$, respectively.

\end{itemize}
This is the content of the analogy in \eqref{eq:analogy}.

\begin{rem}
Using an equivariant K\"unneth formula, it is possible to rewrite the left-hand side of \eqref{eq:bimodule-dl} in terms of the $G$-equivariant cohomology of a single derived scheme.
Here, equivariant cohomology is interpreted as the hypercohomology of the equivariant constant $\ell$-adic sheaf, as in~\cite{bdr}.

In the case where $\bbb{G}$ is split and $\bbb{L}, \bbb{M}$ are maximal tori respectively of types $[w], [v]$, the derived scheme takes the form $\cal{Y}_w \times_G^\ur{R} \cal{Y}_v$ where $\cal{Y}_w, \cal{Y}_v$ are defined as follows.
Let $\cal{B}$ be the flag variety of $G$, parametrizing its Borel subgroups.
For any $w \in W_{\Gamma_\bbb{G}}$, write $B \xrightarrow{w} B'$ to indicate that a pair of Borels $(B, B')$ has relative position $w$, and let
\begin{align}
\cal{Y}_w = \{(g, B) \in G \times \cal{B} \mid B \xrightarrow{w} gF(B)g^{-1}\}.
\end{align}
Let $G$ act on $\cal{Y}_w$ according to $x \cdot (g, B) = (xgF(x)^{-1}, xBx^{-1})$.
The arguments of~\cite[\S{2}]{bdr} show that the $G$-equivariant cohomology of $\cal{Y}_w$ recovers $R_L^G(1)$, and hence, that of $\cal{Y}_w \times_G^\ur{R} \cal{Y}_v$ recovers $R_L^G(1) \otimes_{\QL G^F} R_M^G(1)$.
Note that $\cal{Y}_w, \cal{Y}_v$ are analogues, with $GF \subseteq G \rtimes \langle F\rangle$ in place of $G$, of the varieties $Y_w$ appearing in Lusztig's work on character sheaves.
\end{rem}

\section{The General Linear and Unitary Groups}\label{sec:gl-gu}

\subsection{}

Fix an integer $n \geq 2$.
We write $\GGGL_n$ for the \dfemph{generic general linear group} of rank $n$ and $\GGGU_n$ for the \dfemph{generic general unitary group} of rank $n$, corresponding to the finite reductive groups $\ur{GL}_n(\bb{F}_q)$ and $\ur{GU}_n(\bb{F}_q)$.
These generic groups have the same root datum, but $\GGGL_n$ is split with Weyl group $S_n$, whereas $[f]_{\GGGU_n} = S_nf = -S_n$ with $f \in \Aut(\Gamma_{\GGGU_n})$ induced by the nontrivial involution of the Dynkin diagram of type $A_{n - 1}$.

In this section, we first describe the $\Phi$-cuspidal pairs of $\GGGL_n$ and $\GGGU_n$ in terms of the combinatorics of partitions.
We then describe the Brou\'e--Malle specializations $\cal{S}^\bbb{G}_{\bbb{L}, \lambda}$ for these cases, and prove \Cref{mainthm:1} for the resulting Hecke algebras.

\subsection{Partitions and Abaci}\label{subsec:partitions}

Let $\Pi$ be the set of integer partitions of arbitrary size.
We view an element $\pi \in \Pi$ as a weakly-descreasing sequence of nonnegative integers $\pi_1 \geq \pi_2 \geq \ldots$ such that $\pi_i = 0$ whenever $i$ is large enough.
Its \dfemph{size} $|\pi|$ is the sum of its entries, and its \dfemph{length} $\ell_\pi$ is its number of nonzero entries.
We refer to $e$-tuples of partitions as \dfemph{$e$-partitions}.

Recall that a partition can also be drawn as a Young diagram composed of \dfemph{nodes}.
By definition, $\pi$ is an \dfemph{$e$-core} if and only if its diagram contains no nodes whose hook length is divisible by $e$.
In the language of~\cite[\S{2.3}, {2.7}]{jk}, this means we cannot obtain a smaller Young diagram by removing rim $e$-hooks from $\pi$.

We write $\Pi_{\core{e}} \subseteq \Pi$ for the subset of $e$-cores.
For each integer $e > 0$, there is a bijection $\Pi \xrightarrow{\sim} \Pi_{\core{e}} \times \Pi^e$, called the corresponding \dfemph{core-quotient bijection}.
For our purposes, it is most convenient to express it in terms of the combinatorics of objects called abaci.
Let
\begin{align}
\sf{B} = \{\beta \subseteq \bb{Z} \mid \text{$\bb{Z}_{< x} \subseteq \beta \subseteq \bb{Z}_{< y}$ for some integers $x, y$}\}.
\end{align}
Tuples $\vec{\beta} = (\beta^{(0)}, \ldots, \beta^{(e - 1)}) \in \sf{B}^e$ are usually called \dfemph{$e$-abacus configurations}, or more simply, \dfemph{$e$-abaci}.
The elements of the sets $\beta^{(i)}$ are sometimes called \dfemph{beads}.
There is a bijection from $1$-abaci to $e$-abaci
\begin{align}
\upsilon_e = (\upsilon_e^{(0)}, \upsilon_e^{(1)}, \ldots, \upsilon_e^{(e - 1)}) : \sf{B} \xrightarrow{\sim} \sf{B}^e
\end{align}
defined by setting
\begin{align}\label{eqn:upsilon}
\upsilon_e^{(r)}(\beta) = \{q \in \bb{Z} \mid eq + r \in \beta\}
	\quad\text{for $0 \leq r < e$}.
\end{align}
We now explain how $\upsilon_e$ produces $e$-cores and $e$-quotients.

A \dfemph{charged partition} is a pair $(\pi, s) \in \Pi \times \bb{Z}$.
Henceforth, to follow convention, we will use the notation $|\pi, s\rangle$ rather than $(\pi, s)$.
There is a bijection from charged partitions to $1$-abaci 
\begin{align}
\beta : \Pi \times \bb{Z} \xrightarrow{\sim} \sf{B},
\end{align}
defined by
\begin{align}
\beta(|\pi, s\rangle) = \beta_{\pi, s} \vcentcolon= \{\pi_i - i + s \mid i = 1, 2, 3, \ldots\}.
\end{align}
More generally, a \dfemph{charged $e$-partition} is a pair $|\vec{\pi}, \vec{s}\rangle \in \Pi^e \times \bb{Z}^e$.
We again write $\beta$ to denote the bijection from charged $e$-partitions to $e$-abaci: $\beta : \Pi^e \times \bb{Z}^e \xrightarrow{\sim} \sf{B}^e$.
We then get a bijection from charged partitions to charged $e$-partitions:
\begin{align}
\Upsilon_e : \Pi \times \bb{Z} \xrightarrow{\beta} \sf{B} \xrightarrow{\upsilon_e} \sf{B}^e \xrightarrow{\beta^{-1}} \Pi^e \times \bb{Z}^e.
\end{align}
With this notation, a partition $\pi \in \Pi$ is an $e$-core if and only if, for some (\emph{equiv.}, any) $s \in \bb{Z}$, the charged $e$-partition $\Upsilon_e(|\pi, s\rangle)$ takes the form $(\emptyset^e, \vec{t})$ for some $\vec{t} \in \bb{Z}^e$, where $\emptyset^e = (\emptyset, \ldots, \emptyset)$.
That is, no bead in the output $e$-abacus can be pushed to a more negative position.

In general, if $\Upsilon_e(|\pi, s\rangle) = |\vec{\varpi}, \vec{r}\rangle$, then the charged version of the core-quotient bijection sends $|\pi, s\rangle$ to the pair consisting of:
\begin{enumerate}
\item 	The charged $e$-core $(\Upsilon_e)^{-1}(|\emptyset^e, \vec{r}\rangle)$.
		The underlying $e$-core depends only on $\pi$, so we call it the \dfemph{$e$-core of $\pi$}.
		As in~\cite[\S{2.7}]{jk}, it is the partition that remains after we remove as many rim $e$-hooks from $\pi$ as possible.

\item 	The $e$-partition $\vec{\varpi}$, which we call the \dfemph{$e$-quotient of $|\pi, s\rangle$}.

\end{enumerate}
Thus the charged core-quotient bijection is a map
\begin{align}
\Pi \times \bb{Z} \xrightarrow{\sim} (\Pi_{\core{e}} \times \bb{Z}) \times \Pi^e.
\end{align}

\subsection{}\label{subsec:wreath}

In what follows, we fix a $\Phi_1$-split maximal torus $\bbb{A} \leq \GGGL_n$.
A special feature of $\GGGL_n$ is that $\Uch(\GGGL_n) = \Uch(\GGGL_n)_{\bbb{A}, 1}$:
Every generic unipotent irreducible character belongs to the principal series.
Since $W^{\GGGL_n}_{ \bbb{A}} = S_n$, the symmetric group on $n$ letters, we may now regard $\chi_{\bbb{A}, 1}^{\GGGL_n}$ as a bijection:
\begin{align}\label{eq:partition}
 \Uch(\GGGL_n) \xrightarrow{\sim} \{\text{partitions of $n$}\}.
\end{align}
Henceforth, we write $\rho$ to denote both elements of $\Uch(\bbb{G})$ and their corresponding partitions. 
Following~\cite[45--48]{bmm}, we explain how the $\Phi$-Harish-Chandra series for $\GGGL_n$ are described by cores and quotients.

\subsubsection{}

First, a full set of representatives for the elements of $\sf{HC}_e(\GGGL_n)$ are the $\Phi_e$-cuspidal pairs $(\bbb{L}, \lambda)$ in which:
\begin{enumerate}
\item 	The Levi subgroup $\bbb{L}$ takes the form 
		\begin{align}
		\bbb{L} = \bbb{T} \times \GGGL_{n - ae}
			\quad\text{for some integer $a > 0$},
		\end{align}
		where $\bbb{T} \leq \GGGL_{ae}$ is a subtorus such that $|\bbb{T}|(x) = (x^e - 1)^a$.
		Note that in this case, $\bbb{L} = Z_\bbb{G}(\bbb{T}')$ for a $\Phi_e$-subtorus $\bbb{T}' \leq \bbb{T}$.
		
\item 	The character $\lambda \in \Uch(\bbb{L})$ corresponds to an $e$-core partition of $n - ae$ under the identification $\Uch(\bbb{L}) = \Uch(\GGGL_{n - ae})$ induced by (1).  Henceforth, we also write $\lambda$ for this $e$-core partition.

\end{enumerate}
In general, not every $\Phi_e$-split Levi subgroup of $\GGGL_n$ is $S_n$-conjugate to a Levi of the form in (1).
But for $\bbb{L}$ of that form, the $\Phi_e$-cuspidal elements of $\Uch(\bbb{L})$ are precisely those in (2).

Suppose that $(\bbb{L}, \lambda)$ is a $\Phi_e$-cuspidal pair of the form in (1)--(2).
It turns out that $[f]_\bbb{L} = S_{n - ae} v$ for some $v \in S_{ae}$ of cycle type $e^a$, where we embed $S_{ae} \times S_{n - ae}$ into $S_n$ as a parabolic subgroup.
Moreover,
\begin{align}
W^\bbb{G}_{\bbb{L}, \lambda} \simeq Z_{S_{ae}}(v) \simeq \bb{Z}_e \wr S_a,
\end{align}
where $\bb{Z}_e \vcentcolon= \bb{Z}/e\bb{Z}$.
In Shephard--Todd notation, this wreath product is denoted $G(e, 1, a)$.
Explicitly, if $c_1, \ldots, c_a$ are the individual $e$-cycles that comprise $v$, then $c_i$ generates the $i$th copy of $\bb{Z}_e$ in the wreath product.
Note that the isomorphisms above depend only on $\bbb{L}$, not on $\lambda$.

By Clifford theory, the irreducible characters of $\bb{Z}_e \wr S_a$ are indexed by the $e$-partitions $\vec{\pi} = (\pi^{(0)}, \pi^{(1)}, \ldots, \pi^{(e - 1)})$ with $\sum_i |\pi^{(i)}| = a$.
Henceforth, we conflate each element of $\Irr(W^{\GGGL_n}_{\bbb{L}, \lambda})$ with the corresponding $e$-partition.

\subsubsection{}

The Harish-Chandra series indexed by $(\bbb{L}, \lambda)$ is
\begin{align}
\Uch(\GGGL_n)_{\bbb{L}, \lambda}
	= \{\rho \in \Uch(\GGGL_n) \mid \text{$\rho$ has $e$-core $\lambda$}\}.
\end{align}
In this sense, the map sending $\rho \in \Uch(\GGGL_n)$ to the $\Phi_e$-Harish--Chandra series containing $\rho$ is essentially the map sending a partition of $n$ to its $e$-core.
At the same time, the map
\begin{align}
\chi_{\bbb{L}, \lambda}^{\GGGL_n} : \Uch(\GGGL_n)_{\bbb{L}, \lambda} \to \Irr(W^{\GGGL_n}_{\bbb{L}, \lambda})
\end{align}
is essentially the map sending a partition of $n$ to its $e$-quotient.
To make this more precise:
For all $\pi \in \Pi$ and $s \in \bb{Z}$, let $\vec{\varrho}_{e, s}(\pi)$ be the $\Pi^e$-component of $\Upsilon_e(|\pi, s\rangle)$.
Then the map $\chi_{\bbb{L}, \lambda}^{\GGGL_n}$ of~\cite[46--47]{bmm} is determined by the commutative diagram below, in which the vertical arrows are injective:
\begin{equation}\label{eq:uch-irr-partitions}
\begin{tikzpicture}[baseline=(current bounding box.center), >=stealth]
\matrix(m)[matrix of math nodes, row sep=2.5em, column sep=3.5em, text height=2ex, text depth=0.5ex]
{ 		
	\Uch(\GGGL_n)_{\bbb{L}, \lambda} 
		&\Irr(W^{\GGGL_n}_{\bbb{L}, \lambda})\\
	\Pi
		&\Pi^e\\
};
\path[->, font=\scriptsize, auto]
(m-1-1)	
	edge node{$\chi_{\bbb{L}, \lambda}^{\GGGL_n}$} (m-1-2)
	edge (m-2-1)
(m-1-2)
	edge (m-2-2)
(m-2-1)
	edge node{$\vec{\varrho}_{e, \ell_\lambda}$} (m-2-2);
\end{tikzpicture}
\end{equation}
In \emph{loc.\ cit.}, Brou\'e--Malle--Michel verify that under this definition, the maps $\chi_{\bbb{L}, \lambda}^{\GGGL_n}$ satisfy the commutativity constraint in \Cref{thm:bmm}(2a).

\subsubsection{}\label{subsubsec:ariki-koike}

In what follows, we write $C = \bb{Z}_e \wr S_a$.
For such groups, the Hecke algebra $H_C(\vec{u})$ is also known as an \dfemph{Ariki--Koike algebra}.
To describe it, recall the Coxeter presentation:
\begin{align}
C = \left\langle s_0, s_1, \ldots, s_{a - 1} \,\middle|
	\begin{array}{l}
	s_0^e = s_1^2 = \cdots = s_{a - 1}^2 = 1,\\
	(s_0s_1)^4 = (s_1s_2)^3 = \cdots = (s_{a - 2}s_{a - 1})^3 = 1
	\end{array}
	\right\rangle.
\end{align}
In the notation of \S\ref{subsec:bmr}, the $C$-orbits on the set of hyperplanes $\cal{A}$ correspond to the sets of pseudo-reflections $\{s_0\}$ and $\{s_1, \ldots, s_{a - 1}\}$.
Fix distinguished braid reflections $\tau, \sigma_1, \ldots, \sigma_{a - 1}$ that respectively lift $s_0, s_1, \ldots, s_{a - 1}$ to generators of the monodromy around these hyperplanes, such that the braid relations  in~\cite[\S{2.1}]{lm} hold with our $a, \tau, \sigma_i$ in place of their $n, T_0, T_i$.
We will write $\tau, \sigma_i$ in place of $\sigma_H$, and similarly, $u_{\tau, j}, u_{\sigma, j}$ in place of $u_{\cal{C}, j}$.
With this notation,
\begin{align}
H_C(\vec{u}) = \frac{\bb{Z}[\vec{u}^{\pm 1}][\Br_C]}{\left\langle\!\begin{array}{l}
		(\tau - u_{\tau, 0}) \cdots (\tau - u_{\tau, e - 1}),\\
		\text{$(\sigma_i - u_{\sigma, 0})(\sigma_i - u_{\sigma, 1})$ for $1 \leq i < a$}
	\end{array}\!\right\rangle}.
\end{align} 

\subsection{The General Linear Hecke Algebra}\label{subsec:gl}

When $C = W^{\GGGL_n}_{\bbb{L}, \lambda}$ for some $\Phi_e$-cuspidal pair $(\bbb{L}, \lambda)$ of $\GGGL_n$, the specialization $\cal{S}^{\GGGL_n}_{\bbb{L}, \lambda}$ in \S\ref{subsec:bm_93} is the specialization of $H_C(\vec{u})$ defined as follows.
First, let $\vec{a}_e : \Pi \to \bb{Z}^e$ be defined by
\begin{align}\label{eq:a}
\vec{a}_e(\lambda) = e\vec{b}_e(\lambda) + (0, 1, \ldots, e - 1), 
\end{align}
where $\vec{b}_e(\lambda)$ is the $\bb{Z}^e$-component of $\Upsilon_e(|\lambda, \ell_\lambda \rangle)$.
Note that $\vec{a}_e(\emptyset) = (0, 1, \ldots, e - 1)$, and that $a_e^{(0)}(\lambda) = 0$ for all $\lambda$.
In~\cite[\S{2.9--2.10}]{bm_93}, Brou\'e--Malle set
\begin{align}\label{eq:hecke-gl}
\left\{\begin{array}{r@{\:}ll}
\cal{S}^{\GGGL_n}_{\bbb{L}, \lambda}(u_{\tau, j}) 
	&= x^{a_e^{(j)}(\lambda)}
		&\text{for all $j$},\\[1ex]
\cal{S}^{\GGGL_n}_{\bbb{L}, \lambda}(u_{\sigma, 0})
	&= 1,\\[1ex]
\cal{S}^{\GGGL_n}_{\bbb{L}, \lambda}(u_{\sigma, 1})
	&= -x^e.
	\end{array}\right.
\end{align}
The definition of $\vec{a}_e$ ensures that this is, in fact, a $\Phi_e$-specialization.


\begin{prop}[Digne--Michel, Dudas]\label[prop]{prop:gl}
\Cref{conj:end} holds when $\bbb{G} = \GGGL_n$ and either of the following is true:
\begin{enumerate}
\item 	$\bbb{L}$ is a maximal torus of $\GGGL_n$.

\item 	$\bbb{L} = \bbb{T} \times \GGGL_{n - e}$ for some Coxeter maximal torus $\bbb{T} \leq \GGGL_e$.
		Equivalently, $W^{\GGGL_n}_{\bbb{L}, \lambda}$ is cyclic.

\end{enumerate}
In the notation of \S\ref{subsec:wreath}, (1) corresponds to $n - ae \in \{0, 1\}$, and (2) to $a = 1$.
\end{prop}

\begin{proof}
Case (1) is precisely case (4) of \Cref{thm:ldmr}.
Here, $\lambda$ is the trivial character, which corresponds to the empty partition, so $\vec{a}_e(\lambda) = (0, 1, \ldots, e - 1)$.
Therefore, \eqref{eq:hecke-gl} recovers the specialization in~\cite[Thm.\ 10.1]{dm_06}.

In case (2), the group $\Br_C$ is freely generated by $\tau$.
So it suffices to construct, uniformly for any prime power $q > 1$ and prime $\ell$ invertible in $\bb{F}_q$, a $\ur{GL}_n(\bb{F}_q)$-equivariant action of $\tau$ on $\ur{H}_c^\ast(Y_{L \subseteq P}^G)[\lambda_q]$ with eigenvalues $q^{a_e^{(0)}(\lambda)}, \ldots, q^{a_e^{(e - 1)}(\lambda)}$, where $L, \lambda_q$ arise from $q, L, \lambda$.
As it turns out, we can map $\tau$ to the Frobenius $F$.
Indeed, the work of Dudas in~\cite{dudas} shows that for the cuspidal pairs we are considering, the multiset of eigenvalues of $F$ is precisely $\{q^{a_e^{(i)}(\lambda)}\}_i$.
To translate his notation into ours, observe that his $d$ is our $e$, his $\mu$ is our $\lambda$, and his multiset $\{\gamma_d(X, x)\}_{x \in X'}$ is our multiset $\{q^{a_e^{(i)}(\lambda)}\}_i$.
\end{proof}

\begin{rem}
The strategy of the proof of~\cite[Thm.\ 10.1]{dm_06} is to reduce from the general case to that of Coxeter maximal tori, via a subtle interplay between the positive braids and Frobenius maps that define the relevant (braid-theoretic) Deligne--Lusztig varieties.
It is plausible that this strategy could be generalized to the \dfemph{parabolic Deligne--Lusztig varieties} of~\cite{dm_14}.
One might then reduce the case of general $\bbb{L} \leq \GGGL_n$ to case (2) of \Cref{prop:gl}.
\end{rem}

\subsection{Ennola Duality}\label{subsec:ennola}

For any generic group $\bbb{G}$, write $\bbb{G}^-$ to denote the generic group defined by $\Gamma_{\bbb{G}^-} = \Gamma_\bbb{G}$ and $[f]_{\bbb{G}^-} = [-f]_\bbb{G}$.
Note that the map $\bbb{L} \mapsto \bbb{L}^-$ sends each $\Phi_e$-split Levi subgroup of $\bbb{G}$ to a $\Phi_{e^-}$-split Levi subgroup of $\bbb{G}^-$, where 
		\begin{align}
		\Phi_{e^-}(x) \vcentcolon= \pm\Phi_e(-x),
			\quad\text{meaning $e^-$}
				&= \left\{\begin{array}{ll}
					2e
						&\text{$e$ odd},\\
					e
						&e \equiv 0 \pmod{4},\\
					\frac{1}{2}e
						&e \equiv 2 \pmod{4}.
				\end{array}\right.
		\end{align}

We will refer to the theorem below as the \dfemph{Ennola duality} between $\bbb{G}$ and $\bbb{G}^-$.
Items (1)--(2) comprise Theorem 3.3 of~\cite{bmm}.

\begin{thm}[Brou\'e--Malle--Michel]\label[thm]{thm:ennola}
For any generic group $\bbb{G}$, there is a map
\begin{align}
\rho \mapsto (\varepsilon(\rho), \rho^-) : \Uch(\bbb{G}) \to \{\pm 1\} \times \Uch(\bbb{G}^-)
\end{align}
such that:
\begin{enumerate}
\item 	$\rho \mapsto \rho^-$ is bijective.
\item 	For any integer $e > 0$ and $\Phi_e$-split Levi $\bbb{L} \leq \bbb{G}$, we have a commutative diagram:
		\begin{equation}
\begin{tikzpicture}[baseline=(current bounding box.center), >=stealth]
\matrix(m)[matrix of math nodes, row sep=3em, column sep=5.5em, text height=2ex, text depth=0.5ex]
{ 		
	\bb{Z}\Uch(\bbb{L})
		&\bb{Z}\Uch(\bbb{L}^-)\\	
	\bb{Z}\Uch(\bbb{G})
		&\bb{Z}\Uch(\bbb{G}^-)\\	
};
\path[->, font=\scriptsize, auto]
(m-1-1)		edge node{$\lambda \mapsto \varepsilon(\lambda)\lambda^-$} (m-1-2)
(m-1-1)		edge node[left]{$R_\bbb{L}^\bbb{G}$} (m-2-1)
(m-1-2)		edge node{$R_{\bbb{L}^-}^{\bbb{G}^-}$} (m-2-2)
(m-2-1)		edge node{$\rho \mapsto \varepsilon(\rho)\rho^-$} (m-2-2);
\end{tikzpicture}
\end{equation}

\item 	$\rho \mapsto \rho^-$ transports the partition of $\Uch(\bbb{G})$ into $\Phi_e$-Harish-Chandra series onto the partition of $\Uch(\bbb{G}^-)$ into $\Phi_{e^-}$-Harish-Chandra series.

\item 	For any $\Phi_e$-cuspidal pair $(\bbb{L}, \lambda)$, we have $W^{\bbb{G}^-}_{\bbb{L}^-, \lambda^-} \simeq W^\bbb{G}_{\bbb{L}, \lambda}$.
\end{enumerate}
\end{thm}

\begin{proof}[Proof of (3)--(4)]
We see that (3) follows from (2).
To prove (4):
We always have $W^{\bbb{G}^-}_{\bbb{L}^-} \simeq W^\bbb{G}_\bbb{L}$, so by~\Cref{rem:lambda-indep}, we reduce to the cases where $\bbb{G}$ is of untwisted type $D_n$ for $n$ even or of type $E_7$.
In the latter case, use~\cite[Table 1]{bmm}.
In the former case, if $\lambda$ is indexed by a degenerate symbol, then $W^\bbb{G}_{\bbb{L}, \lambda} \subseteq W^\bbb{G}_\bbb{L}$ is the unique subgroup of index two, and similarly with $\bbb{G}^-, \bbb{L}^-, \lambda^-$ in place of $\bbb{G}, \bbb{L}, \lambda$.
\end{proof}



\subsection{The General Unitary Hecke Algebra}\label{subsec:gu}

In~\cite[\S{2.11}]{bm_93}, Brou\'e--Malle predict that $H^{\GGGU_n}_{\bbb{L}^-, \lambda^-}(x)$ is related to $H^{\GGGL_n}_{\bbb{L}, \lambda}(x)$ by the substitution $x \mapsto -x$.
That is, in the notation of \S\ref{subsec:gl}, $\cal{S}^{\GGGU_n}_{\bbb{L}^-, \lambda^-}$ is defined by:
\begin{align}\label{eq:hecke-gu}
\left\{\begin{array}{r@{\:}ll}
\cal{S}^{\GGGU_n}_{\bbb{L}^-, \lambda^-}(u_{\tau, j}) 
	&= (-x)^{a_e^{(j)}(\lambda)}
		&\text{for all $j$},\\[1ex]
\cal{S}^{\GGGU_n}_{\bbb{L}^-, \lambda^-}(u_{\sigma, 0})
	&= 1,\\[1ex]
\cal{S}^{\GGGU_n}_{\bbb{L}^-, \lambda^-}(u_{\sigma, 1})
	&= -(-x)^e.
	\end{array}\right.
\end{align}

\begin{prop}[Brou\'e--Malle--Michel]\label[prop]{prop:gu}
Suppose that
\begin{align}
e^- \in \{1, 2, 2n_\odd\},
	\quad\text{where $n_\odd \vcentcolon= \left\{\begin{array}{ll}
	n - 1
		&\text{$n$ even},\\
	n 
		&\text{$n$ odd}.
\end{array}\right.$}
\end{align}
Then \Cref{conj:end} holds when $\bbb{G} = \GGGU_n$ and we replace $e, \bbb{L}, \lambda$ with $e^-, \bbb{L}^-, \lambda^-$ in the statement.
\end{prop}

\begin{proof}
If $e^- = 1$, then $\bbb{L}^- \leq \GGGU_n$ belongs to case (1) of \Cref{thm:ldmr}.
It remains to handle $e^- = 2$ and $e^- = 2n_\odd$.

If $n = 2$, then $2n_\odd = 2$.
If instead $n > 2$, then $2n_\odd$ is the twisted Coxeter number of $\GGGL_n$ by~\cite[\S{7}]{springer_74}.
Therefore, by~\cite[Cor.\ 5.10]{bmm_99}, $e^- = 2n_\odd$ implies that $\bbb{L}^-$ belongs to case (3) of \Cref{thm:ldmr}.
Similarly, $2$ is the order of $w_0f = -1$, where $w_0 \in S_n$ is the longest element and $f$ is the Dynkin-diagram automorphism defining $\GGGU_n$.
Therefore, $e^- = 2$ implies that $\bbb{L}^-$ belongs to case (2) of \Cref{thm:ldmr}.
\end{proof}


\subsection{}\label{subsec:lm}

To prove \Cref{mainthm:1}, we need the work of Lyle--Mathas describing the blocks of specialized Ariki--Koike algebras.

Recall that for any partition $\pi$ and node $\square$ in its diagram, the \dfemph{content} of $\square$ is the integer $\sf{c}(\square) = c - r$ when $\square$ occupies the $c$th column and $r$th row, in either English or French conventions.
For instance, the contents in the diagram of the trivial partition of $n$ are $0, 1, \ldots, n - 1$.
What follows is a version of~\cite[Thm.\ 2.11]{lm}.

\begin{thm}[Lyle--Mathas]\label[thm]{thm:lm}
Let $C = \bb{Z}_e \wr S_a$.
Fix a field $\bb{K} \supseteq \bb{Q}$ and units $\alpha_0, \alpha_1, \ldots, \alpha_{e - 1}, \omega \in \bb{K}^\times$, where $\omega \neq 1$.
Let $\cal{T} : \bb{Z}[\vec{u}^{\pm 1}] \to \bb{K}$ be the specialization
\begin{align}
\left\{\begin{array}{r@{\:}ll}
\cal{T}(u_{\tau, j}) 
	&= \alpha_j
		&\text{for all $j$},\\
\cal{T}(u_{\sigma, 0})
	&= 1,\\
\cal{T}(u_{\sigma, 1})
	&= -\omega,
	\end{array}\right.
\end{align}
and let $\bb{K}H_{C, \cal{T}} = \bb{K}  \otimes_{\bb{Z}[\vec{u}^{\pm 1}]} H_C(\vec{u})$ be the base change of $H_C(\vec{u})$ along $\cal{T}$.
For any $e$-partition $\vec{\varpi}$ with $\sum_i |\varpi_i| = a$, let $\sf{c}_{\vec{\varpi}}^\cal{T} : \bb{K}^\times \to \bb{Z}_{\geq 0}$ be defined by
\begin{align}
\sf{c}_{\vec{\varpi}}^\cal{T}(u) = \left|\left\{
		(j, \square) \,\middle|
		\begin{array}{l}
		0 \leq j \leq e - 1,\,
		\square \in \varpi^{(j)},\\
		u = \omega^{\sf{c}(\square)} \alpha_j
		\end{array}
	\right\}\right|.
\end{align}
Then two elements of $\Irr(C)$ map into the same block of $\sf{K}_0(\bb{K}H_{C, \cal{T}})$ if and only if they correspond to $e$-partitions $\vec{\varpi}, \vec{\varrho}$ such that $\sf{c}_{\vec{\varpi}}^\cal{T} = \sf{c}_{\vec{\varrho}}^\cal{T}$.
\end{thm}

Above, the map $\Irr(C) \to \sf{K}_0(\bb{K}H_{C, \cal{T}})$ is defined by the following replacement for the construction in \Cref{sec:blocks}.
Note that by the work of Dipper--James--Mathas, $\bb{K}H_{C, \cal{T}}$ is a cellular algebra in the sense of Graham--Lehrer~\cite[857]{lm}.
Thus there is a map $\Irr(C) \to \sf{K}_0(\bb{K}H_{C, \cal{T}})$ that sends each irreducible character of $C$ to a corresponding cell module of $\bb{K}H_{C, \cal{T}}$, called its \dfemph{Specht module}.
When $\bb{K} = \bb{Q}_\cyc$, and $\cal{T}$ is the specialization we  consider in the corollary to follow, this construction will agree with that of \Cref{sec:blocks}.

For any charged $e$-partition $|\vec{\pi}, \vec{s}\rangle$, written out as $\vec{\pi} = (\pi^{(0)}, \ldots, \pi^{(e - 1)})$ and $\vec{s} = (s^{(0)}, \ldots, s^{(e - 1)})$, let $\sf{c}_{|\vec{\pi}, \vec{s}\rangle} : \bb{Z} \to \bb{Z}_{\geq 0}$ be defined by
\begin{align}
\sf{c}_{|\vec{\pi}, \vec{s}\rangle}(k)
	= \left|\left\{
		(j, \square) \,\middle|
		\begin{array}{l}
		0 \leq j \leq e - 1,\,
		\square \in \pi^{(j)},\\
		k = e(\sf{c}(\square) + s^{(j)}) + j
		\end{array}
	\right\}\right|.
\end{align}
That is, $\sf{c}_{|\vec{\pi}, \vec{s}\rangle}$ describes the multiset of values $e(\sf{c}(\square) + s^{(j)}) + j$ as we run over indices $j$ and nodes $\square \in \pi^{(j)}$.
Let $\sf{c}_{|\vec{\pi}, \vec{s}\rangle}^{(m)} : \bb{Z}_m \to \bb{Z}_{\geq 0}$ be defined by
\begin{align}\label{eq:mod-m}
\sf{c}_{|\vec{\pi}, \vec{s}\rangle}^{(m)}(\bar{k})
	= \sum_{\bar{k} = k + m\bb{Z}} \sf{c}_{|\vec{\pi}, \vec{s}\rangle}(k).
\end{align}
Note that $\sf{c}_{|\vec{\pi}, \vec{s}\rangle}(k) = 0$ for all $k$ positive enough or negative enough, so $\sf{c}_{|\vec{\pi}, \vec{s}\rangle}^{(m)}$ is well-defined.

\begin{cor}\label[cor]{cor:content}
Let $(\bbb{L}, \lambda)$ be a $\Phi_e$-cuspidal pair for $\GGGL_n$.
Then in the terminology of \Cref{sec:blocks}, two elements of $\Irr(W^{\GGGL_n}_{\bbb{L}, \lambda})$ belong to the same $H^{\GGGL_n}_{\bbb{L}, \lambda}(\zeta_m)$-block if and only if their images under
\begin{align}
\Irr(W^{\GGGL_n}_{\bbb{L}, \lambda}) \xrightarrow{(\chi^{\GGGL_n}_{\bbb{L}, \lambda})^{-1}} \Uch(\GGGL_n)_{\bbb{L}, \lambda} \subseteq \Uch(\GGGL_n) \xrightarrow{\sim} \Pi
\end{align}
are partitions $\pi, \rho \vdash n$ such that $\sf{c}_{\Upsilon_e(|\pi,  \ell_\lambda\rangle)}^{(m)} = \sf{c}_{\Upsilon_e(|\rho,  \ell_\lambda\rangle)}^{(m)}$.
The analogous result holds with $\GGGU_n$ in place of $\GGGL_n$.
\end{cor}

\begin{proof}
By the core-quotient bijection, $\vec{b}_e(\lambda)$ is also the $\bb{Z}^e$-component of both $\Upsilon_e(|\pi, \ell_\lambda\rangle)$ and $\Upsilon_e(|\rho, \ell_\lambda\rangle)$.
Apply \Cref{thm:lm}, then the commutativity of \eqref{eq:uch-irr-partitions} and the definition of $\vec{a}_e$.
This proves the statement for $\GGGL_n$.
The minus signs by which \eqref{eq:hecke-gu} differs from \eqref{eq:hecke-gl} do not affect the proof, giving it for $\GGGU_n$.
\end{proof}

\subsection{\texorpdfstring{Proof of Theorem \ref{mainthm:1}}{Proof of Theorem 3}}

The proof will amount to manipulating generating functions.
As preparation:
For any function $\sf{f} : \bb{Z} \to \bb{Z}$ such that $\sf{f}(k) = 0$ for $k$ positive enough, let $\sf{Z}(t, \sf{f}) \in \bb{Z}[\![t^{-1}]\!][t]$ be defined by
\begin{align}
\sf{Z}(t, \sf{f}) = \sum_{k \in \bb{Z}} f(k) t^k.
\end{align}
For any bounded-above subset $\beta \subseteq \bb{Z}$, let $\sf{Z}(t, \beta) \in \bb{Z}[\![t^{-1}]\!][t]$ be defined by
\begin{align}
\sf{Z}(t, \beta) = \sf{Z}(t, \sf{1}_\beta),
\end{align}
where $\sf{1}_\beta$ is the indicator function on $\beta$.

\begin{lem}\label[lem]{lem:content}
For any charged partition $|\pi, s\rangle$, we have
\begin{align}
(1 - t^{-1})\, \sf{Z}(t, \sf{c}_{|\pi, s\rangle}) = \sf{Z}(t, \beta_{\pi, s}) - \sf{Z}(t, \bb{Z}_{< s}).
\end{align}
More generally, if $e > 0$ is an integer and $\lambda$ is the $e$-core of $\pi$, then
\begin{align}
(1 - t^{-e})\, \sf{Z}(t, \sf{c}_{\Upsilon_e(|\pi, s\rangle)}) = \sf{Z}(t, \beta_{\pi, s}) - \sf{Z}(t, \beta_{\lambda, s}).
\end{align}
\end{lem}

\begin{proof}
Write $\pi = (\pi_1, \ldots, \pi_\ell)$, where $\pi_1 \geq \pi_2 \geq \cdots \geq \pi_\ell > 0$. 
Then both sides of the first identity are equal to
\begin{align}
t^s\, \sum_{1 \leq j \leq \ell} {(t^{\pi_j - j} - t^{-j})}.
\end{align}
To prove the second identity, write 
\begin{align}
{\Upsilon_e(|\pi, s\rangle)} = ((\pi^{(0)}, \ldots, \pi^{(e - 1)}), (s^{(0)}, \ldots, s^{(e - 1)})).
\end{align}
Then we have
\begin{align}
(1 - t^{-e})\, \sf{Z}(t, \sf{c}_{\Upsilon_e(|\pi, s\rangle)})
&=	(1 - t^{-e})\, 
	\sum_{0 \leq i \leq e - 1}
		t^i \sf{Z}(t^e, \sf{c}_{\Upsilon_e(|\pi^{(i)}, s^{(i)}\rangle)})\\
&=	\sum_{0 \leq i \leq e - 1}
		t^i \left(\sf{Z}(t^e, \beta_{\pi^{(i)},s^{(i)}}) - \sf{Z}(t^e, \bb{Z}_{< s^{(i)}})\right)\\
&=
	\sf{Z}(t, \beta_{\pi, s}) - \sf{Z}(t, \beta_{\lambda, s}),
\end{align}
where the second equality uses the first claim of the lemma.
\end{proof}

\begin{prop}\label[prop]{prop:content}
Fix integers $n, e, m > 0$ and partitions $\pi, \rho$ of size $n$ that have the same $e$-core.
Consider the following statements:
\begin{enumerate}
\item 	$\sf{c}_{\Upsilon_e(|\pi, s\rangle)}^{(m)} = \sf{c}_{\Upsilon_e(|\rho, s\rangle)}^{(m)}$ for any $s \in \bb{Z}$.
\item 	$\pi, \rho$ have the same $m$-core.
\end{enumerate}
(1) implies (2).
If $m$ is coprime to $e$, then (2) implies (1).
\end{prop}

\begin{proof}
From \eqref{eq:mod-m}, we observe that (1) holds if and only if 
\begin{align}
\text{$1 - t^{-m}$ divides $\sf{Z}(t, \sf{c}_{\Upsilon_e(|\pi, s\rangle)}) - \sf{Z}(t, \sf{c}_{\Upsilon_e(|\rho, s\rangle)})$ in $\bb{Z}[\![t^{-1}]\!][t]$}.
\end{align}
From the abacus definition of the $m$-core of a partition, (2) holds if and only if 
\begin{align}
\text{$1 - t^{-m}$ divides $\sf{Z}(t, \beta_{\pi, s}) - \sf{Z}(t, \beta_{\rho, s})$ in $\bb{Z}[\![t^{-1}]\!][t]$}.
\end{align}
Finally, by \Cref{lem:content},
\begin{align}
(1 - t^{-e})(\sf{Z}(t, \sf{c}_{\Upsilon_e(|\pi, s\rangle)}) - \sf{Z}(t, \sf{c}_{\Upsilon_e(|\rho, s\rangle)}))
	= \sf{Z}(t, \beta_{\pi, s}) - \sf{Z}(t, \beta_{\rho, s}).
\end{align}
The desired claims follow.
\end{proof}

By combining \Cref{prop:content} with \Cref{cor:content}, we get \Cref{mainthm:1} for the generic groups $\GGGL_n$ and $\GGGU_n$:

\begin{cor}\label[cor]{cor:gl-gu}
Let $\bbb{G}$ be $\GGGL_n$ or $\GGGU_n$.
Let $(\bbb{L}, \lambda)$ be a $\Phi_e$-cuspidal pair and $(\bbb{M}, \mu)$ a $\Phi_m$-cuspidal pair for $\bbb{G}$.
If $e$ and $m$ are coprime, then $\Irr(W^{\GGGL_n}_{\bbb{L}, \lambda})_{\bbb{M}, \mu}$ is a single $H^{\GGGL_n}_{\bbb{L}, \lambda}(\zeta_m)$-block and $\Irr(W^{\GGGL_n}_{\bbb{M}, \mu})_{\bbb{L}, \lambda}$ is a single $H^{\GGGL_n}_{\bbb{M}, \mu}(\zeta_e)$-block.
\end{cor}

\section{Uglov's Bijections}\label{sec:fock}

\subsection{}

The goal of this section is to prove \Cref{mainthm:2}, relating the bijections in \Cref{mainthm:1} to those introduced by Uglov in~\cite{uglov}.
We first introduce the maps $\Upsilon_m^e$ and $\tilde{w}_{e, m, s}$ needed for the statement.
It will be convenient to write
\begin{align}
\bb{Z}_{[0, e)} = \{0, 1, \ldots, e - 1\}
\end{align}
throughout what follows.

\subsection{}\label{subsec:uglov}

The maps $\Upsilon_m^e$ will recover the maps $\Upsilon_m$ from \S\ref{subsec:partitions} when $e = 1$.
Just as we defined $\Upsilon_m$ in terms of a bijection $\upsilon_m : \sf{B} \xrightarrow{\sim} \sf{B}^m$, so we define $\Upsilon_m^e$ in terms of a bijection
\begin{align}
\upsilon_m^e = (\upsilon_m^{e, (0)}, \upsilon_m^{e, (1)}, \ldots, \upsilon_m^{e, (m - 1)}) : \sf{B}^e \xrightarrow{\sim} \sf{B}^m.
\end{align}
Explicitly,
\begin{align}
\begin{array}{r@{\:}c@{\:}ll}
	eq + r_e \in \upsilon_m^{e, (r_m)}(\vec{\beta})
	&\iff
	&mq + r_m \in \beta^{(r_e)}
	&\text{for all $q, r_e, r_m \in \bb{Z}$ such that }\\
	&&&\text{$0 \leq r_e < e$ and $0 \leq r_m < m$}.
\end{array}
\end{align}
We define $\Upsilon_m^e$ to be the composition
\begin{align}
\Upsilon_m^e : \Pi^e \times \bb{Z}^e \xrightarrow{\beta} \sf{B}^e \xrightarrow{\upsilon_m^e} \sf{B}^m \xrightarrow{\beta^{-1}} \Pi^m \times \bb{Z}^m.
\end{align}
In particular, $\Upsilon_m^e$ is also a bijection.
It is essentially the map $(\lambda_l, s_l) \mapsto (\lambda_n, s_n)$ constructed by Uglov on pages 273--274 in~\cite[\S{4.1}]{uglov}, except that Uglov's $(l, n)$ is our $(e, m)$, and moreover, his version of $\beta$ differs from ours by an overall shift by $1$.
For another exposition, see Appendix A in~\cite{gerber}, where $\Upsilon_m^e$ is essentially the map that Gerber would denote by $\dot{\tau} \circ \tau^{-1}$.

\subsection{}

To study the combinatorics of $\upsilon_m^e$, it is convenient to introduce the bijection
\begin{align}
(q_m, r_m) : \bb{Z} \xrightarrow{\sim} \bb{Z} \times \bb{Z}_{[0, m)}
\end{align}
defined by $x = mq_m(x) + r_m(x)$.
It is the $e = 1$ case of the bijection
\begin{align}
(q_m^e, r_m^e) : \bb{Z} \times \bb{Z}_{[0, e)} \xrightarrow{\sim} \bb{Z} \times \bb{Z}_{[0, m)}
\end{align}
defined by $q_m^e(x, y) = e q_m(x) + y$ and $r_m^e(x, y) = r_m(x)$.
With this notation,
\begin{align}
\upsilon_m^{e, (r)}(\vec{\beta}) = \{q_m^e(x, y) \mid \text{$x \in \beta^{(y)}$ such that $r_m^e(x, y) = r$}\}
\end{align}
for all $r$.

\begin{rem}\label[rem]{rem:mutually-inverse}
For any $e, m$, we have $(q_e^m, r_e^m) \circ (q_m^e, r_m^e) = \id$.
Indeed, this follows from observing that for all $(x, y) \in \bb{Z} \times \bb{Z}_{[0, e)}$, we have $q_e(e q_m(x) + y) = q_m(x)$ and $r_e(e q_m(x) + y) = y$.
\end{rem}

\begin{lem}\label[lem]{lem:q-r}
Let $(a, b) \in \bb{Z} \times \bb{Z}_{[0, e)}$ and $(c, d) \in \bb{Z} \times \bb{Z}_{[0, m)}$.
\begin{enumerate}
\item 	If $q_m^e(a, b) = c$ and $r_m^e(a, b) = d$, then $ea + mb  = mc + ed$.

\item 	If $e$ and $m$ are coprime, then the converse of (1) holds.

\end{enumerate}
\end{lem}

\begin{proof}
If $q_m^e(a, b) = c$ and $r_m^e(a, b) = d$, then
\begin{align}
mc + e d
	&=	m(e q_m(a) + b) + e r_m(a)
	=	e (m q_m(a) + r_m(a)) + mb
	=	e a + m b.
\end{align}
Conversely, if $e a + mb  = mc + e d$, then $e a \equiv e d \pmod{m}$.
If $e$ and $m$ are coprime, then $a \equiv d \pmod{m}$, whence $r_m(a) = d$ and
\begin{align}
q_m^e(a, b) &= e q_m(a) + b = e(\tfrac{1}{m}(a - d)) + b = c,
\end{align}
as claimed.
\end{proof}

\subsection{}

Henceforth, suppose that $e$ and $m$ are coprime.
We will rewrite \Cref{lem:q-r}(2) in stages.
First, it is equivalent to the statement that, for all $b \in \bb{Z}_{[0, e)}$ and $d \in \bb{Z}_{[0, m)}$ and $x \in \bb{Z}$ such that
\begin{align}\label{eq:mb-ed}
\left\{\begin{array}{r@{\:}l}
x &\equiv m b \pmod{e},\\
x &\equiv e d \pmod{m},
\end{array}\right.
\end{align}
the map $(q_m^e, r_m^e)$ sends $(\tfrac{1}{e}(x - mb), b) \mapsto (\tfrac{1}{m}(x - e d), d)$.
Now fix arbitrary integers $s$ and $t$.
Observe that
\begin{align}
\tfrac{1}{e}(x - mb) &= q_e(x + s) - q_e(mb + s),\\
\tfrac{1}{m}(x - e d) &= q_m(x + t) - q_m(e d + t),\\
r_e(mb + s) &= r_e(x + s),\\
r_m(e d + t) &= r_m(x + t).
\end{align}
Let $w_{e, m, s} : \bb{Z}_{[0, e)} \xrightarrow{\sim} \bb{Z}_{[0, e)}$ be the permutation defined by
\begin{align}
w_{e, m, s}(r_e(mb + s)) = b
	\quad\text{for all $b \in \bb{Z}_{[0, e)}$}.
\end{align}
Let $\xi_{e, m, s} : \bb{Z} \times \bb{Z}_{[0, e)} \xrightarrow{\sim} \bb{Z} \times \bb{Z}_{[0, e)}$ be the bijection defined by
\begin{align}
\xi_{e, m, s}(a, b) = (a - q_e(mb + s), b).
\end{align}
Then for all $x$ satisfying \eqref{eq:mb-ed}, we have
\begin{align}
\xi_{e, m, s} (q_e(x + s), w_{e, m, s}(r_e(x + s)))
	&= (\tfrac{1}{e}(x - mb), b),\\
\xi_{m, e, t} (q_m(x + t), w_{m, e, t}(r_m(x + t)))
	&= (\tfrac{1}{m}(x - e d), d).
\end{align}
Let $\tilde{w}_{e, m, s}$ be the composition
\begin{align}\label{eq:affine-perm}
\tilde{w}_{e, m, s} : \bb{Z} \times \bb{Z}_{[0, e)} \xrightarrow{{\id} \times w_{e, m, s}} \bb{Z} \times \bb{Z}_{[0, e)} \xrightarrow{\xi_{e, m, s}} \bb{Z} \times \bb{Z}_{[0, e)}.
\end{align}
It may be regarded as an element of $\bb{Z}^e \rtimes S_e$: that is, an affine permutation of the cocharacter lattice of $\GGGL_e$.


\begin{prop}
For coprime integers $e, m > 0$ and arbitrary integers $s, s',t,t'$ such that $s' - s = t' - t$, the following diagram commutes:
\begin{equation}
\begin{tikzpicture}[baseline=(current bounding box.center), >=stealth]
\matrix(m)[matrix of math nodes, row sep=2.5em, column sep=11em, text height=2ex, text depth=0.5ex]
{ 		
	\bb{Z} \times \bb{Z}_{[0, e)}
		&\bb{Z}
		&\bb{Z} \times \bb{Z}_{[0, m)}\\	
	\bb{Z} \times \bb{Z}_{[0, e)}
		&
		&\bb{Z} \times \bb{Z}_{[0, m)}\\
};
\path[->, font=\scriptsize, auto]
(m-1-2)		edge node[above]{$(q_e(x' + s'), r_e(x' + s')) \mapsfrom x'$} (m-1-1)
			edge node{$x' \mapsto (q_m(x' + t'), r_m(x' + t'))$} (m-1-3)
(m-1-1)		edge node[left]{$\tilde{w}_{e, m, s}$} (m-2-1)
(m-1-3)		edge node{$\tilde{w}_{m, e, t}$} (m-2-3)
(m-2-1)		edge node{$(q_m^e, r_m^e)$} (m-2-3);
\end{tikzpicture}
\end{equation}
\end{prop}

\begin{proof}
Let $x = x' + s' - s = x' + t' - t$.
Let $b \in \bb{Z}_{[0, e)}$ and $d \in \bb{Z}_{[0, m)}$ be defined in terms of $x$ using~\eqref{eq:mb-ed}.
Then the images of $x'$ in the bottom left and right corners of the diagram are $(\tfrac{1}{e}(x - mb), b)$ and $(\tfrac{1}{m}(x - ed), d)$, respectively, so the proposition follows.
\end{proof}

We again write $\tilde{w}_{e, m, s}$ and $\tilde{w}_{m, e, t}$ for the self-maps of $\sf{B}^e$ and $\sf{B}^m$ that these affine permutations respectively induce, and for the corresponding self-maps of $\Pi^e \times \bb{Z}^e$ and $\Pi^m \times \bb{Z}^m$.
Explicitly, if $\vec{\beta} = (\beta^{(0}, \ldots, \beta^{(e - 1)}) \in \sf{B}^e$, then $\tilde{w}_{e, m, s}(\vec{\beta}) = (\tilde{w}_{e, m, s}^{(0)}(\vec{\beta}), \ldots, \tilde{w}_{e, m, s}^{(e - 1)}(\vec{\beta}))$, where
\begin{align}
\tilde{w}_{e, m, s}^{(j)}(\vec{\beta})
	&=	\{ y \in \bb{Z} \mid \text{$(y, j) = \tilde{w}_{e, m, s}(x, i)$ for some $x \in \beta^{(i)}$} \} \\
	&=	\{ x - q_e(mj + s) \mid x \in \beta^{(r_e(mj + s))} \}.
\end{align}
With this notation, we arrive at:

\begin{cor}\label[cor]{cor:uglov-commute}
For any coprime integers $e, m > 0$ and arbitrary integers $s, s', t, t'$ such that $s' - s = t' - t$, the following diagram commutes:
\begin{equation}
\begin{tikzpicture}[baseline=(current bounding box.center), >=stealth]
\matrix(m)[matrix of math nodes, row sep=2.5em, column sep=6em, text height=2ex, text depth=0.5ex]
{ 		
	\sf{B}^e
		&\sf{B}
		&\sf{B}^m\\	
	\sf{B}^e
		&
		&\sf{B}^m\\
};
\path[->, font=\scriptsize, auto]
(m-1-2)		edge node[above]{$\upsilon_e(\beta + s') \mapsfrom \beta$} (m-1-1)
			edge node{$\beta \mapsto \upsilon_m(\beta + t')$} (m-1-3)
(m-1-1)		edge node[left]{$\tilde{w}_{e, m, s}$} (m-2-1)
(m-1-3)		edge node{$\tilde{w}_{m, e, t}$} (m-2-3)
(m-2-1)		edge node{$\upsilon_m^e$} (m-2-3);
\end{tikzpicture}
\end{equation}
Hence the following diagram commutes:
\begin{equation}
\begin{tikzpicture}[baseline=(current bounding box.center), >=stealth]
\matrix(m)[matrix of math nodes, row sep=2.5em, column sep=7em, text height=2ex, text depth=0.5ex]
{ 		
	\Pi^e \times \bb{Z}^e
		&\Pi
		&\Pi^m \times \bb{Z}^m\\	
	\Pi^e \times \bb{Z}^e
		&
		&\Pi^m \times \bb{Z}^m\\
};
\path[->, font=\scriptsize, auto]
(m-1-2)		edge node[above]{$\Upsilon_e(|\pi, s'\rangle) \mapsfrom \pi$} (m-1-1)
			edge node{$\pi \mapsto \Upsilon_m(|\pi, t'\rangle)$} (m-1-3)
(m-1-1)		edge node[left]{$\tilde{w}_{e, m, s}$} (m-2-1)
(m-1-3)		edge node{$\tilde{w}_{m, e, t}$} (m-2-3)
(m-2-1)		edge node{$\Upsilon_m^e$} (m-2-3);
\end{tikzpicture}
\end{equation}
\end{cor}

\subsection{\texorpdfstring{Proof of Theorem \ref{mainthm:2}}{Proof of Theorem 4}}\label{subsec:proof-of-mainthm-2}

The theorem asserts that if $\bbb{G} = \GGGL_n$ and $e, m$ are coprime, then for any $\Phi_e$-cuspidal pair $(\bbb{L}, \lambda)$ and $\Phi_m$-cuspidal pair $(\bbb{M}, \mu)$ and integers $s, t$ such that $\ell_\lambda - s = \ell_\mu - t$, we have a commutative diagram:
\begin{equation}
\begin{tikzpicture}[baseline=(current bounding box.center), >=stealth]
\matrix(m)[matrix of math nodes, row sep=2.5em, column sep=5em, text height=2ex, text depth=0.5ex]
{ 		
	\Irr(W^\bbb{G}_{\bbb{L}, \lambda})
		&\Uch(\bbb{G})_{\bbb{L}, \lambda} \cap \Uch(\bbb{G})_{\bbb{M}, \mu}
		&\Irr(W^\bbb{G}_{\bbb{M}, \mu})\\
	\Pi^e \times \bb{Z}^e
		&\Pi
		&\Pi^m \times \bb{Z}^m\\	
	\Pi^e \times \bb{Z}^e
		&
		&\Pi^m \times \bb{Z}^m\\
};
\path[->, font=\scriptsize, auto]
(m-1-2)	
	edge node[above]{$\chi_{\bbb{L}, \lambda}^\bbb{G}$} (m-1-1)
	edge node{$\chi_{\bbb{M}, \mu}^\bbb{G}$} (m-1-3)
	edge (m-2-2)
(m-1-1)
	edge (m-2-1)
(m-1-3)
	edge (m-2-3)
(m-2-2)
	edge node[above]{$\Upsilon_e^1(|\rho,  \ell_\lambda\rangle) \mapsfrom \rho$} (m-2-1)
	edge node{$\rho \mapsto \Upsilon_m^1(|\rho,  \ell_\mu\rangle)$} (m-2-3)
(m-2-1)
	edge node[left]{$\tilde{w}_{e, m, s}$} (m-3-1)
(m-2-3)
	edge node{$\tilde{w}_{m, e, t}$} (m-3-3)
(m-3-1)
	edge node{$\Upsilon_m^e$} (m-3-3);
\end{tikzpicture}
\end{equation}
The bottom rectangle commutes by \Cref{cor:uglov-commute}.
The top left and right squares commute, and are compatible with the commutative squares of \Cref{thm:bmm}(2a), due to their definition in \eqref{eq:uch-irr-partitions} and the comment that follows it.

\subsection{\texorpdfstring{Proof of Corollary \ref{maincor:1}}{Proof of Corollary 5}}

Below, we keep the setup of \S\ref{subsec:proof-of-mainthm-2}.
We also use the notation and setup of \Cref{sec:blocks}.
Thus $\zeta_m = e^{2\pi i/m}$ for all $m$.

Let $\sf{b}$, \emph{resp.}\@ $\sf{c}$, be the unique block of $\sf{Rep}(H^\bbb{G}_{\bbb{L}, \lambda}(\zeta_m))$, \emph{resp.}\@ $\sf{Rep}(H^\bbb{G}_{\bbb{M}, \mu}(\zeta_e))$, determined by $(\bbb{M}, \mu)$, \emph{resp.}\@ $(\bbb{L}, \lambda)$ under \Cref{cor:gl-gu}.
In the notation of \S\ref{subsec:s-charge}, let $\vec{\nu}_e$, \emph{resp.}\@ $\vec{\nu}_m$, be a $\cal{T}^\bbb{G}_{\bbb{L}, \lambda, m}$-parameter, \emph{resp.}\@ a $\cal{T}^\bbb{G}_{\bbb{M}, \mu, e}$-parameter.
We want to show that $\chi_\sf{c}^\sf{b}$ is categorified by an equivalence
\begin{align}
\sf{D}^b(\sf{O}_{W^\bbb{G}_{\bbb{L}, \lambda}}^\rat(\vec{\nu}_e)_\sf{b}) \xrightarrow{\sim} \sf{D}^b(\sf{O}_{W^\bbb{G}_{\bbb{M}, \mu}}^\rat(\vec{\nu}_m)_\sf{c}),
\end{align}
by relating it to a categorified form of level-rank duality.

For convenience, let $C_{e, a} = \bb{Z}_e \wr S_a$, with the convention that $C_{e, 0}$ is the trivial group.
Then $W^\bbb{G}_{\bbb{L}, \lambda} = C_{e, a}$ and $W^\bbb{G}_{\bbb{M}, \mu} = C_{m, d}$ for some $a, d$.
Let $\eta_m \in \bb{C}^\times$ be any primitive $m$th root of unity.
For any $\vec{s} \in \bb{Z}^e$, let $\cal{T}_{e, m, \vec{s}} : \bb{Z}[\vec{u}_{C_{e, a}}^{\pm 1}] \to \bb{C}$  be the ring homomorphism defined by
\begin{align}\label{eq:hecke-e-m-s}
	\left\{\begin{array}{r@{\:}ll}
		\cal{T}_{e, m, \vec{s}}(u_{\tau, j}) 
		&= \eta_m^{s^{(j)}}
		&\text{for all $j$},\\
		\cal{T}_{e, m, \vec{s}}(u_{\sigma, 0})
		&= 1,\\
		\cal{T}_{e, m, \vec{s}}(u_{\sigma, 1})
		&= -\eta_m,
	\end{array}\right.
\end{align}
using the same notation $u_{\tau, j}, u_{\sigma, 0}, u_{\sigma, 1}$ as \S\ref{subsec:gl}, and let $H_{C_{e, a}, m, \vec{s}} = H_{C_{e, a}, \cal{T}_{e, m, \vec{s}}}$.
Let $\eta_e, \cal{T}_{m, e, \vec{r}}, H_{C_{m, d}, e, \vec{r}}$ be defined similarly, with $m, e, d, \vec{r}$ in place of $e, m, a, \vec{s}$.
Note that these constructions are essentially independent of $a, d$.

The following result, a form of the Chuang--Miyachi conjecture~\cite[Conj.\ 6, Rem.\ 7]{cm_12}, can be obtained from~\cite[Thm.\ 7.4]{rsvv} (see also~\cite[\S{6}]{svv}).
Note below that $\Upsilon_m^e$ and $\Upsilon_e^m$ are inverse to each other, by~\Cref{rem:mutually-inverse}.
\begin{thm}[Level-Rank Duality]\label[thm]{thm:level-rank}
Fix integers $e, m > 0$.
Let $\vec{s} \in \bb{Z}^e$ and $\vec{r} \in \bb{Z}^m$, 
and let
\begin{align}
	\Pi_{\vec{s} }^e = \{|\vec{\varpi}, \vec{s} \rangle \mid \vec{\varpi} \in \Pi^e\}
	\quad\text{and}\quad
	\Pi_{\vec{r} }^m = \{|\vec{\varrho}, \vec{r} \rangle \mid \vec{\varrho} \in \Pi^m\}.
\end{align}
Then:
\begin{enumerate}
	\item 	
	For any $a \geq 0$, the set of $e$-partitions $\vec{\pi}$ such that $\sum_i |\pi^{(i)}| = a$ and $\Upsilon_m^e(|\vec{\pi}, \vec{s} \rangle) \in \Pi_{\vec{r}}^m$ is either empty or indexes an $H_{C_{e, a}, m, \vec{s} }$-block $\Irr(C_{e, a})_{\sf{b} } \subseteq \Irr(C_{e, a})$.
	
	\item 	
	For any $d \geq 0$, the set of $m$-partitions $\vec{\varrho}$ such that $\sum_j |\varrho^{(j)}| = d$ and $\Upsilon_e^m(|\vec{\varrho}, \vec{r} \rangle) \in  \Pi_{\vec{s} }^e$ is either empty or indexes an $H_{C_{m, d}, e, \vec{r} }$-block $\Irr(C_{m, d})_{\sf{c} } \subseteq \Irr(C_{m, d})$.

\item 
	The bijection $\Pi_{\vec{s} }^e \cap \Upsilon_e^m(\Pi_{\vec{r} }^m) \xrightarrow{\sim} \Pi_{\vec{r} }^m \cap \Upsilon_m^e(\Pi_{\vec{s} }^e)$ induced by $\Upsilon_m^e$ matches blocks, and hence, restricts to bijections of the form
	\begin{align}
	\Irr(C_{e, a})_{\sf{b} } \xrightarrow{\sim} \Irr(C_{m, d})_{\sf{c} }.
	\end{align}
	If $\vec{\nu}_e $ is a $\cal{T}_{e, m, \vec{s} }$-parameter and $\vec{\nu}_m $ is a $\cal{T}_{m, e, \vec{r} }$-parameter, then the bijection above is categorified by a Koszul-type duality (see~\Cref{foot:koszul-ringel}) between the derived categories of $\sf{O}_{C_{e, a}}^\rat(\vec{\nu}_e)_{\sf{b}}$ and $\sf{O}_{C_{m, d}}^\rat(\vec{\nu}_m)_{\sf{c}}$.
	
\end{enumerate} 
\end{thm}


Henceforth, we set $\eta_e = \zeta_e^m$ and $\eta_m = \zeta_m^e$, using the coprimality of $e$ and $m$.
Pick $\vec{s} \in \bb{Z}^e$ and $\vec{r} \in \bb{Z}^m$ such that 
\begin{align}\label{eq:starting-charges}
e\vec{s} &\equiv \vec{a}_e(\lambda) \pmod{m},\\
m\vec{r} &\equiv \vec{a}_m(\mu) \pmod{e}
\end{align}
in the notation of \S\ref{subsec:gl}.
Then $\cal{T}^\bbb{G}_{\bbb{L}, \lambda, m} = \cal{T}_{e, m, \vec{s}}$ and $\cal{T}^\bbb{G}_{\bbb{M}, \mu, e} = \cal{T}_{m, e, \vec{r}}$.
We will give $\vec{s}^\dagger$, $\vec{r}^\dagger$, $\sf{b}^\dagger$, $\sf{c}^\dagger$, $\vec{\nu}_e^\dagger$, $\vec{\nu}_m^\dagger$ that fit into the situation of~\Cref{thm:level-rank}, such that for certain $s, t \in \bb{Z}$, there are equivalences
\begin{align}\label{eq:derived-equiv}
\sf{D}^b(\sf{O}_{C_{e, a}}^\rat(\vec{\nu}_e)_\sf{b})
	&\xrightarrow{\sim} 
	\sf{D}^b(\sf{O}_{C_{e, a}}^\rat(\vec{\nu}_e^\dagger)_{\sf{b}^\dagger}),
&\sf{D}^b(\sf{O}_{C_{m, d}}^\rat(\vec{\nu}_m)_\sf{c})
	&\xrightarrow{\sim} 
	\sf{D}^b(\sf{O}_{C_{m, d}}^\rat(\vec{\nu}_m^\dagger)_{\sf{c}^\dagger}),
\end{align}
categorifying the affine permutations $\tilde{w}_{e, m, s}, \tilde{w}_{m, e, t}$ in \S\ref{subsec:proof-of-mainthm-2}.


Let the affine symmetric group $\bb{Z}^e \rtimes S_e$ act on the set of $C_{e, a}$-invariant vectors in $\bb{C}^{\Refl(C_{e, a})}$ as follows.
First, such a vector $\vec{\nu}$ is determined by the values $\kappa_{\cal{C}, j}$ in \eqref{eq:kappa-vs-nu}.
These, in turn, are determined by the values $\kappa_{\tau, j}$ for $0 \leq j < e$ and $\kappa_{\sigma, 1}$, where the labeling is analogous to that of \S\ref{subsubsec:ariki-koike}.
The action of $\bb{Z}^e \rtimes S_e$ on the vector $\vec{\nu}$ arises from the action of $S_e$, \emph{resp.}\@ $\bb{Z}^e$, on the vector $(\kappa_{\tau, j} + \tfrac{j}{e})_{j = 0}^{e - 1} \in \bb{C}^e$ by permutation, \emph{resp.}\@ translation.

The following result, conjectured by Rouquier~\cite[Conj.\@ 5.6]{rouquier_08}, can be obtained from either~\cite[\S{5.1.3--5.1.4}]{gl} or part (3) of~\cite[Thm.\@ B]{webster}.
Note that our $e$, $\kappa_{\sigma, 1}$, and $\kappa_{\tau, j}$ correspond to the $\ell$, $\kappa$, and $\kappa s_j - j/\ell$ in~\cite[\S{2.3.2}]{gl}.

\begin{thm}[Gordon--Losev, Webster]\label[thm]{thm:gl-webster}
For any $C_{e, a}$-invariant vector $\vec{\nu} \in \bb{C}^{\Refl(C_{e, a})}$ and $\tilde{w} \in \bb{Z}^e \rtimes S_e$, there is an equivalence
\begin{align}\label{eq:webster}
	\sf{D}^b(\sf{O}_{C_{e, a}}^\rat(\vec{\nu}))
	\xrightarrow{\sim}
	\sf{D}^b(\sf{O}_{C_{e, a}}^\rat(\tilde{w} \cdot \vec{\nu})).
\end{align}
\end{thm}

We will only need the case where $\tilde{w}$ is a finite permutation.
Below, recall $\vec{b}_e(\lambda)$ from  \S\ref{subsec:gl}.
The facts that let us combine~\Cref{mainthm:2} with~\Cref{thm:gl-webster} are:

\begin{lem}\label[lem]{lem:dilation}
\begin{enumerate}
\item 
	If $\vec{\nu}$ is a $\cal{T}_{e, m, \vec{s}}$-parameter and $\vec{s}^\dagger \equiv w(\vec{s}) \pmod{m}$ for some $w \in S_e$, then $w \cdot \vec{\nu}$ is a $\cal{T}_{e, m, \vec{s}^\dagger}$-parameter.

\item
	If $\vec{s}$ satisfies \eqref{eq:starting-charges}, $\vec{s}^\dagger = \tilde{w}_{e, m, s}(\vec{b}_e(\lambda))$, and $s$ is a multiple of $m$, then $\vec{s}^\dagger \equiv  w_{e, m, s}(\vec{s}) \pmod{m}$.
	
\end{enumerate}
\end{lem}

\begin{proof}
(1) follows from comparing \eqref{eq:kappa-vs-zeta} and \eqref{eq:hecke-e-m-s}.
To show (2), compute
\begin{align} 
e\tilde{w}_{e, m, s}(\vec{b}_e(\lambda))^{(i)}
	&= eb_e^{(r_e(mi + s))}(\lambda) - eq_e(mi + s)\\
	&= eb_e^{(r_e(mi + s))}(\lambda) - (mi + s) + r_e(mi + s) \\
	&\equiv eb_e^{(r_e(mi + s))}(\lambda) + r_e(mi + s) &&\pmod{m}.
\end{align}
Then recall that $eb_e^{(r_e(mi + s))}(\lambda) + r_e(mi + s) = w_{e, m, s}(\vec{a}_e(\lambda))^{(i)}$.
\end{proof}

Since $e, m$ are coprime, we can pick $x, y \in \bb{Z}$ such that $\ell_\lambda - mx = \ell_\mu - ey$.
Let $\vec{s}^\dagger = \tilde{w}_{e, m, mx}(\vec{b}_e(\lambda))$ and $\vec{r}^\dagger = \tilde{w}_{m, e, ey}(\vec{b}_m(\mu))$.
Let $\vec{\nu}_e^\dagger = w_{e, m, mx} \cdot \vec{\nu}_e$ and $\vec{\nu}_m^\dagger = w_{m, e, ey} \cdot \vec{\nu}_m$.

By~\Cref{lem:dilation}, $\vec{\nu}_e^\dagger$ is a $\cal{T}_{e, m, \vec{s}^\dagger}$-parameter and $\vec{\nu}_m^\dagger$ is a $\cal{T}_{m, e, \vec{r}^\dagger}$-parameter.
So, for some blocks $\sf{b}^\dagger \subseteq \sf{Rep}(H_{C_{e, a}, m, \vec{s}^\dagger})$ and $\sf{c}^\dagger \subseteq \sf{Rep}(H_{C_{m, d}, e, \vec{r}^\dagger})$, \Cref{thm:gl-webster} produces the equivalences in~\eqref{eq:derived-equiv}.
Finally, the core-quotient bijection shows that $\vec{b}_e(\lambda)$ is the $\bb{Z}^e$-component of $\Upsilon_e^1(|\rho, \ell_\lambda\rangle)$ for any $\rho \in \Uch(\bbb{G})_{\bbb{L}, \lambda}$, and analogously with $m, \bbb{M}, \mu$ in place of $e, \bbb{L}, \lambda$.
So the commutativity of the diagram in~\Cref{mainthm:2} shows that $\sf{b}^\dagger$ and $\sf{c}^\dagger$ are matched under $\Upsilon_m^e$ in the sense of~\Cref{thm:level-rank}.

\begin{rem}Note that the sets in parts (1)--(2) of~\Cref{thm:level-rank} are empty unless $\sum_i s^{(i)} = \sum_j r^{(j)}$. We can confirm this identity directly for the vectors $\vec{s}^\dagger, \vec{r}^\dagger$ in the proof above.
Indeed, writing $s = mx$ and $t = ey$, we have
\begin{align}
\sum_i s^{\dagger, (i)}
	&= \sum_i b_e^{(i)}(\lambda)
	- \sum_i q_e(mi + s) = \ell_\lambda - \left(s + \tfrac{(e - 1)(m - 1)}{2} \right),\\
\sum_j r^{\dagger, (j)}
	&= \sum_j b_m^{(j)}(\mu)
	- \sum_{j}q_m(ej + t) = \ell_\mu - \left(t + \tfrac{(m - 1)(e - 1)}{2} \right).
\end{align}
where we have used $\sum_i r_e(mi + s) = \frac{1}{2} e(e - 1)$ and $\sum_j r_m(ej + t) = \frac{1}{2} m(m - 1)$.
Since $\ell_\lambda - s = \ell_\mu - t$, we get the claimed identity.
\end{rem}

\section{The Exceptional Groups}\label{sec:exceptional}

\subsection{}

To conclude the paper, we explain how to check the numerical predictions of \Cref{mainconj:1} when $\Gamma_\bbb{G}$ is irreducible of exceptional type.

\subsection{\texorpdfstring{$\Phi$-Cuspidal Pairs}{Φ-Cuspidal Pairs}}\label{subsec:data-hc}

By \S\ref{subsec:singular}, it suffices to assume that $\bbb{G}$ is adjoint, $e$ is a singular number, and $(\bbb{L}, \lambda)$ is a $\Phi_e$-cuspidal pair with $\bbb{L} \neq \bbb{G}$.

Below we list the singular numbers for each exceptional type, which can be deduced from the table in~\cite[\S{2.9}]{carter_93}.
For the term \dfemph{twisted Coxeter number}, see \S\ref{subsec:regular}.
It is the usual Coxeter number except for the twisted types ${}^3D_4$ and ${}^2E_6$, where it is given by~\cite[Table 10]{springer_74}.
\begin{align}
\begin{array}{rll}
\textsf{type}
	&\textsf{singular numbers}
	&\textsf{twisted Coxeter number}\\
\hline
G_2
	&1, 2, 3
	&6\\
{}^3D_4
	&1, 2, 3, 6
	&12\\
F_4
	&1, 2, 3, 4, 6, 8
	&12\\
E_6
	&1, 2, 3, 4, \myboxed{5}, 6, 8, 9
	&12\\
{}^2E_6
	&1, 2, 3, 4, 6, 8, \myboxed{10}, 12
	&18\\
E_7
	&1, 2, 3, \myboxed{4}, \myboxed{5}, 6, 7, \myboxed{8}, 9, \myboxed{10}, \myboxed{12}, 14
	&18\\
E_8
	&1, 2, 3, 4, 5, 6, \myboxed{7}, 8, \myboxed{9}, 10, 12, \myboxed{14}, 15, \myboxed{18}, 20, 24
	&30
\end{array}
\end{align}
If $e = 1$, then the options for $[\bbb{L}, \lambda] \in \sf{HC}_e(\bbb{G})$ are given by~\cite[Table II]{lusztig_78}.
One can work out the associated relative Weyl group and $\Phi_e$-Harish-Chandra series from the same data:
See \S{3.25}--{3.26} of \emph{ibid.}
If $[\bbb{L}, \lambda] = [\bbb{T}, 1]$ for some $\Phi_e$-split maximal torus $\bbb{T} \leq \bbb{G}$, then by~\Cref{prop:regular}, $e$ is a regular number, so the associated relative Weyl group is given by~\cite[Table 3]{bmm} in the non-cyclic cases.
Here, the $\Phi_e$-Harish-Chandra series is implicitly worked out in~\cite[Ch.\@ 10--11]{lusztig_84}.
It can also be worked out from~\Cref{thm:bmm}(2d).

All other options for $[\bbb{L}, \lambda]$, along with their relative Weyl groups, are listed in Table 1 of~\cite{bmm}.
The associated $\Phi_e$-Harish-Chandra series are then given by their Table 2.


Note that the tables of~\cite{bmm} use Shephard--Todd notation for the relative Weyl groups.
We will also use this notation in what follows.

\subsection{Specializations}\label{subsec:data-s}




The specializations $\cal{S}^\bbb{G}_{\bbb{L}, \lambda}$ that we need are listed in~\cite[\S{5}]{bm_93} and~\cite[Table 9]{malle}.

Note that in~\cite{bm_93}, the data is organized by the isomorphism class of the group $W^\bbb{G}_{\bbb{L}, \lambda}$, rather than by $[\bbb{L}, \lambda]$, and different notation is used to label the complex reflection groups.
Thus, in the tables at the end of this section, we have reorganized the data for the convenience of the reader.
For each $e$, we list the set of options for $[\bbb{L}, \lambda]$, omitting those where $\bbb{L} = \bbb{G}$.
For each $[\bbb{L}, \lambda]$, we state the isomorphism class of $W^\bbb{G}_{\bbb{L}, \lambda}$ in Shephard--Todd notation, the parameters that define $\cal{S}^\bbb{G}_{\bbb{L}, \lambda}$, and the reference where we found this data. 

Our tables omit the cases that can be inferred from Ennola duality as follows.
Recall from~\Cref{thm:ennola}(4) that $W^{\bbb{G}^-}_{\bbb{L}^-, \lambda^-} \simeq W^\bbb{G}_{\bbb{L}, \lambda}$.
Via this isomorphism, Brou\'e--Malle predict that $\cal{S}^{\bbb{G}^-}_{\bbb{L}^-, \lambda^-}$ is related to $\cal{S}^\bbb{G}_{\bbb{L}, \lambda}$ by the substitution $x \mapsto -x$.
Compare to \S\ref{subsec:gu}.

The references are abbreviated as follows:
\begin{itemize}
	\item 	\cite{lusztig_coxeter} refers to (7.3) of \emph{ibid.}
	
	\item 	\cite{lusztig_78} refers to Table II of \emph{ibid.}
			Note that Lusztig uses Iwahori's sign conventions, \emph{cf.}\ \Cref{ex:principal-hecke}.
			In our tables, we have flipped the signs.
	
	\item 	``BM, \_'' refers to~\cite[\_]{bm_93}.
	
	\item 	In type $E_6$, ``\cite{lusztig_78}${}_\ur{E}$'' means that we use Ennola to derive the prediction for $\cal{S}^\bbb{G}_{\bbb{L}, \lambda}$ from the corresponding entry for ${}^2E_6$ in~\cite[Table II]{lusztig_78}.
	
	\item 	In type $E_8$, ``BM/M'' means that the given entry or its Ennola dual can be derived from~\cite[180]{bm_93} and~\cite[Table 9]{malle}.
	
\end{itemize}
We have fixed an apparent typo in~\cite[180]{bm_93} for $\bbb{G}$ of type $E_7$ and $e = 3, 6$.
We thank Malle and Michel for supplying the cases where $\bbb{G}$ is of type $E_7$ and $e = 4$ and $W^\bbb{G}_{\bbb{L}, \lambda} = G(4, 1, 2)$, which were omitted from~\cite{bm_93, malle}:
See, instead,~\cite[Lem.\ 4.22--4.23]{bmm_14}.

\subsection{Blocks}\label{subsec:data-blocks}

In the notation of \Cref{sec:blocks}, the $H^\bbb{G}_{\bbb{L}, \lambda}(\zeta_m)$-blocks of $\Irr(W^\bbb{G}_{\bbb{L}, \lambda})$ can be determined from:
\begin{itemize}
\item 	\cite[Thm.\ 2.11]{lm} when $W^\bbb{G}_{\bbb{L}, \lambda} = G(e, 1, a) \vcentcolon= \bb{Z}_e \wr S_a$, as in \Cref{sec:gl-gu}.

\item 	\cite[Appendix F]{gp} and~\cite[\S{7.2}]{gj} when $W^\bbb{G}_{\bbb{L}, \lambda}$ is an irreducible Weyl group of split exceptional type.

\item 	\cite[\S{3}]{cm} when $W^\bbb{G}_{\bbb{L}, \lambda} = G_4, G_5, G_8, G_9, G_{10}, G_{16}$.
		See also \S{1.3} in \emph{ibid.}

\end{itemize}

\subsection{What We Can Verify}

For each pair of singular numbers $e \neq m$, and for all  $[\bbb{L}, \lambda] \in \sf{HC}_e(\bbb{G})$ and $[\bbb{M}, \mu] \in \sf{HC}_m(\bbb{G})$, the information in~\Crefrange{subsec:data-hc}{subsec:data-blocks} makes it possible to compare the sizes of: the sets $\Uch(\bbb{G})_{\bbb{L}, \lambda} \cap \Uch(\bbb{G})_{\bbb{M}, \mu}$, the $H^\bbb{G}_{\bbb{L}, \lambda}(\zeta_m)$-blocks of $\Irr(W^\bbb{G}_{\bbb{L}, \lambda})$, and the $H^\bbb{G}_{\bbb{M}, \mu}(\zeta_e)$-blocks of $\Irr(W^\bbb{G}_{\bbb{M}, \mu})$.

\begin{prop}\label[prop]{prop:exceptional}
In each case,
these sizes are consistent with the existence of the partitions and bijections predicted by \Cref{conj:bij}.
\end{prop}

In the cases where the references listed in \S\ref{subsec:data-blocks} give the decomposition numbers of the blocks of the Hecke algebras, we can even verify identity \eqref{eq:multiplicity}.
We will return to this point in a sequel.

Below, we give the complete verification for types $G_2$ and $F_4$.
First, we list all options for $e$ and $(\bbb{L}, \lambda)$, omitting Ennola-dual cases and cases where $\bbb{L} = \bbb{G}$, per \S\ref{subsec:data-s}.
We also omit $\lambda$ when it is uniquely determined by $\bbb{L}$: for instance, when $\bbb{L}$ is a torus.
Then, for each pair $(e, m)$ up to reordering, and choice of $[\bbb{L}, \lambda] \in \sf{HC}_e(\bbb{G})$ and $[\bbb{M}, \mu]  \in \sf{HC}_m(\bbb{G})$, we give
\begin{align}\tag{$\ast$}
\begin{array}{c}
\text{a partition of $|\Uch(\bbb{G})_{\bbb{L}, \lambda} \cap \Uch(\bbb{G})_{\bbb{M}, \mu}|$ that describes the sizes of both}\\
\text{the $H^\bbb{G}_{\bbb{L}, \lambda}(\zeta_m)$-blocks of $\Irr(W^\bbb{G}_{\bbb{L}, \lambda})$ and the $H^\bbb{G}_{\bbb{M}, \mu}(\zeta_e)$-blocks of $\Irr(W^\bbb{G}_{\bbb{M}, \mu})$}.
\end{array}
\end{align} 
For the other types, we have prepared similar lists and tables using~\cite{m15},
available upon request.



\subsection{\texorpdfstring{Type $G_2$}{Type G2}}

Here, the only relevant Levis are maximal tori,
so we just list $|\bbb{L}|$.

\begin{longtable}{p{1em}p{1.5em}p{2.5em}p{9.5em}p{4em}}
$e$&$\bbb{L}$&$W^\bbb{G}_{\bbb{L}, \lambda}$&$\cal{S}_{\bbb{G},\bbb{L},1}(u_{\cal{C}, j})$&\textsf{reference}\\
\hline
$1$&$\Phi_1^2$&$W_{G_2}$&$1,-x;1,-x$&Ex.\ \ref{ex:principal-hecke}\\
$6$&$\Phi_6$&$\bb{Z}_6$&$1,\pm x,x^2,\zeta_3x,\zeta_3^{2}x$&\cite{lusztig_coxeter}
\end{longtable}

(The cases $e = 2, 3$ are Ennola-dual to $e = 1, 6$.)
The results for ($\ast$) are:
\begin{longtable}{lll}
$(e, m)$
	&$|\Uch(\bbb{G})_{\bbb{L}, \lambda} \cap \Uch(\bbb{G})_{\bbb{M}, \mu}|$
	&($\ast$)\\
\hline
$(1, 2)$
	&$4$
	&$4$\\
$(1, 3), (1, 6), (2, 3), (2, 6)$
	&$3$
	&$3$\\
$(3, 6)$
	&$4$
	&$2 + 2$
\end{longtable}





\subsection{\texorpdfstring{Type $F_4$}{Type F4}}\label{subsec:f4}

Below, the notation $\Phi_e^a.B_2$ means we extend a generic almost-simple group of type $B_2$ by a torus of order $\Phi_e^a$.
The notations $\phi_{1^2, -}$, $\phi_{-, 2}$ are from~\cite[\S{13.8}]{carter_93}.
\begin{longtable}{p{1em}p{3em}p{2.5em}p{3em}p{16.5em}p{4em}}
$e$&$\bbb{L}$&$\lambda$&$W^\bbb{G}_{\bbb{L}, \lambda}$&$\cal{S}^\bbb{G}_{\bbb{L}, \lambda}(u_{\cal{C}, j})$&\textsf{reference}\\
\hline
$1$&$\Phi_1^4$&&$W_{F_4}$&$1,-x;1,-x$&Ex.\ \ref{ex:principal-hecke}\\
$1$&$\Phi_1^2.B_2$&&$W_{BC_2}$&$1,-x^3;1,-x^3$&\cite{lusztig_78}\\
$3$&$\Phi_3^2$&&$G_5$&$1,x,x^2$&BM, 5.9\\
$4$&$\Phi_4^2$&&$G_8$&$1,\pm x,x^2$&BM, 5.12\\
$4$&$\Phi_4.B_2$&$\phi_{1^2,-}$&$\bb{Z}_4$&$1,\pm x^3,x^6$&BM, 8.1\\
$4$&$\Phi_4.B_2$&$\phi_{-,2}$&$\bb{Z}_4$&$1,\pm x^3,x^6$&BM, 8.1\\
$8$&$\Phi_8$&&$\bb{Z}_8$&$1,x^2,\pm x^3,\pm\zeta_4x^3,x^4,x^6$&BM, 8.1\\
$12$&$\Phi_{12}$&&$\bb{Z}_{12}$&$1,\pm x,\pm x^2,\pm x^3,x^4,\pm \zeta_4x^2,\zeta_3 x^2,\zeta_3^2 x^2$&\cite{lusztig_coxeter}
\end{longtable}

To state the results for ($\ast$):
First, we list the cases where there are unique choices of $[\bbb{L}, \lambda]$ and $[\bbb{M}, \mu]$ with $|\Uch(\bbb{G})_{\bbb{L}, \lambda} \cap \Uch(\bbb{G})_{\bbb{M}, \mu}|$ nonempty:
\begin{longtable}{lll}
$(e, m)$
	&$|\Uch(\bbb{G})_{\bbb{L}, \lambda} \cap \Uch(\bbb{G})_{\bbb{M}, \mu}|$
	&($\ast$)\\
\hline
$(1, 3)$
	&$18$
	&$9 + 9$\\
$(1, 8)$
	&$5$
	&$5$\\
$(3, 6)$
	&$12$
	&$6 + 6$\\
$(3, 8)$
	&$3$
	&$3$\\
$(3, 12), (4, 8)$
	&$6$
	&$3 + 3$\\
$(4, 12)$
	&$9$
	&$3 + 3 + 3$\\
$(8, 12)$
	&$4$
	&$2 + 2$
\end{longtable}










In the remaining cases, it turns out that the partitions ($\ast$) are all trivial.
Below, we list the sizes $|\Uch(\bbb{G})_{\bbb{L}, \lambda} \cap \Uch(\bbb{G})_{\bbb{M}, \mu}|$ as entries of a table, where the rows are the $[\bbb{L}, \lambda]$ and the columns are the $[\bbb{M}, \mu]$.

\begin{longtable}{p{3em}l}
$(e, m)$
	&$|\Uch(\bbb{G})_{\bbb{L}, \lambda} \cap \Uch(\bbb{G})_{\bbb{M}, \mu}|$\\[0.5ex]
\hline\\[-2.5ex]
$(1, 2)$
	&
{\small$\begin{array}{p{4.5em}ll}
	&(\Phi_2^4, 1)
	&(\Phi_2^2B_2, -)\\
$(\Phi_1^4, 1)$
	&18
	&4\\
$(\Phi_1^2 B_2, -)$
	&4
	&1
\end{array}$}\\
\mbox{}\\[-2.5ex]
$(1, 4)$
	&
{\small$\begin{array}{p{4.5em}lll}
	&(\Phi_2^4, 1)
	&(\Phi_4B_2, \phi_{1^2, -})
	&(\Phi_4B_2, \phi_{-, 2})\\
$(\Phi_1^4, 1)$
	&9
	&3
	&3\\
$(\Phi_1^2 B_2, -)$
	&3
	&1
	&1
\end{array}$}\\
\mbox{}\\[-2.5ex]
$(1, 6)$
	&
{\small$\begin{array}{p{4.5em}l}
	&(\Phi_6^2, 1)\\
$(\Phi_1^4, 1)$
	&13\\
$(\Phi_1^2 B_2, -)$
	&4
\end{array}$}\\
\mbox{}\\[-2.5ex]
$(1, 12)$
	&
{\small$\begin{array}{p{4.5em}l}
	&(\Phi_{12}, 1)\\
$(\Phi_1^4, 1)$
	&5\\
$(\Phi_1^2 B_2, -)$
	&3
\end{array}$}\\
\mbox{}\\[-2.5ex]
$(3, 4)$
	&
{\small$\begin{array}{p{4.5em}lll}
	&(\Phi_2^4, 1)
	&(\Phi_4B_2, \phi_{1^2, -})
	&(\Phi_4B_2, \phi_{-, 2})\\
$(\Phi_3^2, 1)$
	&6
	&3
	&3
\end{array}$}
\end{longtable}

\subsection{The Other Exceptional Types}

\subsubsection{${}^3D_4$}

As in type $G_2$, the only relevant Levis are maximal tori.
\begin{longtable}{p{1em}p{2.5em}p{2.5em}p{6em}p{4em}}
$e$&$\bbb{L}$&$W^\bbb{G}_{\bbb{L}, \lambda}$&$\cal{S}^\bbb{G}_{\bbb{L}, \lambda}(u_{\cal{C}, j})$&\textsf{reference}\\
\hline
$1$&$\Phi_1^2\Phi_3$&$W_{G_2}$&$1,-x;1,-x^3$&\cite{lusztig_78}\\
$3$&$\Phi_3^2$&$G_4$&$1,x,x^2$&BM, \S5.6\\
$12$&$\Phi_{12}$&$\bb{Z}_4$&$1,\pm x^3,x^6$&\cite{lusztig_coxeter}
\end{longtable}

\subsubsection{$E_6$}\label{subsubsec:e6}

The notation is similar to that of \S\ref{subsec:f4}.
The notations $\phi_{2^2}, \phi_2, \phi_{1^2}$ are from~\cite[\S{13.8}]{carter_93}, while the notation \Drep{} is from~\cite[\S{13.9}]{carter_93}.
{\small\begin{longtable}{p{1em}p{4.25em}p{3.25em}p{3em}p{16em}p{4.25em}}
$e$&$\bbb{L}$&$\lambda$&$W^\bbb{G}_{\bbb{L}, \lambda}$&$\cal{S}^\bbb{G}_{\bbb{L}, \lambda}(u_{\cal{C}, j})$&\textsf{reference}\\
\hline
$1$&$\Phi_1^6$&&$W_{E_6}$&$1,-x$&Ex.\ \ref{ex:principal-hecke}\\
$1$&$D_4$&&$W_{A_2}$&$1,-x^4$&\cite{lusztig_78}\\
$2$&$\Phi_1^2\Phi_2^4$&&$W_{F_4}$&$1,x;1,-x^2$&\cite{lusztig_78}${}_\ur{E}$\\
$2$&$\Phi_2.A_5$&&$ \bb{Z}_2$&$1,x^9$&\cite{lusztig_78}${}_\ur{E}$\\
$3$&$\Phi_3^3$&&$G_{25}$&$1,x,x^2$&BM, 5.16\\
$3$&$\Phi_3.{}^{3}D_4$&\Drep{}&$ \bb{Z}_3$&$1,x^4,x^8$&BM, 8.1\\
$4$&$\Phi_4^2\Phi_1^2$&&$G_8$&$1,-x,x^2,-x^3$&BM, 5.12\\
$4$&$\Phi_1\Phi_4.{}^2A_3$&$\phi_{2^2}$&$ \bb{Z}_4$&$1,x^3,x^6,x^9$&BM, 8.1\\
$5$&$\Phi_1\Phi_5.A_1$&$\phi_2$&$ \bb{Z}_5$&$1,x^3,x^4,$ $x^6,x^{12}$&BM, 8.1\\
$5$&$\Phi_1\Phi_5.A_1$&$\phi_{1^2}$&$ \bb{Z}_5$&$1,x^6,x^8,$ $x^9,x^{12}$&BM, 8.1\\
$6$&$\Phi_6^2\Phi_3$&&$G_{5}$&$1,-x,x^2;$ $1,x^2,x^4$&BM, 5.9\\
$8$&$\Phi_1\Phi_2\Phi_8$&&$ \bb{Z}_8$&$1,\pm x^3,x^4,x^5,\pm x^6,x^9$&BM, 8.1\\
$9$&$\Phi_9$&&$ \bb{Z}_{9}$&$1,x^2,x^3,x^4,\zeta_3x^4,\zeta_3^2x^4,x^5,x^6,x^8$&BM, 8.1\\
$12$&$\Phi_3\Phi_{12}$&&$ \bb{Z}_{12}$&$1, x,\pm x^2,\pm x^3,\pm x^4,x^5,x^6,\zeta_3 x^3,\zeta_3^2 x^3$&\cite{lusztig_coxeter}
\end{longtable}}

\subsubsection{${}^2E_6$}

By Ennola duality, this table can be derived from the table for $E_6$ by substituting $-x$ for $x$ everywhere.


\subsubsection{$E_7$}\label{subsubsec:e7}

The notation is similar to that of \S\ref{subsec:f4} and \S\ref{subsubsec:e6}. 
{\small\begin{longtable}{p{1em}p{4.75em}p{2.5em}p{2.75em}p{17.25em}p{4.25em}}
$e$&$\bbb{L}$&$\lambda$&$W^\bbb{G}_{\bbb{L}, \lambda}$&$\cal{S}^\bbb{G}_{\bbb{L}, \lambda}(u_{\cal{C}, j})$&\textsf{reference}\\
\hline
$1$&$\Phi_1^7$&&$W_{E_7}$&$1,-x$&Ex.\ \ref{ex:principal-hecke}\\
$1$&$D_4$&&$W_{BC_3}$&$1,-x^4;1,-x$&\cite{lusztig_78}\\
$1$&$E_6$&&$\bb{Z}_2$&$1,-x^9$&\cite{lusztig_78}\\
$4$&$\Phi_4^2.A_1^3$&$\phi_2^3$&$G_{8}$&$1,\pm x,-x^4$&BM, 5.12\\
$4$&$\Phi_4^2.A_1^3$&$\phi_{1^2}^3$&$G_{8}$&$1,\pm x^3,-x^4$&BM, 5.12\\
$4$&$\Phi_4^2.A_1^3$&$\phi_2^2\phi_{1^2}$&$G_{4,1,2}$&$1, \pm x, -x^4; 1, x^6$&\\
$4$&$\Phi_4^2.A_1^3$&$\phi_2\phi_{1^2}^2$&$G_{4,1,2}$&$1, \pm x^3, -x^4; 1, x^6$&\\
$6$&$\Phi_6^3\Phi_2$&&$G_{26}$&$1,-x,x^2;1,x^3$&BM, 180\\
$6$&$\Phi_2\Phi_6.{}^3D_4$&$\phi_{2^11^1}$&$\bb{Z}_6$&$1,x,x^4,x^5,x^8$, $x^9$&BM, 8.1\\
$6$&$\Phi_6.{}^2A_5$&$\phi_{4.2}$&$\bb{Z}_6$&$1,\pm x,x^3,x^5,x^8$&BM, 8.1\\
$6$&$\Phi_6.{}^2A_5$&$\phi_{2^21^2}$&$\bb{Z}_6$&$1,x^3,x^5,\pm x^7, x^8$&BM, 8.1\\
$8$&$\Phi_8{}^2.D_4$&$\phi_2^2$&$\bb{Z}_{8}$&$1,x^2,\pm x^3,\pm x^5, x^6,x^{12}$&BM, 8.1\\
$8$&$\Phi_8{}^2.D_4$&$\phi_2\phi_{1^2}$&$\bb{Z}_{8}$&$1,\pm x,x^4, x^6, \pm x^7,x^{10}$&BM, 8.1\\
$8$&$\Phi_8{}^2.D_4$&$\phi_{1^2}^2$&$\bb{Z}_{8}$&$1,x^6,\pm x^7,\pm x^9, x^{10},x^{12}$&BM, 8.1\\
$8$&$\Phi_8{}^2.D_4$&$\phi_{1^2}\phi_2$&$\bb{Z}_{8}$&$1,\pm x^3,x^4, x^6, \pm x^9,x^{10}$&BM, 8.1\\
$10$&$\Phi_1\Phi_5.{}^2A_2$&$\phi_3$&$\bb{Z}_{10}$&$1,-1,x,\pm x^2, \pm x^3,x^4,x^6,x^9$&BM, 8.1\\
$10$&$\Phi_1\Phi_5.{}^2A_2$&$\phi_{2^11^1}$&$\bb{Z}_{10}$&$1,x, x^3,x^4, \pm\zeta_4 x^{\frac{9}{2}},x^5,x^6, x^8,x^9$&BM, 8.1\\
$10$&$\Phi_1\Phi_5.{}^2A_2$&$\phi_{1^3}$&$\bb{Z}_{10}$&$1,x^3,x^5,\pm x^6, \pm x^7,x^8,\pm x^9$&BM, 8.1\\
$12$&$\Phi_{12}.A_1(q^3)$&$\phi_2$&$\bb{Z}_{12}$&$1,\pm x,x^2,\zeta_3x^2, \zeta_3^2x^2,\pm x^3,x^4, \pm x^5,x^8$&BM, 8.1\\
$12$&$\Phi_{12}.A_1(q^3)$&$\phi_{1^2}$&$\bb{Z}_{12}$&$1,\pm x^3,x^4,\pm x^5, x^6,\zeta_3x^6, \zeta_3^2x^6, \pm x^7,x^8$&BM, 8.1\\
$14$&$\Phi_2\Phi_{14}$&&$\bb{Z}_{14}$&$1,x^2,\pm x^3,\pm x^4, \pm\zeta_4x^{\frac{9}{2}}, \pm x^5, \pm x^6, x^7, x^9$&BM, 8.1\\
$18$&$\Phi_2\Phi_{18}$&&$\bb{Z}_{18}$&$1, x, \pm x^2, \pm x^3, \zeta_3 x^3, \zeta_3^2 x^3, \pm\zeta_4x^{\frac{7}{2}}, \pm x^4$,
$\zeta_3 x^4, \zeta_3^2 x^4,\pm x^5, x^6, x^7$&\cite{lusztig_coxeter}
\end{longtable}}

\subsubsection{$E_8$}

The notation is similar to that of \Cref{subsubsec:e7}.
{\small\begin{longtable}{p{1em}p{4.25em}p{3.25em}p{2.75em}p{17em}p{4.25em}}
$e$&$\bbb{L}$&$\lambda$&$W^\bbb{G}_{\bbb{L}, \lambda}$&$\cal{S}^\bbb{G}_{\bbb{L}, \lambda}(u_{\cal{C}, j})$&\textsf{reference}\\
\hline
$1$&$\Phi_1^8$&&$W_{E_8}$&$1,-x$&Ex.\ \ref{ex:principal-hecke}\\
$1$&$\Phi_1^4.D_4$&&$W_{F_4}$&$1,-x^4;1,-x$&\cite{lusztig_78}\\
$1$&$\Phi_1^2.E_6$&&$W_{G_2}$&$1,-x^9;1,-x$&\cite{lusztig_78}\\
$1$&$\Phi_1.E_7$&&$\bb{Z}_2$&$1,-x^{15}$&\cite{lusztig_78}\\
$4$&$\Phi_4^4$&&$G_{31}$&$1,x^2$&BM, 180\\
$4$&$\Phi_4^2.D_4$&$\phi_{3,1}$&$G_8$&$1,-1,-x,x^5$&BM, 5.12\\
$4$&$\Phi_4^2.D_4$&$\phi_{123,013}$&$G_8$&$1,-x^4,\pm x^5$&BM, 5.12\\
$4$&$\Phi_4^2.D_4$&$\phi_{013,2}$&$G_8$&$1,-x,-x^4,x^5$&BM, 5.12\\
	$4$&$\Phi_4^2.D_4$&$\phi_{012,3}$&$G_8$&$1,x,-x^4,-x^5$&BM, 5.12\\
$6$&$\Phi_6^4$&&$G_{32}$&$1,-x,x^2$&BM, 180\\
$6$&$\Phi_6^2.{}^3D_4$&$\phi_{2,1}$&$G_5$&$1,-x,x^2;1,x^4,x^8$&BM, 5.9\\
$6$&$\Phi_6.{}^2E_6$&$\phi_{9,6}'$&$\bb{Z}_6$&$1,-1,x^2,x^5,$ $x^7,x^{10}$&BM, 8.1\\
$6$&$\Phi_6.{}^2E_6$&$\phi_{9,6}''$&$\bb{Z}_{6}$&$1,x^3,x^5,x^8,\pm x^{10}$&BM, 8.1\\
$6$&$\Phi_6.{}^2E_6$&$\phi_{6,6}''$&$\bb{Z}_{6}$&$1,x,\pm x^5,x^9,x^{10}$&BM, 8.1\\
$8$&$\Phi_8^2$&$1$&$G_{9}$&$1,x^2,x^4,x^6;1,x^4$&BM/M\\
$8$&$\Phi_8.{}^2D_4$&$\phi_{13,-}$&$\bb{Z}_8$&$\pm1,x,x^3,x^5,\pm x^6,x^{15}$&BM, 8.1\\
$8$&$\Phi_8.{}^2D_4$&$\phi_{0123,13}$&$\bb{Z}_8$&$1,\pm x^9,x^{10},x^{12},x^{14},\pm x^{15}$&BM, 8.1\\
$8$&$\Phi_8.{}^2D_4$&$\phi_{023,1}$&$\bb{Z}_8$&$1,x^3,x^5,-x^6,x^9,-x^{10},x^{12},x^{15}$&BM, 8.1
	\\$8$&$\Phi_8.{}^2D_4$&$\phi_{012,3}$&$\bb{Z}_8$&$1,-x^3,-x^5,-x^6,-x^9,-x^{10},x^{12},-x^{15}$&BM, 8.1\\
$8$&$\Phi_8.{}^2D_4$&$\phi_{123,0}$&$\bb{Z}_8$&$1,x^3,-x^5,x^6,-x^9,x^{10},x^{12},x^{15}$&BM, 8.1
	\\$8$&$\Phi_8.{}^2D_4$&$\phi_{013,2}$&$\bb{Z}_8$&$1,-x^3,x^5,x^6,x^9,x^{10},x^{12},-x^{15}$ &BM, 8.1\\
$10$&$\Phi_{10}^2$&$1$&$G_{16}$&$1,-x,x^2,-x^3,x^4$&BM/M\\
$10$&$\Phi_{10}.{}^2A_4$&$\phi_{3^12^1}$&$\bb{Z}_{10}$&$\pm 1,\pm x^3,x^4,\pm x^6,x^7,x^9,x^{12}$&BM, 8.1\\
$10$&$\Phi_{10}.{}^2A_4$&$\phi_{2^21^1}$&$\bb{Z}_{10}$&$1, x^3,x^5,\pm x^6,x^8,\pm x^9,\pm x^{12}$&BM, 8.1\\
$12$&$\Phi_{12}^2$&$1$&$G_{10}$&$1,x^3,-x^3,x^6;1,-x^2,x^4$&BM/M\\
$12$&$\Phi_{12}.{}^3D_4$&$\phi_{1,3}'$&$\bb{Z}_{12}$&$1,\zeta_3,\zeta_3^2,\pm x,x^2,\pm x^3, \pm x^5,x^6,x^{10}$&BM, 8.1\\
$12$&$\Phi_{12}.{}^3D_4$&$\phi_{1,3}''$&$\bb{Z}_{12}$&$1, x^4, \pm x^5,\pm x^7,x^8,\pm x^9,x^{10},\zeta_3x^{10}$, $\zeta_3^2x^{10}$&BM, 8.1\\
$12$&$\Phi_{12}.{}^3D_4$&$\phi_{2,2}$&$\bb{Z}_{12}$&$1,-x,x^2,x^4,\pm x^5,\zeta_3x^{5}, \zeta_3^2x^{5},x^6,x^8,$ $-x^9,x^{10}$&BM, 8.1
	\\$12$&$\Phi_{12}.{}^3D_4$&\Drep{}&$\bb{Z}_{12}$&$1,x,x^2,x^4,\pm x^5,-\zeta_3x^{5}, -\zeta_3^2x^{5},x^6,x^8$, $x^9, x^{10}$&BM, 8.1\\
$14$&$\Phi_2\Phi_{14}.A_1$&$\phi_2$&$\bb{Z}_{14}$&$1,\pm x^3,\pm x^5,\pm x^6,x^7,\pm\zeta_4x^{\frac{15}{2}},x^8,\pm x^9$, $x^{15}$&BM, 8.1\\
$14$&$\Phi_2\Phi_{14}.A_1$&$\phi_{1^2}$&$\bb{Z}_{14}$&$1,\pm x^6,x^7,\pm\zeta_4x^{\frac{15}{2}},x^8,\pm x^9,\pm x^{10},\pm x^{12}$, $x^{15}$&BM, 8.1\\
$18$&$\Phi_{18}.{}^2A_2$&$\phi_3$&$\bb{Z}_{18}$&$1,-1,x,x^2,\zeta_3 x^2,\zeta_3^2 x^2,\pm\zeta_4x^{5/2},\pm x^3$, $\pm x^4,x^5,\zeta_3 x^5,\zeta_3^2 x^5,\pm x^6,x^{10}$&BM, 8.1\\
$18$&$\Phi_{18}.{}^2A_2$&$\phi_{2^11^1}$&$\bb{Z}_{18}$&$1,x,\pm x^3,\pm x^4,\pm x^5, \pm \zeta_3 x^5,\pm\zeta_3^2 x^5,$ $\pm x^6,\pm x^7, x^9,x^{10}$&BM, 8.1\\
$18$&$\Phi_{18}.{}^2A_2$&$\phi_{1^3}$&$\bb{Z}_{18}$&$1,\pm x^4,x^5,\zeta_3 x^5,\zeta_3^2 x^5,\pm x^6,\pm x^7$, $\pm\zeta_4x^{\frac{15}{2}}$, $x^8,\zeta_3 x^8,\zeta_3^2 x^8,x^9,\pm x^{10}$&BM, 8.1\\
$20$&$\Phi_{20}$&&$\bb{Z}_{20}$&$1,\pm x^3,x^4,\pm x^5,\pm x^6,\pm x^7,x^8,\pm x^9$, $\pm\zeta_4x^{6},\zeta_5 x^6,\zeta_5^2 x^6,\zeta_5^3 x^6,\zeta_5^4 x^6,x^{12}$&BM, 8.1\\
$24$&$\Phi_{24}$&&$\bb{Z}_{24}$&$1,x^2,\pm x^3,\pm x^4,\zeta_3 x^4, \zeta_3^2 x^4,\pm x^5,\pm\zeta_4x^{5},$  $\pm \zeta_3 x^5,\pm\zeta_3^2 x^5, \pm x^6,\zeta_3 x^6,\zeta_3^2 x^6,\pm x^7$, $x^8, x^{10}$ &BM, 8.1\\
$30$&$\Phi_{30}$&&$\bb{Z}_{30}$&$1,x,\pm x^2,\pm x^3,\zeta_3 x^3,\zeta_3^2 x^3,\pm\zeta_4x^{\frac{7}{2}},\pm x^4,$ $\pm\zeta_3 x^4,\pm\zeta_3^2 x^4,\zeta_5 x^4,\zeta_5^2 x^4,\zeta_5^3 x^4, \zeta_5^4 x^4$, $\pm\zeta_4x^{\frac{9}{2}},\pm x^5,\zeta_3 x^5, \zeta_3^2 x^5, \pm x^6, x^7,x^8$ &\cite{lusztig_coxeter}
\end{longtable}}


\bibliographystyle{alphaurl}
\bibliography{blocks}

\end{document}